\documentclass{article}
\RequirePackage{a4wide}
\RequirePackage{graphicx}
\RequirePackage{xcolor}
\RequirePackage{amscd}
\RequirePackage{framed}
\RequirePackage{bbm}
\RequirePackage{amsmath,amsthm,amssymb}
\RequirePackage{fancyhdr,a4wide}
\RequirePackage{comment}
\RequirePackage[colorlinks=true,citecolor=blue,urlcolor=blue]{hyperref}
\newcommand{\mc}{\mathcal}
\newcommand{\mbb}{\mathbb}
\newcommand{\mr}{\mathrm}
\newcommand{\argmin}{\mathop{\rm argmin}\limits}

\newcommand{\vecg}{\mathbf{g}}

\newcommand{\vecu}{\mathbf{u}}

\newcommand{\vecv}{\mathbf{v}}

\newcommand{\vecx}{\mathbf{x}}
\newcommand{\vecX}{\mathbf{X}}
\newcommand{\vecy}{\mathbf{y}}
\newcommand{\vecz}{\mathbf{z}}

\newcommand{\vecbeta}{\boldsymbol \beta}
\newcommand{\vectheta}{\boldsymbol \theta}

\newcommand{\veclambda}{\boldsymbol \lambda}
\newcommand{\vecgamma}{\boldsymbol \gamma}
\newcommand{\veciota}{\boldsymbol \iota}

\newcommand{\vecmu}{\boldsymbol \mu}

\RequirePackage{bm}

\RequirePackage{algorithm}
\RequirePackage{algorithmic}
\renewcommand{\algorithmicrequire}{\textbf{Input:}}
\renewcommand{\algorithmicensure}{\textbf{Output:}}

\numberwithin{equation}{section}
\newtheorem{theorem}{Theorem}[section]
\newtheorem{proposition}{Proposition}[section]
\newtheorem{remark}{Remark}[section]
\newtheorem{lemma}{Lemma}[section]
\newtheorem{corollary}{Corollary}[section]
\newtheorem{definition}{Definition}[section]
\newtheorem{assumption}{Assumption}[section]

\begin{document}

\title{An Easily Tunable Approach to Robust and Sparse High-Dimensional Linear Regression}
\author{Takeyuki Sasai
\thanks{Department of Statistical Science, The Graduate University for Advanced Studies, SOKENDAI, Tokyo, Japan. Email: sasai@ism.ac.jp}
\and Hironori Fujisawa
\thanks{The Institute of Statistical Mathematics, Tokyo, Japan, \\
Department of Statistical Science, The Graduate University for Advanced Studies, SOKENDAI, Tokyo, Japan, and\\ 
Center for Advanced Integrated Intelligence Research, RIKEN, Tokyo, Japan.  Email:fujisawa@ism.ac.jp}
}
\maketitle
\begin{abstract}
Sparse linear regression methods such as Lasso require a tuning parameter that depends on the noise variance, which is typically unknown and difficult to estimate in practice. 
In the presence of heavy-tailed noise or adversarial outliers, this problem becomes more challenging. 
In this paper, we propose an estimator for robust and sparse linear regression that eliminates the need for explicit prior knowledge of the noise scale. 
Our method builds on the Huber loss and incorporates an iterative scheme that alternates between coefficient estimation and adaptive noise calibration via median-of-means. 
The approach is theoretically grounded and achieves sharp non-asymptotic error bounds under both sub-Gaussian and heavy-tailed noise assumptions.
Moreover, the proposed method accommodates arbitrary outlier contamination in the response without requiring prior knowledge of the number of outliers or the sparsity level. 
While previous robust estimators avoid tuning parameters related to the noise scale or sparsity, our procedure achieves comparable error bounds when the number of outliers is unknown, and improved bounds when it is known. 
In particular, the improved bounds match the known minimax lower bounds up to constant factors.
\end{abstract}

\section{Introduction}
The Lasso method has proven effective for sparse linear regression problems and has been widely adopted across various fields since its introduction \cite{Tib1996Regression}. Especially in the era of big data, where handling high-dimensional datasets is critical, Lasso continues to play a pivotal role in data analysis. Consider the sparse linear regression model defined as
\begin{align}
	\label{model:normal}
	y_i = \vecx_i^\top\vecbeta^*+\mu^*+\xi_i,\quad  i=1,\cdots,n,
\end{align}
where  $\{y_i\}_{i=1}^n$ denotes the output sequence, $\{\vecx_i\}_{i=1}^n$  is an i.i.d. sequence of $d$-dimensional covariates, $\vecbeta^* \in \mbb{R}^d$ is the true coefficient vector,  $\mu^* \in \mbb{R}$ is the intercept, and $\{\xi_i\}_{i=1}^n$ is an  i.i.d. sequence of random noise. To model sparsity, we assume  $\|\vecbeta^*\|_0=s (\leq d)$, where  $\|\vecv\|_0$ denotes the number of nonzero entries in the vector $\vecv$. Then, the Lasso estimator is given by
\begin{align}
	\label{model:normal2}
	(\hat{\vecbeta}^{\mr{Lasso}},\hat{\mu}^{\mr{Lasso}})= \argmin_{\vecbeta \in \mbb{R}^d, \mu \in \mbb{R}}\frac{1}{2n}\sum_{i=1}^n (y_i-\vecx_i^\top \vecbeta-\mu)^2 +\lambda_{\mr{Lasso}}\|\vecbeta\|_1,
\end{align} 
where for a vector $\vecv=(v_1,\cdots,v_d) \in \mbb{R}^d$, $\|\vecv\|_1=\sum_{i=1}^d |v_i|$, and $\lambda_{\mr{Lasso}}$ is a tuning parameter.

For simplicity, assume that $\{\vecx_i\}_{i=1}^n$ are  i.i.d. standard Gaussian vectors, and the noise $\{\xi_i\}_{i=1}^n$ consists of i.i.d. Gaussian random variables with standard deviation $\sqrt{\mbb{E}\xi^2_i}=\sigma^*$. The optimal tuning parameter is 
\begin{align}
	\lambda_{\mr{Lasso}} = \mr{C}\sigma^* \sqrt{\frac{\log (d/s)}{n}},
\end{align}
where $\mr{C}$ is a sufficiently large numerical constant. Define  $\|\vecv\|_2 = \sqrt{\sum_{i=1}^d v_i^2 }$ for a vector $\vecv =(v_1,\cdots,v_d)$, and $\lesssim$ denotes an inequality up to a numerical factor. Then, from \cite{BelLecTsy2018Slope}, we have 
\begin{align}
	\label{ine:besttune}
	\mbb{P}\left(\|\vecbeta^*-\hat{\vecbeta}^{\mr{Lasso}}\|_2+|\mu^*-\hat{\mu}^{\mr{Lasso}}| \lesssim \sqrt{s} \lambda_{\mr{Lasso}}+\sigma^*\sqrt{\frac{\log(1/\delta)}{n}}\right)\geq 1-\delta,
\end{align}
provided that $\sqrt{s\frac{\log(d/s)}{n}}+\sqrt{\frac{\log(1/\delta)}{n}}$ is sufficiently small.
Notably, employing the SLOPE method \cite{SuCan2016Slope} achieves a similar error bound (up to a constant factor) without requiring knowledge of $s$ for the tuning parameter \cite{BelLecTsy2018Slope}; see Remark \ref{r:s} for further discussion.

A key challenge in applying Lasso is choosing $\sigma^*$ in the tuning parameter $\lambda_{\text{Lasso}}$. Since the noise variables ${\xi_i}_{i=1}^n$ are unobservable, estimating $\sigma^*$ accurately is difficult. 
Analysts often substitute an upper bound of $\sigma^*$; however, using a value that is too small can cause the estimation to fail, whereas using an overly large value leads to loose error bounds.

To mitigate this issue, methods such as square-root Lasso \cite{BelCheWan2011Square} and scaled Lasso \cite{SunZha2012Scaled} have been proposed. 
These approaches modify the loss function to eliminate the dependence on noise variance. 
Additionally, cross-validation techniques have been suggested \cite{CheLiaChe2021Cross} to select an appropriate tuning parameter from a candidate set.

In this paper, we propose an alternative method tailored for robust settings, including cases with outliers or heavy-tailed noise, which are scenarios that have attracted considerable attention in recent studies \cite{LugMen2019Mean,DiaKan2023Algorithmic}.
Our method allows for reliable analysis without requiring precise knowledge of the standard deviation of the noise.

We now consider a robust version of the model defined in \eqref{model:normal}:
\begin{align}
	\label{model:normalad}
	y_i = \vecx_i^\top\vecbeta^*+\mu^*+\sqrt{n}\theta_i+\xi_i,\quad  i=1,\cdots,n,
\end{align}
where  $\vectheta = (\theta_1,\cdots,\theta_n)$ represents outliers, with  $\theta_i = 0$ for $i \in \{1,\cdots,n\}\setminus \mc{O}$ and $|\mc{O}|=o$, and the noise $\{\xi_i\}_{i=1}^n$ follows a heavy-tailed distribution. 
We assume that the set of outlier indices $\mc{O}$ is chosen arbitrarily by an adversary. Furthermore, for each $i\in \mc{O}$, the corresponding outlier value $\theta_i$ can also be freely selected by the adversary.

This model has been studied in a number of recent works. Notably, \cite{MinNdaWan2024Robust} combines square-root Lasso with the Huber loss to construct an estimator that does not require prior knowledge of the noise variance while achieving sharp error bounds.
Specifically, when $\{\xi_i\}_{i=1}^n$ are sub-Gaussian, their estimator $(\tilde{\vecbeta}^{\text{sg}}, \tilde{\mu}^{\text{sg}})$ satisfies
\begin{align}
	\label{ine:intro-pr1}
	\mbb{P}\left(\|\vecbeta^*-\tilde{\vecbeta}^{\mr{sg}}\|_2 +|\mu^*-\tilde{\mu}^{\mr{sg}} |\lesssim \sigma^* \left(\sqrt{\frac{s\log(d/s)}{n}}+\sqrt{\frac{\log(1/\delta)}{n}}+\frac{o}{n}\log \frac{n}{o}\right)\right)\geq 1-\delta,
\end{align}
provided that $\sqrt{s\frac{\log(d/s)}{n}}+\sqrt{\frac{\log(1/\delta)}{n}}$ and $\frac{o}{n}$ are sufficiently small.
On the other hand, if $\{\xi_i\}_{i=1}^n$ are drawn from a heavy-tailed distribution such that for $\text{M} \geq 2$, we have $\left( \mathbb{E}|\xi_i|^\text{M} \right)^{1/\text{M}} \leq c_{\mr{M}} \sigma^*$ for a constant $c_{\mr{M}}$, then with probability at least $1 - \delta$, the estimator $(\tilde{\vecbeta}^{\text{h}}, \tilde{\mu}^{\text{h}})$ satisfies
\begin{align}
	\label{ine:intro-pr2}
	\|\vecbeta^*-\tilde{\vecbeta}^{\mr{h}}\|_2+|\mu^*-\tilde{\mu}^{\mr{h}} | \lesssim \sigma^*\left(\sqrt{\frac{s\log(d/s)}{n}}+\left(\frac{\log(1/\delta)}{n}\right)^{\min\{\frac{1}{2},1-\frac{2}{\mr{M}}\}}+\left(\frac{o}{n}\right)^{1-\frac{2}{\mr{M}}}\right),
\end{align}
provided that $\sqrt{s\frac{\log(d/s)}{n}}+\sqrt{\frac{\log(1/\delta)}{n}}$ and $\frac{o}{n}$ are sufficiently small.
While our objective is aligned with that of \cite{MinNdaWan2024Robust}, our approach differs in method. Although our procedure is slightly more complex, it yields sharper error bounds under certain settings.

Section \ref{sec:RWOA} presents related work and a high-level overview of our approach. Section \ref{sec:R} provides a precise statement of our method and main results. Details of the procedure are explained in Sections \ref{sec:AESC} and \ref{sec:EI}.

\section{Related Work and Our Approach}
\label{sec:RWOA}
To simplify the exposition in this section, we assume that $\{\vecx_i\}_{i=1}^n$ is an i.i.d. sequence of Gaussian random vectors such that $\mathbb{E}\vecx_i = 0$ and $\mathbb{E}\vecx_i\vecx_i^\top = I$ for each $i = 1, \ldots, n$.

In practice, cross-validation is commonly used to select the tuning parameter for Lasso.
However, the theoretical understanding of non-asymptotic properties for estimators based on cross-validated parameters remains limited. 
A notable exception is \cite{CheLiaChe2021Cross}, which shows that under suitable conditions—such as the inclusion of near-optimal candidates in the search grid—the resulting estimators can achieve sharp error bounds, even without access to the true value of $\sigma^*$.

Square-root Lasso \cite{BelCheWan2011Square} and scaled Lasso \cite{SunZha2012Scaled} offer alternative approaches. 
These methods construct estimators with theoretical guarantees comparable to those of standard Lasso (cf. \eqref{ine:besttune}) but without requiring prior knowledge of $\sigma^*$. The estimator is defined as
\begin{align}
	\label{ine:SQLasso}
	(\hat{\vecbeta}^{\mr{SLasso}},\hat{\mu}^{\mr{SLasso}}) = \argmin_{\vecbeta \in \mbb{R}^d,\mu \in \mbb{R}} \sqrt{\frac{1}{2n}\sum_{i=1}^n (y_i-\vecx_i^\top \vecbeta-\mu)^2 }+\lambda_{\mr{SLasso}}\|\vecbeta\|_1.
\end{align} 
This estimator can equivalently be written as (see Chapter 3 of \cite{Gee2016Estimation}):
\begin{align}
		\label{ine:SQLasso2}
	(\hat{\vecbeta}^{\mr{SLasso}},\hat{\sigma}^{\mr{SLasso}},\hat{\mu}^{\mr{SLasso}}) = \argmin_{\vecv \in \mbb{R}^d, \sigma \in \mbb{R},\mu \in \mbb{R}} \frac{1}{2n\sigma}\sum_{i=1}^n (y_i-\vecx_i^\top \vecbeta-\mu)^2+\sigma+\lambda_{\mr{SLasso}}\|\vecbeta\|_1.
\end{align} 
This formulation highlights that $\lambda_{\text{SLasso}}$ does not depend on prior knowledge of $\sigma^*$. For non-asymptotic analysis of this method, see \cite{Der2018Improved}.

Robust estimation has recently drawn considerable interest \cite{LugMen2019Mean,DiaKan2023Algorithmic}, particularly for models like \eqref{model:normalad}. Several studies—e.g., \cite{NguTra2012Robust,DalTho2019Outlier,Chi2020Erm,Tho2020Outlier,Tho2023Outlier,MinNdaWan2024Robust}—investigate such settings.

Before discussing a key insight from \cite{MinNdaWan2024Robust}, we briefly review the connection between $\ell_1$-penalization and the Huber loss, which is known to be robust to outliers and heavy-tailed noise \cite{NguTra2012Robust,DalTho2019Outlier,AlqCotLec2019Estimation,Chi2020Erm,Tho2020Outlier,Tho2023Outlier}.
The Huber loss function is defined as
\begin{align}
	H(t) = \begin{cases}
	|t| -1/2 & (|t| > 1) \\
	t^2/2  & (|t| \leq 1)
	\end{cases},
	\end{align}
	and let
	\begin{align}
		h(t) =	\frac{d}{dt} H(t) =   \begin{cases}
		t\quad &(|t|> 1)\\
		\mr{sgn}(t)\quad &(|t| \leq 1)
	\end{cases}.
\end{align}
Using this loss, we define the following $\ell_1$-penalized estimator:
\begin{align}
	\label{ine:l1-penalized-Huber-re}
	(\hat{\vecbeta}^{\ell_1\text{-pH}},\hat{\mu}^{\ell_1\text{-pH}})=\argmin_{\vecbeta  \in \mbb{R}^d,\mu \in \mbb{R}} \sum_{i=1}^n (\sigma^*\iota\sqrt{n})^2 H\left(\frac{y_i-\vecx_i^\top\vecbeta-\mu}{\sigma^*\iota\sqrt{n}}\right)+\sigma^*\lambda\|\vecbeta\|_1.
\end{align} 
As shown in \cite{SheOwe2011Outlier}, this estimator can be equivalently expressed as
\begin{align}
	\label{ine:l1-penalized-Huber-re2}
	(\hat{\vecbeta}^{\ell_1\text{-pH}},\hat{\mu}^{\ell_1\text{-pH}},\hat{\vectheta}^{\ell_1\text{-pH}})=\argmin_{\vecbeta  \in \mbb{R}^d,\mu \in \mbb{R},\vectheta \in \mbb{R}^n} \frac{1}{2n}\sum_{i=1}^n(y_i-\vecx_i^\top\vecbeta-\mu-\sqrt{n}\theta_i)^2+\sigma^*\iota \|\vectheta\|_1+\sigma^*\lambda\|\vecbeta\|_1.
\end{align}

Extending this idea, \cite{MinNdaWan2024Robust} combines square-root Lasso with the Huber loss and propose the following estimator:
\begin{align}
	\label{ine:l1-penalized-Huber-re3}
	(\hat{\vecbeta}^{\ell_1\text{-spH}},\hat{\mu}^{\ell_1\text{-spH}},\hat{\vectheta}^{\ell_1\text{-spH}})=\argmin_{\vecbeta  \in \mbb{R}^d,\mu \in \mbb{R},\vectheta \in \mbb{R}^n} \sqrt{\frac{1}{2n}\sum_{i=1}^n(y_i-\vecx_i^\top\vecbeta-\mu-\sqrt{n}\theta_i)^2}+\iota\|\vectheta\|_1+\lambda\|\vecbeta\|_1.
\end{align}

\begin{remark}
	\label{r:s}
	In fact, \cite{MinNdaWan2024Robust} analyzes a variant of the above formulation where the $\ell_1$ norms are replaced by the SLOPE norms \cite{SuCan2016Slope}. The SLOPE norm for a vector $\vecv =(v_1,\cdots,v_d)\in \mathbb{R}^d$ with a weight vector $\vecgamma = (\gamma_1,\cdots,\gamma_d)$ is defined as $\|\vecv\|_{\vecgamma} = \sum_{i=1}^d |v_i^\sharp| \gamma_i$, where $|v_i^\sharp|$ is the $i$-th largest absolute entry in $\vecv$.
	This modification enables estimation without prior knowledge of the sparsity level $s$, achieving error rates same to those of Lasso up to an absolute constant factor. Furthermore, using the SLOPE norm for both the regression and outlier components allows estimation without requiring knowledge of the number of outliers $o$.
\end{remark}

\begin{remark}
	\label{r:rel}
\cite{Tho2020Outlier,Tho2023Outlier} also propose a form of tuning-free estimation. Their estimators can be viewed as a non-root version of the formulation in \eqref{ine:l1-penalized-Huber-re3}, and therefore still require the noise standard deviation as a tuning parameter. On the other hand, their approach enjoys certain advantages. For example, it avoids the need for the small-ball condition on the noise distribution—required in both \cite{MinNdaWan2024Robust} and our analysis (see Section~\ref{sec:R1})—and allows for relaxed assumptions on the dependence structure between covariates and noise. Additionally, their analysis relies on multiplier process techniques and only requires $\mathbb{E}\xi_i \vecx_i = \mathbf{0}$, without assuming independence between $\vecx_i$ and $\xi_i$.
\end{remark}
We now present an overview of our proposed method. Detailed analysis is provided in Sections \ref{sec:R} through \ref{sec:EI}.
Before introducing the full procedure, we highlight a property of the estimator defined in \eqref{ine:l1-penalized-Huber-re}. 
When the intercept $\mu^* = 0$, we consider the following simplified estimator:
\begin{align}
	\label{ine:ourssec2-0}
	\hat{\vecbeta}^{\ell_1\text{-pH}}=\argmin_{\vecbeta  \in \mbb{R}^d} \sum_{i=1}^n (\varsigma^*\iota\sqrt{n})^2 H\left(\frac{y_i-\vecx_i^\top\vecbeta}{\varsigma^*\iota\sqrt{n}}\right)+\varsigma^*\|\vecbeta\|_{\veclambda},
\end{align} 
where $\varsigma^* = \mathbb{E}|\xi_i|$ is the absolute first moment of the noise. Under this setup, the estimator satisfies the following error bound:
\begin{align}
	\label{ine:ourssec2-1}
	\mbb{P}\left(\|\vecbeta^*-\hat{\vecbeta}^{\ell_1\text{-pH}}\|_2 \lesssim  \varsigma^* \left(\sqrt{\frac{s\log(ed/s)}{n}}+\sqrt{\frac{\log(1/\delta)}{n}}+\frac{o}{n}\sqrt{\log \frac{n}{o}}\right)\right)\geq 1-\delta,
\end{align} 
where $\iota$ is a sufficiently large constant. A similar result appears in \cite{Chi2020Erm}; see Remark \ref{r:co} for a detailed comparison.

Notably, this error bound depends on $\varsigma^*$ (the absolute moment) rather than $\sigma^*$ (the standard deviation), and it exhibits faster convergence than the bound in \eqref{ine:intro-pr1}—regardless of the tail behavior of the noise. This suggests that Huber loss-based methods offer stronger robustness, particularly when there is no intercept bias. For related insights, see \cite{SunZhoFan2020Adaptive, WanZheZhoZho2021New} for related insights.

Building on this observation, we propose the following iterative procedure:
\begin{enumerate}
  \item Initialize with a constant $\mathrm{C}_{\text{ini}} > \varsigma^*$.
  \item Normalize the dataset $\{y_i, \vecx_i\}_{i=1}^n$ to eliminate the intercept $\mu^*$.
  \item Estimate $\vecbeta^*$ from the normalized data using the estimator in \eqref{ine:ourssec2-0}, with the current estimate of $\varsigma^*$.
  \item Estimate $\varsigma^*$ from the normalized data using the current estimate of $\vecbeta^*$.
  \item Repeat 3. and 4. sufficiently.
  \item Using the final estimate of $\vecbeta^*$, estimate $\mu^*$ from the original (non-normalized) data.
\end{enumerate}
Several observations regarding the above procedure are as follows:
\begin{itemize}
	\item \textbf{Relation to scaled Lasso:} This procedure is inspired by $(3)$ in \cite{SunZha2012Scaled}, and can be regarded as a robust extension of their method to robust settings with random design.
  \item \textbf{Normalization:} Step 2 eliminates the intercept by differencing paired samples. This technique is also used in \cite{MenZhi2020Robust,WanZheZhoZho2021New,PenJogLoh2024Robust} ; see Algorithm~\ref{alg:N}  for more details.
 	\item \textbf{Estimating $\varsigma^*$ via median-of-means:} In Step 4, the absolute noise level $\varsigma^* = \mathbb{E}|\xi_i|$ is estimated using the median-of-means estimator \cite{AloMatSze1996Space}. We apply it to the sequence $\{ |\vecx_i^\top (\vecbeta^* - \hat{\vecbeta}_{\mathrm{cur}}) + \xi_i + \sqrt{n}\theta_i^*| \}_{i=1}^n$. If both $\|\vecbeta^* - \hat{\vecbeta}_{\mathrm{cur}}\|_2$ and $o/n$ are small, this quantity is close to $\mathbb{E}|\xi_i|$. While the estimator requires a hyperparameter depending on the moment condition of $\xi_i$, overestimating it does not degrade the final error bound (see Remark~\ref{r:mm}).
	 \item  \textbf{Convergence through iteration:} As the coefficient estimate improves over iterations, the estimate of $\varsigma^*$ decreases and approaches its true value. For convergence guarantees, it is sufficient that $\sqrt{\frac{s\log(ed/s)}{n}} + \sqrt{\frac{\log(1/\delta)}{n}} + \frac{o}{n} \sqrt{\log(n/o)}$ is small. This parallels the assumptions required for Lasso estimation in standard settings.
	 \item \textbf{Final coefficient accuracy:} After sufficient iterations, the estimator $\hat{\vecbeta}$ satisfies the same error bound as \eqref{ine:ourssec2-1}, independently of the initial value $\mathrm{C}_{\mathrm{ini}}$ (see Theorem~\ref{t:Huberabs}).
	 \item \textbf{Intercept estimation:} In Step 6, we solve the optimization
	 \begin{align}
		 \label{opt:mu}
		 (\hat{\mu}^{\ell_1\text{-slpH}}, \hat{\vectheta}^{\ell_1\text{-slpH}}) = \argmin_{\mu \in \mathbb{R},\, \vectheta \in \mathbb{R}^n} \left\{ \sqrt{ \frac{1}{2n} \sum_{i=1}^n (y_i - \vecx_i^\top \hat{\vecbeta}^{\mathrm{output}} - \mu - \sqrt{n}\theta_i)^2 } + \|\vectheta\|_{\veciota} \right\},
	 \end{align}
	 where $\hat{\vecbeta}^{\mathrm{output}}$ is the final estimator of the coefficient. Unlike the analysis of the SLOPE version of \eqref{ine:l1-penalized-Huber-re3}, this approach avoids estimating $\vecbeta^*$, simplifying theoretical analysis. 
	 As a result, the estimator of $\mu^*$ achieves an error bound matching or improving upon those in \cite{MinNdaWan2024Robust}, especially when an accurate estimate of $o$ is available (see Theorem~\ref{t:intercept} and Corollary~\ref{c:main}).
	\item \textbf{No sample splitting:} Unlike iterative procedures in \cite{BalWaiYun2017Statistical,PraSugBalRav2020Robust}, which often rely on sample splitting for analysis, our method avoids sample splitting by applying empirical process theory. This yields theoretical guarantees comparable to tuning-dependent methods, while preserving data efficiency.
\end{itemize}

\section{Method and Result}
\label{sec:R}
In Section \ref{sec:R1}, we introduce the assumptions on the covariates and noise, and in Section \ref{sec:R2}, we present our results.
\subsection{Method}
\label{sec:R1}
Throughout the remainder of this paper, we assume for notational simplicity that $n/o\geq e$, $\delta\leq 1/9$.
Before presenting our method and the main theoretical results, we introduce several definitions and assumptions used throughout the analysis.
\begin{definition}[$\psi_\alpha$-norm]
\label{d:orlicz}
	Let $f$ be a real-valued random variable. The $\psi_\alpha$-norm of $f$ is defined as
	\begin{align}
			\|f\|_{\psi_\alpha}:=	\inf\left\{ \eta>0\,:\, \mbb{E}\exp(|f/\eta|^\alpha)\leq 2\right\} < \infty.
	\end{align}	
\end{definition}
\begin{definition}[$L$-subGaussian random vector]
\label{d:l}
	A random vector $\vecx \in \mbb{R}^d$ with mean zero is said to be an $L$-subGaussian random vector if,  for any fixed $\vecv \in \mbb{R}^d$,
	\begin{align}
	\label{ine:d:l-01}
		\|\langle \vecx,\vecv\rangle\|_{\psi_2}\leq L\left(\mbb{E}|\langle \vecx,\vecv\rangle|^2\right)^\frac{1}{2} = L\|\vecv\|_\Sigma,
	\end{align}
	where $\|\cdot\|_{\psi_2}$ is $\psi_\alpha$-norm from Definition \ref{d:orlicz},  and $\|\vecv\|_\Sigma^2:= \vecv^\top \Sigma \vecv $.
	This condition is equivalent to several tail and moment bounds (see  (2.14) - (2.16) of \cite{Ver2018High}), such as:
	for any $\vecv \in \mbb{R}^d$, $t\geq 0$ and $p\in \mbb{N}$ such that $p\geq 2$,   redefining $L$ with a larger value if necessary,
	\begin{align}
	\label{ine:d:l-02}
		&\left(\mbb{E}|\vecv^\top \vecx|^p \right)^\frac{1}{p}\leq L \sqrt{p} \|\vecv\|_\Sigma,\, \mbb{E}\exp(\vecv^\top \vecx) \leq \exp(L^2 \|\vecv\|_\Sigma^2 ),\nonumber \\
		&\mbb{E}\exp\left(\frac{(\vecv^\top \vecx)^2}{L^2\|\vecv\|_\Sigma^2}\right) \leq 2,\, 	\mbb{P}\left(|\vecv^\top \vecx|>t\right) \leq 2 \exp \left(-\frac{t^2}{L^2 \|\vecv\|_\Sigma^2}\right).
	\end{align}
\end{definition}
Then, we introduce an assumption for covariates:
\begin{assumption}
	\label{a:cov}
	The covariates $\{\vecx_i\}_{i=1}^n$ is an i.i.d. sequence of random vectors drawn from an $L$-subGaussian distribution with $\mbb{E}\vecx_i=\bf{0}$ for $i=1,\cdots,n$.
\end{assumption}

Next, we introduce a condition on the covariance matrix of the covariates. This assumption is commonly imposed in the context of SLOPE estimation \cite{BelLecTsy2018Slope}. It enables accurate estimation even when the covariance matrix is singular, which frequently occurs in high-dimensional settings.
\begin{assumption}
	\label{a:cov-RE}
	The covariance matrix of the covariates $\Sigma$ satisfies 
	\begin{align}
		\label{a:RE-1}
		\|\vecv\|_\Sigma^2 \geq \kappa \|\vecv\|_2^2,
	\end{align}
	for any $\vecv \in \mbb{R}^d$ and for a constant $c_{\mr{RE}}>1$, such that
	\begin{align}
		\label{a:RE-2}
		\|\vecv\|_{\veclambda}\leq c_{\mr{RE}} \sqrt{\sum_{i=1}^s\lambda_i^2} \|\vecv\|_2^2.
	\end{align}
\end{assumption}

Note that if the minimum eigenvalue of $\Sigma$ exists, then $\kappa \geq \lambda_{\min}(\Sigma)$, and Assumption~\ref{a:cov-RE} can hold even if the minimum eigenvalue is zero.

Before presenting the main results, we introduce several assumptions regarding the noise distribution. We denote by $\mathbb{I}_E$ the indicator function of the event $E$.
\begin{assumption}
	\label{a:noise}
	The noise sequence $\{\xi_i\}_{i=1}^n$ is an i.i.d. sequence of random variables drawn from a distribution such that $\mbb{E}|\xi_i|= \varsigma^*$ and $\mbb{E}|\xi_i|^2= \sigma^2$ with $\mbb{E}\xi_i=0$ for $i=1,\cdots,n$.
	Additionally, assume that there exists  $\mr{c}_\xi(\geq 1)$ such that $\sigma \leq \mr{c}_\xi \varsigma^*$.
\end{assumption}
\begin{assumption}[heavy-tailed noise]
	\label{a:noise:EI}
	The noise $\{\xi_i\}_{i=1}^n$ satisfies Assumption \ref{a:noise}. Additionally, $\{\xi_i\}_{i=1}^n$ satisfies $\mbb{E}|\xi_i|^{\mr{M}} \leq (\sigma^* \times c_{\mr{M}})^{\mr{M}}:=\sigma_{\mr{M}}^{\mr{M}}$ for $i=1,\cdots,n$, where $\mr{M}>2$ and $c_{\mr{M}}$ is a numerical constant. Additionally, assume $\mbb{E}\xi_i^2\mathbb{I}_{|\xi_i|\leq \frac{1}{2}}\geq 1/4$, and $\max\{o/n,\log(1/\delta)/n\}\leq 1/1000$.
\end{assumption}
\begin{assumption}[subGaussian noise]
	\label{a:noise:EI2}
	The noise $\{\xi_i\}_{i=1}^n$ satisfies Assumption \ref{a:noise}. Additionally, $\{\xi_i\}_{i=1}^n$ satisfies $\mbb{P}(|\xi_i|\geq t)\leq 2\exp(-(t/(\sigma^* c_{\mr{M}}))^2)$ for $i=1,\cdots,n$, where $c_{\mr{M}}>1$ is a numerical constant. Additionally, assume $\mbb{E}\xi_i^2\mathbb{I}_{|\xi_i|\leq \frac{1}{2}}\geq 1/4$, and $\max\{o/n,\log(1/\delta)/n\}\leq 1/1000$.
\end{assumption}
The assumption of $\mbb{E}\xi_i^2\mathbb{I}_{|\xi_i|\leq \frac{1}{2}}\geq 1/4$ is called a small-ball-type assumption, and \cite{MinNdaWan2024Robust} also use this assumption. The selection of the explicit constants $1/2$ and $1/4$ is not important.
Assumption~\ref{a:noise} is primarily used for the estimation of $\vecbeta^*$, while Assumptions~\ref{a:noise:EI} and~\ref{a:noise:EI2} are used for the estimation of $\mu^*$.

We now present our method and main result.
\begin{algorithm}[H]
	\caption{ADAPTIVE-ROBUST-AND-SPARSE-ESTIMATION}
	\label{alg:RSEI}
	\begin{algorithmic}[1]
	\REQUIRE Data $\{y_i, \vecx_i\}_{i=1}^{2n}$, initial value of the noise absolute moment $\mr{C}_{\mathrm{ini}}$, tuning parameter for Huber loss $\mr{C}_{\mathrm{H}}$, SLOPE penalty parameters $\mr{C}_{\veclambda}, \{\lambda_i\}_{i=1}^d$ for estimating $\vecbeta^*$, number of the iteration $ \mr{N}_{\mathrm{iter}}$, number of the blocks $\mr{B}$, SLOPE penalty $\{\iota_i\}_{i=1}^n$ for estimating $\mu^*$
	\ENSURE Estimators $\hat{\vecbeta}$ and $\hat{\mu}$
	\STATE $\{y_i', \vecx_i'\}_{i=1}^{n} \leftarrow \text{NORMALIZING}(\{y_i, \vecx_i\}_{i=1}^{2n})$
	\STATE $\hat{\vecbeta} \leftarrow \text{COEFFICIENTS-ESTIMATION}(\{y_i', \vecx_i'\}_{i=1}^n, \mr{C}_{\mathrm{ini}}, \mr{C}_{\mathrm{H}}, \mr{C}_{\veclambda}, \{\lambda_i\}_{i=1}^d, \mr{N}_{\mathrm{iter}}, \mr{B})$
	\STATE $\hat{\mu} \leftarrow \text{INTERCEPT-ESTIMATION}(\{y_i, \vecx_i\}_{i=1}^{2n}, \hat{\vecbeta}, \{\iota_i\}_{i=1}^n)$
	\end{algorithmic}
\end{algorithm}

We briefly describe each step of Algorithm \ref{alg:RSEI}. COEFFICIENTS-ESTIMATION and INTERCEPT-ESTIMATION are discussed in detail in Sections \ref{sec:AESC} and \ref{sec:EI}, respectively. NORMALIZING is performed as follows:
\begin{algorithm}[H]
	\caption{NORMALIZING}
	\begin{algorithmic}[1]
			\label{alg:N}
		\REQUIRE Data $\left\{y_i,\vecx_i\right\}_{i=1}^{2n}$
		\ENSURE Normalized data $\{y'_i,\vecx'_i\}_{i=1}^n$
		\STATE  for $i$ in $1,\cdots, n$
		\STATE 	\quad \quad $(y'_i, \vecx'_i) \leftarrow \left(\frac{y_i-y_{n+i}}{\sqrt{2}},\frac{\vecx_i-\vecx_{n+i}}{\sqrt{2}}\right)$
	\end{algorithmic} 
\end{algorithm}
The NORMALIZING step corresponds to Step 2 in Section~\ref{sec:RWOA}, and produces a version of the data with zero intercept. This transformation preserves the covariance and noise properties. Lemma~\ref{l:L1+} in the appendix guarantees that the absolute moment of the noise is preserved up to a constant factor.
The procedure COEFFICIENTS-ESTIMATION is responsible for estimating $\vecbeta^*$. It iteratively updates both the coefficient and the noise scale estimate until convergence.
INTERCEPT-ESTIMATION uses the final coefficient estimate in a one-dimensional version of the robust penalized objective to estimate $\mu^*$ (Step 6 of Section~\ref{sec:RWOA}).

\subsection{Result}
\label{sec:R2}
Using ADAPTIVE-ROBUST-AND-SPARSE-ESTIMATION, we obtain the following results.
\begin{theorem}
	\label{t:main}
	Suppose Assumptions \ref{a:cov}-\ref{a:noise:EI} hold. 
	Let the tuning parameters be set as  $\lambda_i= \sqrt{\log(ed/i)/n}$ and $\iota_i=\frac{1}{\sqrt{n}}\left(\frac{n}{i}\right)^{\frac{1}{\mr{M}}}$.
	Assume that  $\mr{C}_{\veclambda},\mr{C}_{\mr{H}} $ and $\mr{B} $ to be sufficiently large and that  $\sqrt{\frac{s\log(ed/s)}{n}}+\sqrt{\frac{\log(1/\delta)}{n}}+\sqrt{\frac{o}{n}}$ is sufficiently small. Then, we have
	\begin{align}
		\mbb{P}\left(\|\vecbeta^*-\hat{\vecbeta}\|_\Sigma +|\mu^*-\hat{\mu}|\lesssim  \sigma^* L^5\left(\frac{\rho}{\kappa}\sqrt{\frac{\log(ed/s)}{n}}+\left(\frac{o}{n}\right)^{1-\frac{2}{\mr{M}}}+\sqrt{\frac{\log(1/\delta)}{n}}\right)\right)\geq 1-14\delta.
	\end{align} 
If Assumption \ref{a:noise:EI} is replaced by Assumption~\ref{a:noise:EI2}, and we set $\iota_i=\sqrt{\log(en/i)/n}$, we have
	\begin{align}
		\mbb{P}\left(\|\vecbeta^*-\hat{\vecbeta}\|_\Sigma +|\mu^*-\hat{\mu}|\lesssim  \sigma^* L^5\left(\frac{\rho}{\kappa}\sqrt{\frac{\log(ed/s)}{n}}+\frac{o}{n}\log \frac{n}{o}+\sqrt{\frac{\log(1/\delta)}{n}}\right)\right)\geq 1-14\delta.
	\end{align} 
\end{theorem}
\begin{remark}
	The error bound in Theorem~\ref{t:main} matches that of Theorem 2.2 in \cite{MinNdaWan2024Robust} up to constants. 
	In particular, our coefficient estimation achieves sharper bounds (for details, see  Theorems~\ref{t:Huber} and Remark \ref{r:co}). 
	The conditions stated as `sufficiently small' or `sufficiently large' are clarified in Theorems~\ref{t:Huber}, \ref{t:abs}, and \ref{t:Huberabs}, as well as the accompanying remarks.
\end{remark}
If the value of  $o$ is known, we have the following corollary:
\begin{corollary}
	\label{c:main}
	Suppose Assumptions~\ref{a:cov}-\ref{a:noise:EI} hold. 
	Let the tuning parameters be set as $\lambda_i= \sqrt{\log(ed/i)/n}$ and  $\iota_i=\frac{1}{\sqrt{n}}\left(\frac{n}{o}\right)^{\frac{1}{\mr{M}}}$, using $o$.
	Assume that  $\mr{C}_{\veclambda},\mr{C}_{\mr{H}} $ and $\mr{B} $ to be sufficiently large and that $\sqrt{\frac{s\log(ed/s)}{n}}+\sqrt{\frac{\log(1/\delta)}{n}}+\sqrt{\frac{o}{n}}$ is sufficiently small.
	Then, we have
	\begin{align}
	\label{ine:c:main-1}
			\mbb{P}\left(\|\vecbeta^*-\hat{\vecbeta}\|_\Sigma +|\mu^*-\hat{\mu}|\lesssim  \sigma^* L^5\left(\frac{\rho}{\kappa}\sqrt{\frac{\log(ed/s)}{n}}+\left(\frac{o}{n}\right)^{1-\frac{1}{\mr{M}}}+\sqrt{\frac{\log(1/\delta)}{n}}\right)\right)\geq 1-14\delta.
	\end{align} 
	If Assumption \ref{a:noise:EI} is replaced by Assumption \ref{a:noise:EI2}, and we set  $\iota_i=\sqrt{\log(en/o)/n}$ using $o$,  we have
	\begin{align}
		\label{ine:c:main-2}
		\mbb{P}\left(\|\vecbeta^*-\hat{\vecbeta}\|_\Sigma +|\mu^*-\hat{\mu}|\lesssim  \sigma^* L^5\left(\frac{\rho}{\kappa}\sqrt{\frac{\log(ed/s)}{n}}+\frac{o}{n}\sqrt{\log \frac{n}{o}}+\sqrt{\frac{\log(1/\delta)}{n}}\right)\right)\geq 1-14\delta.
	\end{align} 
\end{corollary}
\begin{remark}
	Corollary~\ref{c:main} requires prior knowledge of $o$, but in exchange, yields sharper error bounds.  
	The rates match the minimax lower bounds established in \cite{MinNdaWan2024Robust} up to constant factors.	
\end{remark}

Finally, if we assume only Assumption~\ref{a:noise} for the noise distribution, and if the number of outliers $o$ is known, we can still obtain the following result by applying an alternative version of \textsc{INTERCEPT-ESTIMATION}. 
This variant differs from the versions used in the preceding theorem and corollary, and is based on the median-of-means principle. It is similar in spirit to the estimation procedure used for $\varsigma^*$ in \textsc{COEFFICIENTS-ESTIMATION}. 
The alternative method is detailed in Algorithm~\ref{alg:IEII} in Section~\ref{sec:EI}. Unlike the formulation in \eqref{opt:mu}, it relies on a block-wise median-of-means strategy and closely resembles the \textsc{ROBUST-ABSOLUTE-MOMENT} algorithm (Algorithm~\ref{alg:ram}). It uses the number of blocks $B'$ as the tuning parameter, in place of the sequence $\veciota = \{\iota_i\}_{i=1}^n$.

\begin{theorem}
	\label{t:main:two}
	Suppose Assumptions \ref{a:cov} - \ref{a:noise} hold. 
	Let the tuning parameters be set as $\lambda_i= \sqrt{\log(ed/i)/n}$, and set $\mr{B}'=\max\{128\log(1/\delta), 8o\}$, using the number of the outlier $o$.
	Assume that  $\mr{C}_{\veclambda},\mr{C}_{\mr{H}} $ and $\mr{B} $ sufficiently large and that $\sqrt{\frac{s\log(ed/s)}{n}}+\sqrt{\frac{\log(1/\delta)}{n}}+\sqrt{\frac{o}{n}}$ is sufficiently small.  Then, we have
			\begin{align}
				\mbb{P}\left(\|\vecbeta^*-\hat{\vecbeta}\|_\Sigma +|\mu^*-\hat{\mu}|\lesssim  \sigma^* L^5\left(\frac{\rho}{\kappa}\sqrt{\frac{\log(ed/s)}{n}}+\sqrt{\frac{o}{n}}+\sqrt{\frac{\log(1/\delta)}{n}}\right)\right)\geq 1-6\delta.
			\end{align} 
\end{theorem}
\begin{remark}
	Although Theorem~\ref{t:main:two} assumes only finite variance noise (Assumption~\ref{a:noise}), it achieves an error bound that matches the lower bound from \cite{MinNdaWan2024Robust} up to constants, assuming the number of outliers $o$ is known.
\end{remark}

\section{Estimation of the Sparse Coefficients}
\label{sec:AESC}
In this section, we analyze the procedure \textsc{COEFFICIENTS-ESTIMATION}. For notational simplicity, we continue to write the normalized data as $\{y_i,\vecx_i\}_{i=1}^n$ instead of $\{y'_i,\vecx'_i\}_{i=1}^n$
\begin{algorithm}[H]
	\caption{COEFFICIENTS-ESTIMATION}
	\begin{algorithmic}[1]
			\label{alg:ARSCE}
		\renewcommand{\algorithmicrequire}{\textbf{Input:}}
		\renewcommand{\algorithmicensure}{\textbf{Output:}}
		\REQUIRE Normalized data $\left\{y_i,\vecx_i\right\}_{i=1}^{n}$, initial value of the noise absolute moment $\mr{C}_{\mr{ini}}$,  tuning parameter for Huber loss $\mr{C}_{\mr{H}}$, SLOPE penalty parameters $\mr{C}_{\veclambda} ,\{\lambda_i\}_{i=1}^d$, number of the iteration $\mr{N}_\mr{Iter}$, number of the blocks $\mr{B}$
		\ENSURE  $\hat{\vecbeta}_{\mr{N}_{\mr{Iter}}}$ 
		\STATE  $\hat{\varsigma}_0 \leftarrow \mr{C}_{\mr{ini}}$
		\STATE  for $\mr{l}$ in $1,\cdots, \mr{N}_\mr{Iter}$
		\STATE 	\quad \quad $\hat{\vecbeta}_l \leftarrow \text{PENALIED-HUBER-REGRESSION}(\left\{y_i,\vecx_i\right\}_{i=1}^{n},\mr{C}_{\mr{H}},\hat{\varsigma}_{l-1},\mr{C}_{\veclambda}, \{\lambda_i\}_{i=1}^d)$
		\STATE \quad \quad $\hat{\varsigma}_l \leftarrow \text{ROBUST-ABSOLUTE-MOMENT}(\left\{y_i,\vecx_i\right\}_{i=1}^{n}, \hat{\vecbeta}_l,\mr{B})$
	\end{algorithmic} 
\end{algorithm}
The subroutines \textsc{PENALIZED-HUBER-REGRESSION} and \textsc{ROBUST-ABSOLUTE-MOMENT} are explained in Sections \ref{sec:ESC} and \ref{sec:RAM}, respectively. Section \ref{sec:CE} presents the final result of the iterative procedure.

	\subsection{PENALIZED-HUBER-REGRESSION}
	\label{sec:ESC}

	This step solves a SLOPE-penalized Huber regression problem:
		\begin{algorithm}[H]
			\caption{PENALIZED-HUBER-REGRESSION}
			\label{alg:PHR}
			\begin{algorithmic}
			\REQUIRE{Normalized data $\left\{y_i,\vecx_i\right\}_{i=1}^{n}$, current estimator of the noise absolute moment $\hat{\varsigma}_{\mr{cur}}$,  tuning parameter for Huber loss $\mr{C}_{\mr{H}}$, SLOPE penalty parameters $\mr{C}_{\veclambda} ,\{\lambda_i\}_{i=1}^d$}
			\ENSURE{estimator of the coefficient $\hat{\vecbeta}$}\\
			Let $\hat{\vecbeta}$ be the solution to 
			\begin{align}
				\label{ine:l1-penalized-Huber}
				\argmin_{\vecbeta  \in \mbb{R}^d} \sum_{i=1}^n (\mr{C}_{\mr{H}} \hat{\varsigma}_{\mr{cur}})^2 H\left(\frac{y_i-\vecx_i^\top\vecbeta}{\mr{C}_{\mr{H}} \hat{\varsigma}_{\mr{cur}}}\right)+\mr{C}_{\veclambda}\mr{C}_{\mr{H}} \hat{\varsigma}_{\mr{cur}}\|\vecbeta\|_{\veclambda},
				\end{align}
			 {\bf return}	$\hat{\vecbeta}$
		\end{algorithmic}
		\end{algorithm}
		Define 
		\begin{align}
			\mr{R}_\Sigma &= \frac{\mr{C}_{\veclambda}}{\kappa}\sqrt{\frac{\log(ed/s)}{n}}+L\sqrt{\frac{\log(1/\delta)}{n}}+L\frac{o}{n}\sqrt{\log \frac{n}{o}},\nonumber\\
			\mr{R}_{\veclambda} &= \text{numerical constant}\times \frac{1}{\rho}\times \mr{R}_\Sigma.
		\end{align}
		Then, for the output of the Algorithm \ref{alg:PHR}, we have the following theorem:
		\begin{theorem}
			\label{t:Huber}
			Suppose that Assumptions \ref{a:cov} and \ref{a:cov-RE} hold,  the noise sequence $\{\xi_i\}_{i=1}^n$ is an i.i.d. sequence of random variables drawn from a distribution such that $\mbb{E}|\xi_i|= \varsigma^*$ with $\mbb{E}\xi_i=0$ for $i=1,\cdots,n$, and that 
			for $i=1,\cdots,n$, $\mbb{E}h(\xi_i/A)\vecx_i=\bf{0}$ holds for any fixed variable $A$. Let the tuning parameters be set as $\lambda_i=\sqrt{\frac{\log(ed/i)}{n}}$, and $\mr{C}_{\veclambda}$ and $ \mr{C}_{\mr{H}}$ as
			\begin{align}
				\mr{C}_{\veclambda} = \text{numerical constant}\times L\rho\text{ and } \mr{C}_{\mr{H}}\geq 72L^4.
			\end{align}
			Assume that
			\begin{align}
				\label{ine:t:Huber:abs}
				\varsigma^*\leq \hat{\varsigma}_{\mr{cur}}.
			\end{align}
			Then, for any $\hat{\varsigma}_{\mr{cur}}$, with probability at least $1-3\delta$, the optimal solution of \eqref{ine:l1-penalized-Huber} satisfies
			\begin{align}
				\|\vecbeta^*-\hat{\vecbeta}\|_\Sigma \leq  \mr{C}_{\mr{PHR}}\mr{C}_{\mr{H}} \hat{\varsigma}_{\mr{cur}}\mr{R}_\Sigma \text{ and }\|\vecbeta^*-\hat{\vecbeta}\|_{\veclambda} \leq  \mr{R}_{\veclambda} \times\|\vecbeta^*-\hat{\vecbeta}\|_\Sigma,
			\end{align} 
			where $\mr{C}_{\mr{PHR}}$ is a sufficiently large constant, provided $\mr{R}_\Sigma$ is sufficiently small so that $(\mr{C}_{\mr{PHR}} \mr{R}_\Sigma)^2 \leq 1/48 L^4$.
		\end{theorem}
		To clarify the implications of the theorem, we provide several remarks below.
		\begin{remark}
			Theorem \ref{t:Huber}  holds for any $\hat{\varsigma}_{\mr{cur}}(\geq \varsigma^*)$, so we do not need to apply a union bound over iterations.
		\end{remark}
		\begin{remark}
			\label{r:co}
			If we set $\hat{\varsigma}_{\mr{cur}}=\varsigma^*$ and $\mr{C}_{\mr{H}}=72L^4$, we have
			\begin{align}
				\|\vecbeta^*-\hat{\vecbeta}\|_\Sigma \lesssim   \varsigma^* L^5\left(\frac{\rho}{\kappa}\sqrt{\frac{\log(ed/s)}{n}}+\sqrt{\frac{\log(1/\delta)}{n}}+\frac{o}{n}\sqrt{\log \frac{n}{o}}\right).
			\end{align}
			A similar bound is obtained in \cite{Chi2020Erm}, which also studies $\ell_1$-penalized Huber regression under the assumption $\mu^*=0$, as in our setting. In \cite{Chi2020Erm}, the derived error bound is:
			\begin{align}
			\label{ine:prehuber}
				\|\vecbeta^*-\hat{\vecbeta}\|_\Sigma \lesssim   \varsigma^* L^5\left(\frac{\rho}{\kappa}\sqrt{\frac{\log(ed/s)}{n}}+\sqrt{\frac{\log(1/\delta)}{n}}+\frac{o}{n}\right).
			\end{align}
			At first glance, this may appear sharper than our result due to the absence of the $\sqrt{\log(n/o)}$ factor.
			However, if we modify our setting to align with the assumptions of \cite{Chi2020Erm}, in particular by assuming that outliers are randomly located instead of being adversarially placed, then our bound also eliminates the $\sqrt{\log(n/o)}$ factor. This is because the additional logarithmic term in our analysis arises from the need to control all possible placements of $o$ outliers among $n$ observations. If such enumeration is not required, as in the case of randomly placed outliers, the logarithmic term does not appear.
			Moreover, the tuning parameter for the $\ell_1$-penalty in \cite{Chi2020Erm} is set as
			\begin{align}
				\label{ine:tuningpro}
				\lambda = O\left(\varsigma^* L\left(\rho\sqrt{\frac{\log(d/s)}{n}}+\kappa \frac{o}{ n\sqrt{s}}+\kappa \sqrt{\frac{\log(1/\delta)}{sn}}\right)\right),
			\end{align}
			which explicitly depends on the unknown values of $o$, $s$, and $\delta$.  In contrast, the tuning parameter required in Theorem~\ref{t:Huber} is independent of these quantities, making it easier to set in practice. This simplification is made possible by our refined analysis of the Huber loss.
		\end{remark}

		\begin{remark}
			\label{r:const}
			Theorem~\ref{t:Huber} involves three tuning parameters: $\mr{C}_{\mr{PHR}}$, $L$, and $\rho$. 
			The constant $\mr{C}_{\mr{PHR}}$ primarily reflects bounds derived from generic chaining techniques \cite{Dir2015Tail} and the majorizing measure \cite{Tal2014Upper}. Obtaining tight estimates for such constants remains a significant challenge both theoretically and practically.
			The parameters $L$ and $\rho$ are determined by the distribution of the observed data $\{\vecx_i\}_{i=1}^n$. mong these, estimating $\rho$ from data is relatively straightforward. 
			In contrast, estimating  $L$ tends to be more difficult. 
			These quantities are common to related works, so studying their estimation may provide interesting insights.
			These quantities appear frequently in related literature, and a better understanding of how to estimate them from finite samples could provide valuable directions for future research.
		\end{remark}

		\begin{remark}
			\label{r:const2}
			If an estimate $\hat{\varsigma}_{\mr{cur}}$ sufficiently close to the true noise level $\varsigma^*$ is available, Theorem~\ref{t:Huber} yields a sharp error bound. However, if $\hat{\varsigma}_{\mr{cur}}$ is set too large, the resulting bound becomes less informative. Even in such cases, the iterative structure of Algorithm~\ref{alg:ARSCE} mitigates this issue: the \textsc{ROBUST-ABSOLUTE-MOMENT} procedure enables refinement of the noise estimate $\hat{\varsigma}_{\mr{cur}}$ at each iteration, gradually bringing it closer to the true value $\varsigma^*$.
		\end{remark}

	\subsection{ROBUST-ABSOLUTE-MOMENT}
	\label{sec:RAM}
	In this step, we use a variant of the median-of-means estimator \cite{AloMatSze1996Space} to robustly estimate the absolute first moment of the noise. The procedure is defined as follows:
		\begin{algorithm}[H]
			\caption{ROBUST-ABSOLUTE-MOMENT}
			\label{alg:ram}
			\begin{algorithmic}
			\REQUIRE{Normalized data $\left\{y_i,\vecx_i\right\}_{i=1}^n$, current estimator of the coefficient $\hat{\vecbeta}$, and the number of the blocks $\mr{B}$}
			\ENSURE{estimator of the absolute noise $\hat{\varsigma}$}\\
			Let $\tilde{\varsigma}$ be the median of 
			\begin{align}
					\frac{1}{N}\sum_{i=(k-1)N+1}^{kN} |y_i-\ \vecx_i^\top \hat{\vecbeta}|,\quad k=1,\cdots,\mr{B},
				\end{align}
				where $k=1,\cdots,\mr{B}$ and $\mr{B}\mr{N}=n$\\
			 {\bf return}	$\hat{\varsigma}= 8\tilde{\varsigma}$

		\end{algorithmic}
		\end{algorithm}

		The following theorem guarantees the performance of Algorithm \ref{alg:ram}.
		\begin{theorem}
			\label{t:abs}
			Suppose that Assumption \ref{a:cov} holds, and that Assumption \ref{a:noise} holds for the normalized data. Furthermore, assume that $\{\vecx_i\}_{i=1}^n$ is independent of $\{\xi_i\}_{i=1}^n$. 
			Let $\hat{\varsigma}_{\mr{cur}}>0$ be any current estimate of the absolute noise moment. Define the following set:
			\begin{align}
				\label{vabs}
				\mr{V}_{\mr{abs}}:=\left\{\vecv \in \mbb{R}^d\,\mid\, \|\vecv\|_\Sigma\leq  \mr{C}_{\mr{PHR}}\mr{C}_{\mr{H}} \hat{\varsigma}_{\mr{cur}}\mr{R}_\Sigma,\,\|\vecv\|_{\veclambda}\leq \mr{R}_{\veclambda} \mr{C}_{\mr{PHR}}\mr{C}_{\mr{H}} \hat{\varsigma}_{\mr{cur}}\mr{R}_\Sigma\right\}.
			\end{align}
			Assume  that the current estimator satisfies  $\hat{\vecbeta}-\vecbeta^* \in \mr{V}_{\mr{abs}}$, and that the constant $\mr{C}_{\mr{B}}$  is set large enough to satisfy  
			\begin{align}
				\max\left\{2\times 24^2, \left(\frac{24\times 160}{\mr{c}_{\xi}}\left(\frac{\mr{c}_\xi}{20}+L^4 \right)\right)^2, \left(\frac{8\times 24}{\mr{c}_{\xi}} (\mr{c}_\xi+20L^4)\right)^2\right\}\leq \mr{C}_{\mr{B}}.
			\end{align}
			Let the number of blocks be set as $\mr{B} = \frac{n}{\mr{c}_\xi^2 \mr{C}_{\mr{B}}}$, and assume $\sqrt{\mr{B}}\geq 48$. We also assume that $\mr{B}$ and $(\mr{N}: =)\mr{c}_\xi^2 \mr{C}_{\mr{B}}$ are integers.
			Additionally, assume the following conditions hold:
			\begin{align}
			\rho \mr{R}_{\veclambda}\leq 1,\, L\left(\mr{R}_{\veclambda}+\sqrt{s\log(ed/s)/n}\right)&\leq 1/(64\sqrt{e} \mr{C}_{\mr{B}}\mr{c}_\xi),\nonumber\\
			L\left(\rho \mr{R}_{\veclambda}+\sqrt{\log(1/\delta)/n}\right)\text{ and } (L^5+\mr{c}_\xi )/\sqrt{n}&\text{ are sufficiently small},\nonumber\\
			\max\{o/n,\log(1/\delta)/n\}&\leq 1/ \mr{c}_\xi^2\mr{C}^2_{\mr{B}}.\nonumber
			\end{align}
			Then, with probability at least $1-2\delta$, the output of Algorithm \ref{alg:ram} ($\hat{\varsigma}$) satisfies the following properties: 
			\begin{align}
				\label{ine:abs-0}
				\varsigma^* \leq \hat{\varsigma} \leq 8\max\{ 3\mr{C}_{\mr{PHR}}\mr{C}_{\mr{H}}\hat{\varsigma}_{\mr{cur}}\mr{R}_\Sigma , 25 L^4 \varsigma\}.
			\end{align}
		\end{theorem}

		To clarify the implications of the theorem, we provide several remarks below.
		\begin{remark}
			Theorem~\ref{t:abs} holds for any choice of $\hat{\varsigma}_{\mr{cur}}$ and any $\hat{\vecbeta}$ such that $\hat{\vecbeta}-\vecbeta^*\in \mr{V}_{\mr{abs}}$. Therefore, it is not necessary to apply a union bound across iterations.
		\end{remark}
		\begin{remark}
			\label{r:mmre}
			The use of the median-of-means estimator for estimating noise moments, as in Algorithm~\ref{alg:ram}, is also studied in \cite{ComColNdaTsy2021Adaptive}. However, our setting introduces additional analytical challenges because $\hat{\varsigma}_{\mr{cur}}$ depends on the data. This dependency necessitates the use of empirical process theory. Related techniques have also been developed for robust mean estimation \cite{LugMen2019Sub} and covariance matrix estimation \cite{MenZhi2020Robust}.
		\end{remark}
		\begin{remark}
			\label{r:mmin}
			With high probability, the output $\hat{\varsigma}$ satisfies $\hat{\varsigma}\geq \varsigma^*$, which is crucial for valid use as a tuning parameter in subsequent Huber regression steps.
			Moreover, if $\hat{\varsigma}\leq 8 \times 25 L^4 \varsigma^*$, then the estimate remains within a constant factor of the true value. 
			Additionally, if  $\hat{\varsigma}\leq 8 \times 3\mr{C}_{\mr{PHR}}\mr{C}_{\mr{H}}\hat{\varsigma}_{\mr{cur}}\mr{R}_\Sigma$  and $24 \mr{C}_{\mr{PHR}}\mr{C}_{\mr{H}}\mr{R}_\Sigma < 1$, then  $\hat{\varsigma}$ is strictly smaller than $\hat{\varsigma}_{\mr{cur}}$, implying convergence of the estimate across iterations.
		\end{remark}
		\begin{remark}
			\label{r:mm}
			The constant $\mr{c}_\xi$, which serves as an upper bound on the ratio between the standard deviation and the absolute moment of the noise, is typically unknown. 
			However, even if $\mr{c}_\xi$ is overestimated when setting the block count, the final result remains valid, provided that conditions such as
			\begin{align}
				L\left(\mr{R}_{\veclambda}+\sqrt{s\log(ed/s)/n}\right)\leq 1/(64\sqrt{e} \mr{C}_{\mr{B}}\mr{c}_\xi),\quad \max\{o/n,\log(1/\delta)/n\}\leq 1/ \mr{c}_\xi^2\mr{C}^2_{\mr{B}}
			\end{align}
			are satisfied, and $ (L^5+\mr{c}_\xi )/\sqrt{n}$ being sufficiently small. The required bounds for being `sufficiently small' depend on tools such as generic chaining \cite{Dir2015Tail} and majorizing measures \cite{Tal2014Upper}, as also discussed in Remark~\ref{r:const}.
		\end{remark}

		\subsection{COEFFICIENTS-ESTIMATION}
		\label{sec:CE}
		We now analyze the final output of the iterative procedure \textsc{COEFFICIENTS-ESTIMATION}. The following theorem provides an upper bound on the estimation error of the final coefficient estimator  $\hat{\vecbeta}_{\mr{N}_{\mr{Iter}}}$.
		\begin{theorem}
			\label{t:Huberabs}
			Suppose the assumptions in Theorems \ref{t:Huber} and \ref{t:abs} hold. 
			Additionally, assume  the following conditions are satisfies: 
			\begin{align}
			\mr{C}_{\mr{ini}}&\geq \varsigma^*,\quad 24  \mr{C}_{\mr{PHR}}\mr{C}_{\mr{H}} \mr{R}_\Sigma<1,\nonumber\\
			 \mr{N}_{\mr{Iter}}&\text{ is sufficiently large so that }\mr{C}_{\mr{ini}} (24  \mr{C}_{\mr{PHR}}\mr{C}_{\mr{H}} \mr{R}_\Sigma)^{\mr{N}_{\mr{Iter}}-1}\leq  200L^4 \varsigma^*.\nonumber
		\end{align}
			Then, with probability at least $1-5\delta$, the output of \textsc{COEFFICIENTS-ESTIMATION} satisfies the following bound:
			\begin{align}
				\label{ine:t:Huberabs}
				\|\vecbeta^*-\hat{\vecbeta}_{\mr{N}_{\mr{Iter}}} \|_\Sigma \leq 200L^4 \mr{C}_{\mr{PHR}}\mr{C}_{\mr{H}}\varsigma^* \mr{R}_\Sigma\,\text{ and }\,\|\vecbeta^*-\hat{\vecbeta}_{\mr{N}_{\mr{Iter}}} \|_{\veclambda} \leq   \mr{R}_{\veclambda}\times \|\vecbeta^*-\hat{\vecbeta}_{\mr{N}_{\mr{Iter}}} \|_\Sigma.
			\end{align}
		\end{theorem}

		\begin{remark}
			\label{r:coefinalresult}
			The error bound established in Theorem~\ref{t:Huberabs} aligns with the result in Theorem~\ref{t:Huber}, up to multiplicative constants and the dependence on $L$. 
			This indicates that our iterative estimator attains the same statistical performance as the one that requires prior knowledge of $\varsigma^*$.
			Furthermore, our analysis leverages tools from empirical process theory, allowing us to avoid both sample splitting and repeated union bounds across iterations. As a result, it is possible to initialize the procedure with a conservatively large value $\mr{C}_{\mr{ini}}$ and simply choose the number of iterations $\mr{N}_{\mr{Iter}}$ large enough to ensure convergence.
		\end{remark}

\section{Estimation of the Intercept}
\label{sec:EI}
In this section, we present a method for estimating the intercept $\mu^*$ in the regression model, based on the data  $\{\vecx_i,y_i\}_{i=1}^{2n}$ and a previously obtained estimator of the coefficient vector $\hat{\vecbeta}_{\mr{N}_{\mr{Iter}}}$.
Our goal is to construct a procedure that remains robust in the presence of outliers and heavy-tailed noise.
For simplicity of the notation, we describe the procedure using only the first $n$ observations, $\{\vecx_i,y_i\}_{i=1}^{n}$, and assume that the number of outliers in this subset is $o$. Let $\hat{\vecbeta}$ denote the final output of the \textsc{COEFFICIENTS-ESTIMATION} algorithm from Section~\ref{sec:AESC}, $\hat{\vecbeta}_{\mr{N}_{\mr{Iter}}}$.

We estimate the intercept using a one-dimensional version of the robust penalized regression problem introduced earlier \eqref{ine:l1-penalized-Huber-re2}. Specifically, we consider the following algorithm:
\begin{algorithm}[H]
	\caption{INTERCEPT-ESTIMATION}
	\label{alg:I}
	\begin{algorithmic}
	\REQUIRE{Input data $\left\{y_i,\vecx_i\right\}_{i=1}^n$, estimator of the coefficient $\hat{\vecbeta}$, tuning parameter for penalization loss $\{\iota_i\}_{i=1}^n$}
	\ENSURE{estimator of the intercept $\hat{\mu}$}\\
	Let $\hat{\mu}$ be the solution to 
	\begin{align}
		\label{ine:SQ}
		\argmin_{\mu \in \mbb{R},\,\vectheta \in \mbb{R}^n} \sqrt{\sum_{i=1}^n (y_i-\vecx_i^\top\hat{\vecbeta}-\mu-\sqrt{n} \theta)^2}+80\|\vectheta\|_{\veciota}
		\end{align}
	 {\bf return}	$\hat{\mu}$   

\end{algorithmic}
\end{algorithm}
To state the estimation guarantee, define the following quantities:
\begin{align}
	r_\iota:=\sqrt{\sum_{i=1}^{o+\max\{\log(1/\delta),o\}} \iota_i^2},\quad
	r_\delta: = \sqrt{\frac{\log (1/\delta)}{n}},\quad  r_\Sigma := 200L^4 \mr{C}_{\mr{PHR}}\mr{C}_{\mr{H}} \mr{R}_\Sigma.
\end{align}
Also, under the results from Section~\ref{sec:CE}, we can assume that, with probability at least $1-5\delta$, the estimation error of $\hat{\vecbeta}$ satisfies:
\begin{align}
	\|\vecbeta^*-\hat{\vecbeta} \|_\Sigma \leq \sigma^* r_\Sigma\,\text{ and }\,\|\vecbeta^*-\hat{\vecbeta} \|_{\veclambda} \leq   \sigma^*  \mr{R}_{\veclambda}r_\Sigma .
\end{align}

\begin{theorem}[Heavy-tailed case]
	\label{t:intercept}
	Define $r_{\mr{M}}$ as
	\begin{align}
		r_{\mr{M}}=\left(\frac{\max\{\log(1/\delta),o\}}{n}\right)^{1-\frac{1}{\mr{M}}}.
	\end{align}
	Assume that $\|\hat{\vecbeta}-\vecbeta^*\|_\Sigma\leq \sigma^* r_\Sigma$ and $\|\hat{\vecbeta}-\vecbeta^*\|_{\veclambda}\leq \sigma^* \mr{R}_{\veclambda}r_{\Sigma}$ hold, and that Assumptions \ref{a:cov} and \ref{a:noise:EI} hold.  Let the sequence  $\veciota=(\iota_1,\cdots,\iota_n) $ satisfy  $\iota_1\geq  \cdots \geq\iota_n>0$, and suppose 
	\begin{align}
		\label{ine:noiseslope}
		\iota_{ \max\{\log(1/\delta),o\}}= \frac{1}{\sqrt{n}}\left(\frac{n}{ \max\{\log(1/\delta),o\}}\right)^{1/\mr{M}}.
	\end{align}
	Then, if  $r_{\mr{M}}, r_\iota, r_\Sigma$  and $\rho \mr{R}_{\veclambda}$ are sufficiently small, with probability at least $1-9\delta$, we have  
	\begin{align}
		\label{ine:noiseslope1-2}
		&|\mu^*-\hat{\mu}| \lesssim c_{\mr{M}}\sigma^*  \left(r_{\mr{M}}\left(r_\iota+r_\delta+r_{\mr{M}}+Lr_\Sigma \right) +\frac{   \left(r_\iota+r_\delta+r_{\mr{M}}+L(r_\Sigma+\rho \mr{R}_{\veclambda}r_\Sigma)\right)^2}{\sqrt{n}\iota_{n}}\right).
	\end{align}	
\end{theorem}

\begin{theorem}[SubGaussian case]
	\label{t:intercept'}
	Define $r_{\mr{M}}$ as
	\begin{align}
		r_{\mr{M}}=\frac{o}{n}\sqrt{\log \frac{n}{o}}.
	\end{align}
	Assume that $\|\hat{\vecbeta}-\vecbeta^*\|_\Sigma\leq \sigma^* r_\Sigma$ and $\|\hat{\vecbeta}-\vecbeta^*\|_{\veclambda}\leq \sigma^* \mr{R}_{\veclambda}r_{\Sigma}$ hold, and that Assumptions \ref{a:cov} and \ref{a:noise:EI2} hold.  Let the sequence  $\veciota=(\iota_1,\cdots,\iota_n) $ satisfy  $\iota_1\geq  \cdots \geq\iota_n>0$, and suppose 
	\begin{align}
		\label{ine:noiseslope'}
		\iota_{\max\{o,\log(1/\delta)\}}=\sqrt{\frac{\log(en/(\max\{o,\log(1/\delta)\}))}{n}}.
	\end{align}
	Then, if  $r_{\mr{M}}, r_\iota, r_\Sigma$  and $\rho \mr{R}_{\veclambda}$ are sufficiently small, with probability at least $1-9\delta$, we have  
	\begin{align}
		\label{ine:noiseslope'1-2}
		&|\mu^*-\hat{\mu}| \lesssim c_{\mr{M}}\sigma^*  \left(r_{\mr{M}}\left(r_\iota+r_\delta+r_{\mr{M}}+Lr_\Sigma \right) +\frac{   \left(r_\iota+r_\delta+r_{\mr{M}}+L(r_\Sigma+\rho \mr{R}_{\veclambda}r_\Sigma)\right)^2}{\sqrt{n}\iota_{n}}\right).
	\end{align}	
\end{theorem}

\begin{remark}
	In the heavy-tailed case, if we choose $\iota_i= \frac{1}{\sqrt{n}}\left(\frac{n}{i}\right)^{1/\mr{M}}$ for $i=1,\cdots,n$, we have $\sqrt{n}\iota_n=1$ and
	\begin{align}
		\label{ine:rem:interceptorder1}
		r_\iota \lesssim  \left(\frac{o}{n}\right)^{\frac{1}{2}-\frac{1}{\mr{M}}}+ \left(\frac{\log(1/\delta)}{n}\right)^{\frac{1}{2}-\frac{1}{\mr{M}}}.
	\end{align}
	This implies that the dominant term about $o$ and $\log (1/\delta)$ in the error bound scales as $r_\iota^2$, consistent with the rate in \eqref{ine:intro-pr2} for heavy-tailed noise.
	
	In the sub-Gaussian case, if we choose $\iota_i= \sqrt{\frac{\log(en/i)}{n}}$ for $i=1,\cdots,n$, we have $\sqrt{n} \iota_n=1$ and 
	\begin{align}
		r_\iota \lesssim \sqrt{\frac{o}{n}\log \frac{n}{o}}+\left(\frac{\log(1/\delta)}{n}\right)^\frac{1}{4}.
	\end{align}
	Thus, the error bound matches \eqref{ine:intro-pr1} up to constants.
\end{remark}

Combining Theorems \ref{t:intercept}, \ref{t:intercept'} and \ref{t:Huber}, we have Theorem \ref{t:Huberabs}.
If we know  $o$, then Theorems~\ref{t:intercept} and~\ref{t:intercept2} immediately yield the following corollary, which implies Corollary \ref{c:main}.
\begin{corollary}
	\label{c:intercept}
	Assume that  $\|\hat{\vecbeta}-\vecbeta^*\|_\Sigma\leq \sigma^* r_\Sigma$ and $\|\hat{\vecbeta}-\vecbeta^*\|_{\veclambda}\leq \sigma^* \mr{R}_{\veclambda} r_{\Sigma}$. Let the sequence $\veciota=(\iota_1,\cdots,\iota_n) $ satisfy $\iota_i= \frac{1}{\sqrt{n}}\left(\frac{n}{o}\right)^{1/\mr{M}}$ for $i=1,\cdots,n$. 	
	Then, if $\left(\frac{\max\{\log(1/\delta),o\}}{n}\right)^{1-\frac{1}{\mr{M}}}, r_\iota, r_\Sigma$  and $\rho \mr{R}_{\veclambda}$ are sufficiently small, with probability at least $1-9\delta$, we have  
	\begin{align}
		|\mu^*-\hat{\mu}|&\lesssim c_{\mr{M}}\sigma^* \times\left(L r_\Sigma +		\sqrt{\frac{\max\{\log(1/\delta),o\}}{n}} \sqrt{\frac{o}{n}}\left(\frac{n}{o}\right)^{\frac{1}{\mr{M}}}+ \frac{o}{n}\left(\frac{n}{o}\right)^{\frac{1}{\mr{M}}}\right).
	\end{align}	
	If we assume Assumption \ref{a:noise:EI2} instead of Assumption \ref{a:noise:EI},  let the sequence  $\veciota=(\iota_1,\cdots,\iota_n) $ satisfy  $\iota_i= \sqrt{\frac{\log(en/o)}{n}}$ for $i=1,\cdots,n$. Then, if $\frac{o}{n}\sqrt{\log \frac{n}{o}}, r_\iota, r_\Sigma$  and $\rho \mr{R}_{\veclambda}$ are sufficiently small, with probability at least $1-9\delta$,  we have
	\begin{align}
		|\mu^*-\hat{\mu}|&\lesssim c_{\mr{M}}\sigma^* \times\left(Lr_\Sigma +		\sqrt{\frac{\max\{\log(1/\delta),o\}}{n}} \sqrt{\frac{o}{n}}\sqrt{\log \frac{en}{o}}+ \frac{o}{n}\sqrt{\log \frac{en}{o}}\right).
	\end{align}
\end{corollary}

When only Assumption~\ref{a:noise} is available and an upper bound $O$ of $o$ is known, we use a median-of-means version for intercept estimation:
\begin{algorithm}[H]
	\caption{INTERCEPT-ESTIMATION (for only finite variance noise)}
	\label{alg:IEII}
	\begin{algorithmic}
	\REQUIRE{Data $\left\{y_i,\vecx_i\right\}_{i=1}^n$, estimator of the coefficient $\hat{\vecbeta}$, and  the number of the blocks $\mr{B}$}
	\ENSURE{estimator of the intercept $\hat{\mu}$}\\
	Let $\hat{\mu}$ be the median of 
	\begin{align}
			\frac{1}{N}\sum_{i=(k-1)N+1}^{kN} (y_i-\ \vecx_i^\top \hat{\vecbeta}),\quad k=1,\cdots,\mr{B},
		\end{align}
		where $k=1,\cdots,\mr{B}$ and $\mr{B}\mr{N}=n$\\
	 {\bf return}	$\hat{\mu}$   

\end{algorithmic}
\end{algorithm}
\begin{theorem}
	\label{t:intercept2}
	Suppose that Assumptions \ref{a:cov} and  \ref{a:noise} hold, and that $\mr{B}=\max\{128\log(1/\delta), 8O\}$ and $\rho\mr{R}_{\veclambda}\leq 1$ hold.
	Let $\hat{\vecbeta}$ satisfy 
	\begin{align}
		\hat{\vecbeta}-\vecbeta^* \in  \mr{V}_{\mr{mean}}:=\left\{\vecv \in \mbb{R}^d\,\mid\, \|\vecv\|_\Sigma\leq  \sigma^* r_\Sigma ,\,\|\vecv\|_{\veclambda}\leq \sigma^* \mr{R}_{\veclambda} r_\Sigma\right\}.
	\end{align}
	Then, we have
	\begin{align}
		\mbb{P}\left(\left|\hat{\mu}-\mu^*\right| \lesssim \sigma^*(1+r_\Sigma )\left(\sqrt{\frac{O}{n}}+\sqrt{\frac{\log(1/\delta)}{n}}\right)+L\sigma^* r_\Sigma\right)\geq 1-\delta.
	\end{align}
\end{theorem}
\begin{remark}
	In Theorem  \ref{t:intercept2}, if we can set $O=o$, then,  we have
	\begin{align}
		\mbb{P}\left(\left|\hat{\mu}-\mu^*\right| \lesssim \sigma^*(1+r_\Sigma )\left(\sqrt{\frac{o}{n}}+\sqrt{\frac{\log(1/\delta)}{n}}\right)+L\sigma^* r_\Sigma\right)\geq 1-\delta,
	\end{align}
	and we derive Theorem \ref{t:main:two}. 
\end{remark}

\bibliographystyle{plain}
\bibliography{ETARSHDLR}

\begin{thebibliography}{10}

\bibitem{AloMatSze1996Space}
Noga Alon, Yossi Matias, and Mario Szegedy.
\newblock The space complexity of approximating the frequency moments.
\newblock In {\em Proceedings of the twenty-eighth annual ACM symposium on Theory of computing}, pages 20--29, 1996.

\bibitem{AlqCotLec2019Estimation}
Pierre Alquier, Vincent Cottet, and Guillaume Lecu{\'e}.
\newblock Estimation bounds and sharp oracle inequalities of regularized procedures with lipschitz loss functions.
\newblock {\em The Annals of Statistics}, 47(4):2117--2144, 2019.

\bibitem{BalWaiYun2017Statistical}
Sivaraman Balakrishnan, Martin~J. Wainwright, and Bin Yu.
\newblock {Statistical guarantees for the EM algorithm: From population to sample-based analysis}.
\newblock {\em The Annals of Statistics}, 45(1):77 -- 120, 2017.

\bibitem{BelLecTsy2018Slope}
Pierre~C Bellec, Guillaume Lecu{\'e}, and Alexandre~B Tsybakov.
\newblock Slope meets lasso: improved oracle bounds and optimality.
\newblock {\em The Annals of Statistics}, 46(6B):3603--3642, 2018.

\bibitem{BelCheWan2011Square}
Alexandre Belloni, Victor Chernozhukov, and Lie Wang.
\newblock Square-root lasso: pivotal recovery of sparse signals via conic programming.
\newblock {\em Biometrika}, 98(4):791--806, 2011.

\bibitem{BouLugMas2013concentration}
St{\'e}phane Boucheron, G{\'a}bor Lugosi, and Pascal Massart.
\newblock {\em Concentration inequalities: A nonasymptotic theory of independence}.
\newblock Oxford university press, 2013.

\bibitem{Bou2002Bennett}
Olivier Bousquet.
\newblock A bennett concentration inequality and its application to suprema of empirical processes.
\newblock {\em Comptes Rendus Mathematique}, 334(6):495--500, 2002.

\bibitem{CheLiaChe2021Cross}
Denis Chetverikov, Zhipeng Liao, and Victor Chernozhukov.
\newblock On cross-validated lasso in high dimensions.
\newblock {\em The Annals of Statistics}, 49(3):1300--1317, 2021.

\bibitem{Chi2020Erm}
Geoffrey Chinot.
\newblock Erm and rerm are optimal estimators for regression problems when malicious outliers corrupt the labels.
\newblock {\em Electronic Journal of Statistics}, 14(2):3563--3605, 2020.

\bibitem{ComColNdaTsy2021Adaptive}
La{\"e}titia Comminges, Olivier Collier, Mohamed Ndaoud, and Alexandre~B Tsybakov.
\newblock Adaptive robust estimation in sparse vector model.
\newblock {\em The Annals of Statistics}, 49(3):1347--1377, 2021.

\bibitem{DalTho2019Outlier}
Arnak Dalalyan and Philip Thompson.
\newblock Outlier-robust estimation of a sparse linear model using $\ell_1$-penalized huber's $m$-estimator.
\newblock In {\em Advances in Neural Information Processing Systems 32}, pages 13188--13198. Curran Associates, Inc., 2019.

\bibitem{PenGin2012Decoupling}
Victor De~la Pena and Evarist Gin{\'e}.
\newblock {\em Decoupling: from dependence to independence}.
\newblock Springer Science \& Business Media, 2012.

\bibitem{Der2018Improved}
Alexis Derumigny.
\newblock Improved bounds for square-root lasso and square-root slope.
\newblock {\em Electronic Journal of Statistics}, 12(1):741--766, 2018.

\bibitem{DiaKan2023Algorithmic}
Ilias Diakonikolas and Daniel~M Kane.
\newblock {\em Algorithmic high-dimensional robust statistics}.
\newblock Cambridge university press, 2023.

\bibitem{Dir2015Tail}
Sjoerd Dirksen.
\newblock Tail bounds via generic chaining.
\newblock {\em Electronic Journal of Probability}, 20:1--29, 2015.

\bibitem{FanLiuSunZha2018Lamm}
Jianqing Fan, Han Liu, Qiang Sun, and Tong Zhang.
\newblock I-lamm for sparse learning: Simultaneous control of algorithmic complexity and statistical error.
\newblock {\em Annals of statistics}, 46(2):814, 2018.

\bibitem{Haa1982Best}
Uffe Haagerup.
\newblock The best constants in the khintchine inequality.
\newblock {\em STUDIA MATHEMATICA}, 70, 1982.

\bibitem{LecLer2017Robust}
Guillaume Lecu{\'e} and Matthieu Lerasle.
\newblock Robust machine learning by median-of-means: theory and practice.
\newblock {\em arXiv preprint arXiv:1711.10306}, 2017.

\bibitem{LugMen2019Mean}
G{\'a}bor Lugosi and Shahar Mendelson.
\newblock Mean estimation and regression under heavy-tailed distributions: A survey.
\newblock {\em Foundations of Computational Mathematics}, 19(5):1145--1190, 2019.

\bibitem{LugMen2019Sub}
G{\'a}bor Lugosi and Shahar Mendelson.
\newblock Sub-gaussian estimators of the mean of a random vector.
\newblock {\em The annals of statistics}, 47(2):783--794, 2019.

\bibitem{Men2015lLearning}
Shahar Mendelson.
\newblock Learning without concentration.
\newblock {\em Journal of the ACM (JACM)}, 62(3):1--25, 2015.

\bibitem{MenZhi2020Robust}
Shahar Mendelson and Nikita Zhivotovskiy.
\newblock {Robust covariance estimation under $L_{4}-L_{2}$ norm equivalence}.
\newblock {\em The Annals of Statistics}, 48(3):1648 -- 1664, 2020.

\bibitem{MinNdaWan2024Robust}
Stanislav Minsker, Mohamed Ndaoud, and Lang Wang.
\newblock Robust and tuning-free sparse linear regression via square-root slope.
\newblock {\em SIAM Journal on Mathematics of Data Science}, 6(2):428--453, 2024.

\bibitem{NguTra2012Robust}
Nam~H Nguyen and Trac~D Tran.
\newblock Robust lasso with missing and grossly corrupted observations.
\newblock {\em IEEE transactions on information theory}, 59(4):2036--2058, 2012.

\bibitem{PenJogLoh2024Robust}
Ankit Pensia, Varun Jog, and Po-Ling Loh.
\newblock Robust regression with covariate filtering: Heavy tails and adversarial contamination.
\newblock {\em Journal of the American Statistical Association}, pages 1--12, 2024.

\bibitem{PraSugBalRav2020Robust}
Adarsh Prasad, Arun~Sai Suggala, Sivaraman Balakrishnan, and Pradeep Ravikumar.
\newblock Robust estimation via robust gradient estimation.
\newblock {\em Journal of the Royal Statistical Society Series B: Statistical Methodology}, 82(3):601--627, 2020.

\bibitem{SheOwe2011Outlier}
Yiyuan She and Art~B Owen.
\newblock Outlier detection using nonconvex penalized regression.
\newblock {\em Journal of the American Statistical Association}, 106(494):626--639, 2011.

\bibitem{SuCan2016Slope}
Weijie Su and Emmanuel Candes.
\newblock Slope is adaptive to unknown sparsity and asymptotically minimax.
\newblock {\em The Annals of Statistics}, 44(3):1038--1068, 2016.

\bibitem{SunZhoFan2020Adaptive}
Qiang Sun, Wen-Xin Zhou, and Jianqing Fan.
\newblock Adaptive {H}uber regression.
\newblock {\em Journal of the American Statistical Association}, 115(529):254--265, 2020.

\bibitem{SunZha2012Scaled}
Tingni Sun and Cun-Hui Zhang.
\newblock Scaled sparse linear regression.
\newblock {\em Biometrika}, 99(4):879--898, 2012.

\bibitem{Tal2014Upper}
Michel Talagrand.
\newblock {\em Upper and lower bounds for stochastic processes}, volume~60.
\newblock Springer, 2014.

\bibitem{Tho2020Outlier}
Philip Thompson.
\newblock Outlier-robust sparse/low-rank least-squares regression and robust matrix completion.
\newblock {\em arXiv preprint arXiv:2012.06750}, 2020.

\bibitem{Tho2023Outlier}
Philip Thompson.
\newblock Outlier-robust additive matrix decomposition and robust matrix completion.
\newblock {\em arXiv preprint arXiv:2310.19136}, 2023.

\bibitem{Tib1996Regression}
Robert Tibshirani.
\newblock Regression shrinkage and selection via the lasso.
\newblock {\em Journal of the Royal Statistical Society: Series B}, 58(1):267--288, 1996.

\bibitem{Gee2016Estimation}
Sara~A Van~de Geer et~al.
\newblock {\em Estimation and testing under sparsity}.
\newblock Springer, 2016.

\bibitem{Ver2018High}
Roman Vershynin.
\newblock {\em High-dimensional probability: An introduction with applications in data science}, volume~47.
\newblock Cambridge university press, 2018.

\bibitem{WanZheZhoZho2021New}
Lili Wang, Chao Zheng, Wen Zhou, and Wen-Xin Zhou.
\newblock A new principle for tuning-free huber regression.
\newblock {\em Statistica Sinica}, 31(4):2153--2177, 2021.

\end{thebibliography}

\appendix

\begin{center} \textbf{Appendix} \end{center}
Throughout the appendix, the value of the numerical constant $C$ is allowed to vary from line to line. Additionally, $\vecg^d = (g_1, \cdots, g_d)$ and $\vecg^n = (g_1, \cdots, g_n)$ denote $d$- and $n$-dimensional standard Gaussian vectors, respectively, and $g$ denotes a standard Gaussian random variable.

\section{Key Tools in Sections \ref{sec:AESC} and \ref{sec:EI}}
\label{sec:KP}

	In Section~\ref{sec:KP}, we present five propositions and one corollary. First, Proposition~\ref{p:main} provides error bounds for the optimal solution of \eqref{ine:l1-penalized-Huber}, without assuming any randomness in $\{\vecx_i, \xi_i\}_{i=1}^n$ or in $\hat{\varsigma}$. Second, Proposition~\ref{p:ub} verifies inequality \eqref{ine:p:main-01} stated in Proposition~\ref{p:main}. Third, Proposition~\ref{p:sc} verifies inequality \eqref{ine:p:main-02} from Proposition~\ref{p:main}. Fourth, Proposition~\ref{p:oq} supports the control of outlier effects in the verification of \eqref{ine:p:main-01}. Fifth, Proposition~\ref{p:main-abs} ensures that the output of \textsc{ROBUST-ABSOLUTE-MOMENT} behaves as expected. Lastly, Corollary~\ref{c:main-mean} is presented, which is used in the proof of Theorem~\ref{t:intercept2}. The proofs of these results are given in Section~\ref{sec:keyproposition}.

	\begin{proposition}
	\label{p:main}
	Define
	\begin{align}
		r_\Sigma(\hat{\varsigma}_{\mr{cur}}) := 6\mr{C}_{\mr{H}} \hat{\varsigma}_{\mr{cur}} \left(r_{\mr{l},\Sigma } + r_{\mr{u},\Sigma}+\left(\mr{C}_{\veclambda}+\frac{r_{\mr{l},\lambda} +r_{\mr{u},\lambda}}{\mr{C}_{\veclambda}/2-r_{\mr{l},\lambda } -r_{\mr{u},\lambda}}\right)\sqrt{\frac{\sum_{i=1}^s \lambda_i^2}{\kappa}} \right)
	\end{align}
	
	Suppose that, for any $\vecv \in \mathbb{R}^d$, the following holds:
	\begin{align}
	\label{ine:p:main-01}
		&\left| \frac{1}{n}\sum_{i=1}^n \mr{C}_{\mr{H}} \hat{\varsigma}_{\mr{cur}}h\left(\frac{y_i- \vecx_i^\top \vecbeta^* }{\mr{C}_{\mr{H}} \hat{\varsigma}_{\mr{cur}}}\right)\vecx_i^\top \vecv\right|\leq \mr{C}_{\mr{H}} \hat{\varsigma}_{\mr{cur}}r_{\mr{u},\lambda}\|\vecv\|_{\veclambda}+\mr{C}_{\mr{H}} \hat{\varsigma}_{\mr{cur}}r_{\mr{u},\Sigma} \|\vecv\|_\Sigma,
	\end{align}
	and for any $\vecv$ such that $\|\vecv\|_\Sigma = r_\Sigma(\hat{\varsigma}_{\mathrm{cur}})$,
	\begin{align}
	\label{ine:p:main-02}
		&\frac{1}{3}\|\vecv\|_\Sigma^2- \mr{C}_{\mr{H}} \hat{\varsigma}_{\mr{cur}}r_{\mr{l},\Sigma }\|\vecv\|_\Sigma- \mr{C}_{\mr{H}} \hat{\varsigma}_{\mr{cur}}r_{\mr{l},\lambda} \|\vecv\|_{\veclambda}\nonumber\\
		&\quad \quad \leq \mr{C}_{\mr{H}} \hat{\varsigma}_{\mr{cur}} \left(-\frac{1}{n}\sum_{i=1}^n h\left(\frac{y_i- \vecx_i^\top(\vecbeta^*+\vecv )}{\mr{C}_{\mr{H}} \hat{\varsigma}_{\mr{cur}}}\right) \vecx_i^\top \vecv+\frac{1}{n}\sum_{i=1}^n h\left(\frac{y_i- \vecx_i^\top \vecbeta^* }{\mr{C}_{\mr{H}} \hat{\varsigma}_{\mr{cur}}}\right) \vecx_i^\top \vecv \right),
	\end{align}
	where $r_{\mathrm{u},\Sigma}, r_{\mathrm{u},\lambda}, r_{\mathrm{l},\Sigma} > 0$, $r_{\mathrm{l},\lambda} \geq 0$, $\mr{C}_{\mathrm{H}}$ is a numerical constant, and $\hat{\varsigma}_{\mathrm{cur}}$ is the current estimate of the absolute moment of the noise. Suppose that Assumption~\ref{a:cov-RE} holds, and that $\mr{C}_{\veclambda}$ is sufficiently large so that
	\begin{align}
	\label{ine:p:main-03}
		\mr{C}_{\veclambda}/2-r_{\mr{l},\lambda } -r_{\mr{u},\lambda} >0.
	\end{align}
	Then, the optimal solution $\hat{\vecbeta}$ of \eqref{ine:l1-penalized-Huber} satisfies
	\begin{align}
		\|\vecbeta^*-\hat{\vecbeta}\|_\Sigma &\leq  r_\Sigma(\hat{\varsigma}_{\mr{cur}}),\nonumber\\
		\|\vecbeta^*-\hat{\vecbeta}\|_{\veclambda} &\leq  \frac{r_{\mr{u},\Sigma}+ \mr{C}_{\veclambda}\sqrt{\frac{\sum_{i=1}^s \lambda_i^2}{\kappa}}}{\mr{C}_{\veclambda}/2-r_{\mr{u},\lambda}}\times \|\vecbeta^*-\hat{\vecbeta}\|_\Sigma\leq  \frac{r_{\mr{u},\Sigma}+ \mr{C}_{\veclambda}\sqrt{\frac{\sum_{i=1}^s \lambda_i^2}{\kappa}}}{\mr{C}_{\veclambda}/2-r_{\mr{u},\lambda}}\times r_\Sigma(\hat{\varsigma}_{\mr{cur}}). 
	\end{align}
	\end{proposition}

	\begin{proposition}
		\label{p:ub}
		Suppose that Assumption~\ref{a:cov} holds, and that $\{\xi_i\}_{i=1}^n$ is an i.i.d.\ sequence of random variables such that $\mathbb{E}h(\xi_i / A) \vecx_i = \mathbf{0}$ for all $i = 1, \ldots, n$ and for any fixed constant $A$. Then, for any $\tau > 0$ and any $\vecv \in \mathbb{R}^d$, we have

		\begin{align}
			\label{p:ub-01}
			\mbb{P}\left(\left|\frac{1}{n}\sum_{i=1}^{n} \mr{C}_{\mr{H}} \tau  h\left(\frac{\xi_i}{\mr{C}_{\mr{H}} \tau}\right)  \vecx_i^\top  \vecv \right| \leq \mr{c}_1L\mr{C}_{\mr{H}}\tau\left(\rho\|\vecv\|_{\veclambda}+ \sqrt{\frac{\log(1/\delta)}{n}}\|\vecv\|_\Sigma\right)\right)\geq 1-\delta.
		\end{align} 
	\end{proposition}
	\begin{proposition}
		\label{p:sc}
		Suppose that Assumption~\ref{a:cov} holds, and that $\{\xi_i\}_{i=1}^n$ is an i.i.d.\ sequence of random variables such that $\mathbb{E}|\xi_i| = \varsigma$ for $i = 1, \ldots, n$, and that $\{\xi_i\}_{i=1}^n$ is independent of $\{\vecx_i\}_{i=1}^n$. Assume that $\mr{C}_{\mathrm{H}} \geq 72 L^4$. For any $\mr{r}^2 \leq \frac{1}{48 L^4}$, define the set
		\begin{align}
			\mr{V}_{\mr{sc}}:=\left\{\vecv \in \mbb{R}^{d}, \tau \in \mbb{R}_+\,\,\mid\,\,\|\vecv\|_\Sigma = \mr{C}_{\mr{H}}\tau \mr{r},\,\varsigma \leq \tau\right\}.
		\end{align}
		Assume further that $\mr{r} \geq \sqrt{ \frac{\log(1/\delta)}{n} }$. Then, for any $(\vecv, \tau) \in \mathcal{V}_{\mathrm{sc}}$, with probability at least $1 - \delta$, the following inequality holds:
		\begin{align}
			\label{p:sc-01}
			\frac{1}{n} \sum_{i=1}^{n} \mr{C}_{\mr{H}} \tau\left(-h\left(\frac{\xi_i}{\mr{C}_{\mr{H}} \tau}-\frac{\vecv^\top \vecx_i}{\mr{C}_{\mr{H}} \tau}\right)+h\left(\frac{\xi_i}{\mr{C}_{\mr{H}} \tau}\right)  \right)\vecv^\top \vecx_i\geq \frac{\|\vecv\|_\Sigma^2}{3 }- \mr{c}_2L\mr{C}_{\mr{H}}\tau \left(\rho\|\vecv\|_{\veclambda}-\sqrt{\frac{\log(1/\delta)}{n}}\|\vecv\|_\Sigma\right).
		\end{align}
	\end{proposition}
	\begin{proposition}
		\label{p:oq}
		Suppose that Assumption~\ref{a:cov} holds. Then, for any $\vecv \in \mathbb{R}^d$, with probability at least $1 - \delta$, we have
		\begin{align}
			\label{p:oq-01}
			\sqrt{\frac{o}{n}}\sqrt{\frac{1}{n}\left| \sum_{i\in \mc{O}} (\vecv^\top \vecx_i)^2-\mbb{E}(\vecv^\top \vecx_i)^2\right|}\leq \mr{c}_3L\left(\rho\|\vecv\|_{\veclambda}+\frac{o}{n}\sqrt{\log\frac{n}{o}}\|\vecv\|_\Sigma+\sqrt{\frac{\log(1/\delta)}{n}}\|\vecv\|_\Sigma\right).
		\end{align}
	\end{proposition}
	\begin{proposition}
	\label{p:main-abs}
	Define 
	\begin{align}
		\tilde{\mr{V}}_{\mr{abs}}:=\left\{\vecv \in \mbb{R}^{d}, \eta \in \mbb{R}_+\,\,\mid\,\,\|\vecv\|_\Sigma/\eta = 1,\,\|\vecv\|_{\veclambda}/\eta \leq \mr{R}_{\veclambda} \right\}
	\end{align}
	Let $\mr{B}$ and $N$ be integers such that $n = N \mr{B}$ and $\sqrt{\mr{B}} \geq 48$. For each $k = 1, \ldots, \mr{B}$, define
	\begin{align}
		\tilde{\varsigma}_k = \frac{1}{N}\sum_{i=(k-1)N+1}^{kN} |\xi_i + \vecx_i^\top \vecv|.
	\end{align}
	Assume the following conditions hold:$\mr{C}_{\mr{B}}\geq 2\times 24^2$, 
	\begin{align}
		L\left(\mr{R}_{\veclambda}+\sqrt{\frac{s\log(ed/s)}{n}}\right)&\leq \frac{1}{64\sqrt{e} \mr{C}_{\mr{B}}\mr{c}_\xi},\,\, \frac{\log(1/\delta)}{n}\leq \frac{1}{\mr{c}_\xi \mr{C}^2_{\mr{B}}},\text{ and }\mr{B}=\frac{n}{\mr{c}_\xi^2 \mr{C}_{\mr{B}}}.
	\end{align}
	Then, for any $(\vecv, \eta) \in \tilde{\mr{V}}_{\mathrm{abs}}$, we have with probability at least $1 - \delta$:
	\begin{align}
		\mbb{P}\left(\sum_{k=1}^{\mr{B}}\mathbb{I}_{\left|\frac{\tilde{\varsigma}_k-\mbb{E}\frac{1}{n}\sum_{i=1}^n|\xi_i + \vecx_i^\top \vecv|}{24(\sigma^*+\eta)/\sqrt{N}} \geq 1\right|}\leq \frac{11}{24}\mr{B}\right)\geq 1-\delta.
	\end{align}
	\end{proposition}

	\begin{corollary}
	\label{c:main-mean}
	As in the preceding proposition, define
	\begin{align}
		\tilde{\mr{V}}_{\mr{abs}}:=\left\{\vecv \in \mbb{R}^{d}, \eta \in \mbb{R}_+\,\,\mid\,\,\|\vecv\|_\Sigma/\eta = 1,\,\|\vecv\|_{\veclambda}/\eta \leq \mr{R}_{\veclambda} \right\}
	\end{align}
	Let $\mr{B}$ and $N$ be integers such that $n = N \mr{B}$ and $\sqrt{\mr{B}} \geq 6$. For each $k = 1, \ldots, \mr{B}$, define
	\begin{align}
		\tilde{\mu}_k = \frac{1}{N}\sum_{i=(k-1)N+1}^{kN} (\xi_i + \vecx_i^\top \vecv).
	\end{align}
	Assume the following conditions:
	\begin{align}
		L\left(\mr{R}_{\veclambda}+\sqrt{\frac{s\log(ed/s)}{n}}\right)\leq \frac{1}{8},\,\,\frac{\log(1/\delta)}{n}\leq \frac{1}{4\times 128},\text{ and }\mr{B}\geq 128 \log(1/\delta).
	\end{align}
	Then, for any $(\vecv, \eta) \in \tilde{\mr{V}}_{\mathrm{abs}}$, with probability at least $1 - \delta$, we have
	\begin{align}
			\mbb{P}\left(\sum_{k=1}^{\mr{B}}\mathbb{I}_{\left|\frac{\tilde{\varsigma}_k-\mbb{E}\frac{1}{n}\sum_{i=1}^n|\xi_i + \vecx_i^\top \vecv|}{48(\sigma^*+\eta)/\sqrt{N}} \geq 1\right|}\leq \frac{11}{24}\mr{B}\right)\geq 1-\delta.
	\end{align}
	\end{corollary}

\section{Proofs of Theorems~\ref{t:Huber}, \ref{t:abs}, \ref{t:Huberabs}, \ref{t:intercept} and~\ref{t:intercept2}}
\label{sec:PT}
In Section~\ref{sec:PT}, we provide the proofs of Theorems~\ref{t:Huber}, \ref{t:abs}, \ref{t:Huberabs}, \ref{t:intercept} and~\ref{t:intercept2}.

	\subsection{Proof of Theorem \ref{t:Huber}}
	\label{sec:PT1}
	To prove Theorem~\ref{t:Huber}, it suffices to verify that inequalities~\eqref{ine:p:main-01}, \eqref{ine:p:main-02}, and~\eqref{ine:p:main-03} hold under the stated assumptions. From the assumptions, Propositions~\ref{p:ub}, \ref{p:sc}, and~\ref{p:oq} hold with probability at least $1 - 3\delta$. In the remainder of this subsection, we work conditionally on the event that all three propositions hold.

	\noindent
	\underline{Verification of \eqref{ine:p:main-01}}

	We have
	\begin{align}
		&\left| \frac{1}{n}\sum_{i=1}^n \mr{C}_{\mr{H}} \hat{\varsigma}_{\mr{cur}} h\left(\frac{\xi_i+\sqrt{n}\vectheta_i}{\mr{C}_{\mr{H}} \hat{\varsigma}_{\mr{cur}} }\right) \vecx_i^\top \vecv\right|\nonumber\\
		&\stackrel{(a)}{\leq} \left| \frac{1}{n}\sum_{i=1}^n \mr{C}_{\mr{H}} \hat{\varsigma}_{\mr{cur}} h\left(\frac{\xi_i}{\mr{C}_{\mr{H}} \hat{\varsigma}_{\mr{cur}} }\right) \vecx_i^\top \vecv\right|+\left| \frac{1}{n}\sum_{i\in\mc{O}} \mr{C}_{\mr{H}} \hat{\varsigma}_{\mr{cur}} h\left(\frac{\xi_i+\sqrt{n}\vectheta_i}{\mr{C}_{\mr{H}} \hat{\varsigma}_{\mr{cur}} }\right) \vecx_i^\top \vecv\right|+\left| \frac{1}{n}\sum_{i\in \mc{O}} \mr{C}_{\mr{H}} \hat{\varsigma}_{\mr{cur}} h\left(\frac{\xi_i}{\mr{C}_{\mr{H}} \hat{\varsigma}_{\mr{cur}} }\right) \vecx_i^\top \vecv\right|\nonumber\\
		&\stackrel{(b)}{\leq} \left| \frac{1}{n}\sum_{i=1}^n \mr{C}_{\mr{H}} \hat{\varsigma}_{\mr{cur}} h\left(\frac{\xi_i}{\mr{C}_{\mr{H}} \hat{\varsigma}_{\mr{cur}} }\right) \vecx_i^\top \vecv\right|+2\mr{C}_{\mr{H}} \hat{\varsigma}_{\mr{cur}} \sqrt{\frac{o}{n}}\sqrt{\frac{1}{n}\sum_{i \in \mc{O}} (\vecx_i^\top \vecv)^2-\mbb{E}(\vecx_i^\top \vecv)}+2\mr{C}_{\mr{H}} \hat{\varsigma}_{\mr{cur}} \frac{o}{n}\|\vecv\|_\Sigma\nonumber\\
		&\stackrel{(c)}{\leq} \mr{c}_1L\mr{C}_{\mr{H}}\hat{\varsigma}_{\mr{cur}}\left(\rho\|\vecv\|_{\veclambda}+ \sqrt{\frac{\log(1/\delta)}{n}}\|\vecv\|_\Sigma\right)+2\mr{c}_3L\mr{C}_{\mr{H}} \hat{\varsigma}_{\mr{cur}} \left(\rho\|\vecv\|_{\veclambda}+\frac{o}{n}\sqrt{\log\frac{n}{o}}\|\vecv\|_\Sigma+\sqrt{\frac{\log(1/\delta)}{n}}\|\vecv\|_\Sigma\right)\nonumber\\
		&\quad \quad +2\mr{C}_{\mr{H}} \hat{\varsigma}_{\mr{cur}} \frac{o}{n}\|\vecv\|_\Sigma\nonumber\\
		&\stackrel{(d)}{\leq} \mr{c}_1L\mr{C}_{\mr{H}}\hat{\varsigma}_{\mr{cur}}\left(\rho\|\vecv\|_{\veclambda}+ \sqrt{\frac{\log(1/\delta)}{n}}\|\vecv\|_\Sigma\right)+4\mr{c}_3L\mr{C}_{\mr{H}} \hat{\varsigma}_{\mr{cur}} \left(\rho\|\vecv\|_{\veclambda}+\frac{o}{n}\sqrt{\log\frac{n}{o}}\|\vecv\|_\Sigma+\sqrt{\frac{\log(1/\delta)}{n}}\|\vecv\|_\Sigma\right),
	\end{align}
	where (a) follows from triangular inequality, (b) follows from H{\"o}lder's inequality,  (c) follows from Propositions \ref{p:ub} and \ref{p:oq}, and (d) follows from  $\log \frac{n}{o}, L,\mr{c}_3 \geq 1$.
	Then, inequality~\eqref{ine:p:main-01} holds with:
	\begin{align}
		r_{\mr{u},\lambda}=(\mr{c}_1+4\mr{c}_3)L\rho ,\quad r_{\mr{u},\Sigma} = (\mr{c}_1+4\mr{c}_3)L \left(\frac{o}{n}\sqrt{\log\frac{n}{o}}+\sqrt{\frac{\log(1/\delta)}{n}}\right).
	\end{align}

	\noindent
	\underline{Verification of \eqref{ine:p:main-02}}

	As in the previous case, combining Propositions~\ref{p:ub} and~\ref{p:sc}, we obtain:
	\begin{align}
		&\frac{1}{n}\sum_{i=1}^n \mr{C}_{\mr{H}} \hat{\varsigma}_{\mr{cur}} \left(-h\left(\frac{\xi_i- \vecx_i^\top\vecv }{\mr{C}_{\mr{H}} \hat{\varsigma}_{\mr{cur}}}\right) +h\left(\frac{\xi_i}{\mr{C}_{\mr{H}} \hat{\varsigma}_{\mr{cur}}}\right)\right) \vecx_i^\top \vecv\nonumber\\
		&\geq \frac{\|\vecv\|_\Sigma^2}{3}- \mr{c}_2L\mr{C}_{\mr{H}}\hat{\varsigma}_{\mr{cur}} \left(\rho\|\vecv\|_{\veclambda}- \sqrt{\frac{\log(1/\delta)}{n}}\|\vecv\|_\Sigma\right)-8\mr{c}_3L\mr{C}_{\mr{H}} \hat{\varsigma}_{\mr{cur}} \left(\rho\|\vecv\|_{\veclambda}+\frac{o}{n}\sqrt{\log\frac{n}{o}}\|\vecv\|_\Sigma+\sqrt{\frac{\log(1/\delta)}{n}}\|\vecv\|_\Sigma\right).
	\end{align}
	Then, inequality~\eqref{ine:p:main-02} holds with:
	\begin{align}
		r_{\mr{l},\lambda}=(\mr{c}_2+8\mr{c}_3)L\rho ,\quad r_{\mr{l},\Sigma} = (\mr{c}_2+8\mr{c}_3)L \left(\frac{o}{n}\sqrt{\log\frac{n}{o}}+\sqrt{\frac{\log(1/\delta)}{n}}\right).
	\end{align}

	\noindent
	\underline{Verification of \eqref{ine:p:main-03}}

This inequality follows directly from the assumption that $\mr{C}_{\veclambda}$ is sufficiently large so that $\mr{C}_{\veclambda} / 2 - r_{\mathrm{l},\lambda} - r_{\mathrm{u},\lambda} > 0$.

	From the discussion above, we see that \eqref{ine:p:main-03} holds from the assumption on $\mr{C}_{\veclambda}$.

	\noindent
	\underline{Conclusion}

	Finally, noting that $\sqrt{ \sum_{i=1}^s \lambda_i } \leq \sqrt{ \log(d/s)/n }$ and combining the bounds above, the proof is complete.

	\subsection{Proof of Theorem \ref{t:abs}}
	\label{sec:PT2}	
	We begin by verifying the concentration property of $\hat{\varsigma}$.
	Define
	\begin{align}
		\mc{O}_\mr{B}&= \{k=1,\cdots,\mr{B}\,\mid\, \mc{O} \cap \{(k-1)\mr{N},(k-1)\mr{N}+1,\cdots,k\mr{N}\} \neq \emptyset\},\\
		\Delta & = \hat{\vecbeta}-\vecbeta^*.
	\end{align} 
	Then, for any positive scalar $\eta$ and any vector $\vecv \in \mathbb{R}^d$ such that $\|\Delta\|_\Sigma / \eta = 1$ and $\|\vecv\|_\Sigma / \eta \leq \mr{R}_{\veclambda}$, we have:
	\begin{align}
		&\sum_{k=1}^{\mr{B}}\mathbb{I}_{\frac{|\hat{\varsigma}_k-\mbb{E}\frac{1}{n}\sum_{i=1}^n|\xi_i+\vecx_i^\top \vecv||}{24(\sigma^*+\eta)/\sqrt{N}} \geq 1}\nonumber\\
		&\stackrel{(a)}{=}  \sum_{k=1}^{\mr{B}}\mathbb{I}_{\frac{|\tilde{\varsigma}_k-\mbb{E}\frac{1}{n}\sum_{i=1}^n|\xi_i+\vecx_i^\top \vecv||}{24(\sigma^*+\eta)/\sqrt{N}} \geq 1}-\sum_{k\in\mc{O}_{\mr{B}}}\mathbb{I}_{\frac{|\hat{\varsigma}_k-\mbb{E}\frac{1}{n}\sum_{i=1}^n|\xi_i+\vecx_i^\top \vecv||}{24(\sigma^*+\eta)/\sqrt{N}} \geq 1}+\sum_{k\in\mc{O}_{\mr{B}}}\mathbb{I}_{\frac{|\tilde{\varsigma}_k-\mbb{E}\frac{1}{n}\sum_{i=1}^n|\xi_i+\vecx_i^\top \vecv||}{24(\sigma^*+\eta)/\sqrt{N}} \geq 1}\nonumber\\
		&\leq \sum_{k=1}^{\mr{B}}\mathbb{I}_{\frac{|\tilde{\varsigma}_k-\mbb{E}\frac{1}{n}\sum_{i=1}^n|\xi_i+\vecx_i^\top \vecv||}{24(\sigma^*+\eta)/\sqrt{N}} \geq 1}+\sum_{k\in\mc{O}_{\mr{B}}}\mathbb{I}_{\frac{|\tilde{\varsigma}_k-\mbb{E}\frac{1}{n}\sum_{i=1}^n|\xi_i+\vecx_i^\top \vecv||}{24(\sigma^*+\eta)/\sqrt{N}} \geq 1}\nonumber\\
		&\stackrel{(b)}{\leq} \sum_{k=1}^{\mr{B}}\mathbb{I}_{\frac{|\tilde{\varsigma}_k-\mbb{E}\frac{1}{n}\sum_{i=1}^n|\xi_i+\vecx_i^\top \vecv||}{24(\sigma^*+\eta)/\sqrt{N}} \geq 1}+o\nonumber\\
		&\stackrel{(c)}{\leq}    \frac{11}{24}\mr{B}+\frac{1}{\mr{C}_{\mr{B}}}\mr{B}\nonumber\\
		&\stackrel{(d)}{\leq} \frac{1}{2}\mr{B},
	\end{align}
	where (a) follows from the definition of $\mc{O}_{\mr{B}}$, (b) follows from the fact the upper bound of the number of elements of $\mc{O}_{\mr{B}}$ is $o$, (c) follows from  Proposition \ref{p:main-abs} and the assumption $o\leq \frac{n}{\mr{c}_\xi^2\mr{C}^2_{\mr{B}}}$ and  $\mr{B} = \frac{n}{\mr{c}_\xi^2 \mr{C}_{\mr{B}}}$, and (d) follows from $\mr{C}_{\mr{B}}\geq 24$.

	Setting $\eta = \|\Delta\|_\Sigma$, we obtain
	\begin{align}
		\label{ine:varsigmacer}
		\left|\hat{\varsigma}-\mbb{E}\frac{1}{n}\sum_{i=1}^n|\xi_i+\vecx_i^\top \Delta|\right| \leq 24(\mr{c}_\xi\varsigma^*+\|\Delta\|_\Sigma)\frac{1}{\mr{c}_{\xi}\sqrt{\mr{C}_{\mr{B}}}},
	\end{align}
	using the noise assumption $\sigma^*\leq \mr{c}_{\xi}\varsigma^*$.

	\noindent
	\underline{Case $\|\Delta\|_\Sigma \geq 20 L^4\varsigma^* $}

	\textit{Lower bound.}
	Using the triangle inequality:
	\begin{align}
		\mbb{E}\frac{1}{n}\sum_{i=1}^n|\xi_i+\vecx_i^\top \Delta|\geq\mbb{E}\frac{1}{n}\sum_{i=1}^n(|\vecx_i^\top \Delta|-|\xi_i|)=\mbb{E}\frac{1}{n}\sum_{i=1}^n|\vecx_i^\top \Delta|-\varsigma.
	\end{align}
	From Lemma \ref{l:no-rad-exp},  with probability at least $1-\delta$, for any $\vecv \in \mbb{R}^d$ we have
	\begin{align}
		\mbb{E}\sum_{i=1}^n \frac{1}{n}|\vecx_i^\top \vecv|\geq -CL\rho \|\vecv\|_{\lambda}-CL\sqrt{\frac{\log (1/\delta)}{n}}\|\vecv\|_\Sigma+\frac{1}{72L^4}\|\vecv\|_\Sigma,
	\end{align}
	Then,  we obtain:
	 \begin{align}
		\mbb{E}\sum_{i=1}^n \frac{1}{n}|\vecx_i^\top \Delta|&\geq  -CL\rho \|\Delta\|_{\lambda}-CL\sqrt{\frac{\log (1/\delta)}{n}}\|\Delta\|_\Sigma+\frac{1}{72L^4}\|\Delta\|_\Sigma\nonumber\\
		&\stackrel{(a)}{\geq}  -CL\rho \mr{R}_{\veclambda}\|\Delta\|_\Sigma-CL\sqrt{\frac{\log (1/\delta)}{n}}\|\Delta\|_\Sigma+\frac{1}{72L^4}\|\Delta\|_\Sigma\nonumber\\
		&\stackrel{(b)}{\geq} \frac{1}{80L^4}\|\Delta\|_\Sigma,
	 \end{align}
	 where (a) follows from  $\|\Delta\|_{\lambda}\leq \mr{R}_{\veclambda}\|\Delta\|_\Sigma$, and (b) follows from the assumption $L\left(\rho \mr{R}_{\veclambda}+\sqrt{\log(1/\delta)/n}\right)$ is sufficiently small. 
		Combining with~\eqref{ine:varsigmacer}, we have
		\begin{align}
			\label{ine:varsigmacer2}
			\left|\hat{\varsigma}-\mbb{E}\frac{1}{n}\sum_{i=1}^n|\xi_i+\vecx_i^\top \Delta|\right|& \leq 24(\mr{c}_\xi\varsigma^*+\|\Delta\|_\Sigma)\frac{1}{\mr{c}_{\xi}\sqrt{\mr{C}_{\mr{B}}}}\nonumber \\
			&\leq 24\left(\frac{\mr{c}_\xi}{20L^4}+1 \right)\frac{1}{\mr{c}_{\xi}\sqrt{\mr{C}_{\mr{B}}}}\|\Delta\|_\Sigma\nonumber \\
			&\leq \frac{1}{L^4 160}\|\Delta\|_\Sigma,
		\end{align}
		where we use the assumption $\left(\frac{24\times 160}{\mr{c}_{\xi}}\left(\frac{\mr{c}_\xi}{20}+L^4 \right)\right)^2\leq \mr{C}_{\mr{B}}$.
		Consequently, 
		\begin{align}
			\hat{\varsigma}& \geq \frac{1}{160L^4}\|\Delta\|_\Sigma\geq \frac{1}{8}\varsigma^*.
		\end{align}

		\textit{Upper bound.} From Lemma \ref{l:L1ev}  with  $\mr{c}=\mr{C}_{\mr{PHR}}\mr{C}_{\mr{H}}\hat{\varsigma}_{\mr{cur}}\mr{R}_\Sigma L\rho \sqrt{n}\mr{R}_{\veclambda} $ and $\mr{c}' = \mr{C}_{\mr{PHR}}\mr{C}_{\mr{H}}\hat{\varsigma}_{\mr{cur}}\mr{R}_\Sigma L$, triangular inequality and the assumption $\sigma^* \leq \mr{c}_\xi \varsigma^*$, we have
		\begin{align}
			\mbb{E}\frac{1}{n}\sum_{i=1}^n|\vecx_i^\top \Delta+\xi_i|& \leq C\frac{\mr{C}_{\mr{PHR}}\mr{C}_{\mr{H}}\hat{\varsigma}_{\mr{cur}}\mr{R}_\Sigma L(\rho \sqrt{n}\mr{R}_{\veclambda} +1)+\mr{c}_\xi \varsigma^*}{n}+\mr{C}_{\mr{PHR}}\mr{C}_{\mr{H}}\hat{\varsigma}_{\mr{cur}}\mr{R}_\Sigma+\varsigma^*\\
			&\leq 2 \max\{ \mr{C}_{\mr{PHR}}\mr{C}_{\mr{H}}\hat{\varsigma}_{\mr{cur}}\mr{R}_\Sigma , \varsigma^*\}.
		\end{align}
		From \eqref{ine:varsigmacer2} and \eqref{vabs}, we have
		\begin{align}
			\left|\hat{\varsigma}-\mbb{E}\frac{1}{n}\sum_{i=1}^n|\xi_i+\vecx_i^\top \Delta|\right|&\leq \frac{1}{L^4 160} \mr{C}_{\mr{PHR}}\mr{C}_{\mr{H}}\hat{\varsigma}_{\mr{cur}}\mr{R}_\Sigma.
		\end{align}
	Consequently, from $L\geq 1$, we have
	\begin{align}
		\hat{\varsigma}& \leq 3 \max\{ \mr{C}_{\mr{PHR}}\mr{C}_{\mr{H}}\hat{\varsigma}_{\mr{cur}}\mr{R}_\Sigma , \varsigma^*\}.
	\end{align}

	\noindent
	\underline{Case $\|\Delta\|_\Sigma \leq 20L^4\varsigma^* $}

	\textit{Lower bound.}
	From Lemma \ref{l:L1ev} with  $\mr{c}=\mr{R}_{\veclambda} 20L^4 \varsigma^*$ and $\mr{c}' = 20L^4 \varsigma^*$, and the assumption $\sigma\leq \mr{c}_\xi \varsigma$, we have
	\begin{align}
		\mbb{E}\frac{1}{n}\sum_{i=1}^n|\vecx_i^\top \Delta+\xi_i|& \geq -C\frac{L(\rho \sqrt{n}\mr{R}_{\veclambda} 20L^4 \varsigma^*+20L^4 \varsigma^*)+\mr{c}_\xi \varsigma^*}{n}+\inf_{\vecv  \in \mr{c}\mbb{B}^d_\lambda \cap \mr{c}'\mbb{B}^d_\Sigma}\left(\mbb{E}|\vecx^\top \vecv+\xi|\right)\nonumber\\
		&\stackrel{(a)}{\geq} -C\frac{40L^5+\mr{c}_\xi }{\sqrt{n}}\varsigma^*+\frac{\varsigma^*}{2\sqrt{2}}\nonumber\\
		&\stackrel{(b)}{\geq} \frac{\varsigma^*}{4},
	\end{align}
	where (a) follows from Lemma \ref{l:L1+}, $\rho \mr{R}_{\veclambda}\leq 1$ and $L\geq 1$, and (b) follows from  the assumption tha  $ \frac{L^5+\mr{c}_\xi }{\sqrt{n}}$ is sufficiently small. From \eqref{ine:varsigmacer}, we have
	\begin{align}
		\label{ine:varsigmacer3}
		\left|\hat{\varsigma}-\mbb{E}\frac{1}{n}\sum_{i=1}^n|\xi_i+\vecx_i^\top \Delta|\right| \leq 24(\mr{c}_\xi\varsigma^*+\|\Delta\|_\Sigma)\frac{1}{\mr{c}_{\xi}\sqrt{\mr{C}_{\mr{B}}}}\stackrel{(a)}{\leq} 24(\mr{c}_\xi+20L^4)\frac{\varsigma^*}{\mr{c}_{\xi}\sqrt{\mr{C}_{\mr{B}}}}\stackrel{(b)}{\leq} \frac{1}{8}\varsigma^*,
	\end{align}
	where (a) follows from $\|\Delta\|_\Sigma \leq 20L^4\varsigma^* $, and (b) follows from  the assumption $\left(\frac{8\times 24}{\mr{c}_{\xi}} (\mr{c}_\xi+20L^4)\right)^2\leq \mr{C}_{\mr{B}}$. Consequently, we have
	\begin{align}
		\hat{\varsigma}& \geq \frac{1}{8}\varsigma^*.
	\end{align}

	\textit{Upper bound.} 
	From Lemma \ref{l:L1ev} with  $\mr{c}=\mr{R}_{\veclambda} 20L^4 \varsigma^*$ and $\mr{c}' = 20L^4 \varsigma^*$, triangular inequality and the assumption $\sigma \leq \mr{c}_\xi \varsigma^*$, we have
	\begin{align}
		\mbb{E}\frac{1}{n}\sum_{i=1}^n|\vecx_i^\top \Delta+\xi_i|& \leq C\frac{20L^4 (\rho \sqrt{n}\mr{R}_{\veclambda} +1)+\mr{c}_\xi }{n}\varsigma^*+20L^4\varsigma^* + \varsigma^*\stackrel{(a)}{\leq} 21 L^4 \varsigma^*,
	\end{align}
	where (a) follows from  the assumption $\rho \mr{R}_{\veclambda}\leq 1$, $L\geq 1$ and  $\frac{L^4+\mr{c}_\xi}{\sqrt{n}}$ is sufficiently small.
	Consequently, combining this and \eqref{ine:varsigmacer3}
	\begin{align}
		\hat{\varsigma}\leq 22L^4\varsigma^*.
	\end{align}

	\noindent
	\underline{Conclusion}
	Combining both  cases, we have
	\begin{align}
		\frac{1}{8}\varsigma^* \leq \hat{\varsigma} \leq \max\{ 3\mr{C}_{\mr{PHR}}\mr{C}_{\mr{H}}\hat{\varsigma}_{\mr{cur}}\mr{R}_\Sigma , 25L^4 \varsigma^*\}.
	\end{align}

	\subsection{Proof of Theorem \ref{t:Huberabs}}
	\label{sec:PT3}

	From the assumptions, we conclude that the results in Propositions~\ref{p:ub}--\ref{p:oq-01} hold with probability at least $1 - 4\delta$. In the remainder of the proof, we assume these propositions hold and proceed to analyze the outputs at each iteration.

	\subsubsection{Analysis of the first iteration}
	\noindent
	\underline{Estimation of $\hat{\vecbeta}_1$}

	By the assumption on $\mr{C}_{\mathrm{ini}}$, inequality~\eqref{ine:t:Huber:abs} holds. Then, by Theorem~\ref{t:Huber}, we obtain:
	\begin{align}
		\label{ine:FI}
			\|\vecbeta^*-\hat{\vecbeta}_1\|_\Sigma \leq  \mr{C}_{\mr{PHR}}\mr{C}_{\mr{H}} \mr{C}_{\mr{ini}}\mr{R}_\Sigma\text{ and }\|\vecbeta^*-\hat{\vecbeta}_1\|_{\veclambda} \leq  \|\vecbeta^*-\hat{\vecbeta}_1\|_\Sigma\times \mr{R}_{\veclambda},
	\end{align}

	\noindent
	\underline{Estimation of $\hat{\varsigma}_1$}

	If $\max \left\{3\mr{C}_{\mr{PHR}}\mr{C}_{\mr{H}}\mr{C}_{\mr{ini}}\mr{R}_\Sigma,25 L^4 \varsigma^*\right\}=25L^4 \varsigma^*$,
	then, from Theorem \ref{t:abs},  we have 
	\begin{align}
	\label{ine:FI-2}
		\varsigma^*\leq \hat{\varsigma}_1 \leq  8\times 25L^4 \varsigma^*.
	\end{align}
	Otherwise, from Theorem \ref{t:abs}, we have 
	\begin{align}
	\label{ine:FI-3}
		\varsigma^* \leq 	\hat{\varsigma}_1 \leq  8\times 3  \mr{C}_{\mr{PHR}}\mr{C}_{\mr{H}} \mr{C}_{\mr{ini}}\mr{R}_\Sigma \leq  24  \mr{C}_{\mr{PHR}}\mr{C}_{\mr{H}} \mr{C}_{\mr{ini}}\mr{R}_\Sigma:=\mr{C}_{\mr{ini}}\hat{\varsigma}_{\mr{ini}}^{\mr{u}},
	\end{align}
	We note that from the assumption $\mr{C}_{\mr{PHR}}\mr{C}_{\mr{H}} \mr{R}_\Sigma< 1/24$, and therefore, we have $\hat{\varsigma}_{\mr{ini}}^{\mr{u}}<1$.

	Consequently, we have
	\begin{align}
	\label{ine:FI-4}
		\varsigma^* \leq 	\hat{\varsigma}_1 \leq  \max\{200L^4 \varsigma^*, \mr{C}_{\mr{ini}}\hat{\varsigma}_{\mr{ini}}^{\mr{u}}\}.
	\end{align}

	\subsubsection{Analysis the second iteration}

	\noindent
	\underline{Estimation of $\hat{\vecbeta}_2$}

	When $\max \left\{3\mr{C}_{\mr{PHR}}\mr{C}_{\mr{H}}\mr{C}_{\mr{ini}}\mr{R}_\Sigma,25L^4 \varsigma^*\right\}= 25 L^4 \varsigma^*$ in the previous iteration, from \eqref{ine:FI-2} and Theorem \ref{t:Huber}, we have
	\begin{align}
		\label{ine:correctbound}
		\|\vecbeta^*-\hat{\vecbeta}_2\|_\Sigma \leq   \mr{C}_{\mr{PHR}}\mr{C}_{\mr{H}}200L^4 \varsigma^* \mr{R}_\Sigma\text{ and }\|\vecbeta^*-\hat{\vecbeta}_2\|_{\veclambda} \leq  \|\vecbeta^*-\hat{\vecbeta}_2\|_\Sigma \times \mr{R}_{\veclambda},
\end{align}
Otherwise, from \eqref{ine:FI-3} and  Theorem \ref{t:Huber}, we have
	\begin{align}
		\label{ine:SE}
			\|\vecbeta^*-\hat{\vecbeta}_2\|_\Sigma \leq  \mr{C}_{\mr{PHR}}\mr{C}_{\mr{H}} \hat{\varsigma}_{\mr{ini}}^{\mr{u}}\mr{C}_{\mr{ini}}\mr{R}_\Sigma\text{ and }\|\vecbeta^*-\hat{\vecbeta}_2\|_{\veclambda} \leq   \|\vecbeta^*-\hat{\vecbeta}_2\|_\Sigma \times \mr{R}_{\veclambda}.
	\end{align}
	Consequently, we have
	\begin{align}
			\|\vecbeta^*-\hat{\vecbeta}_2\|_\Sigma \leq  \mr{C}_{\mr{PHR}}\mr{C}_{\mr{H}}\mr{R}_\Sigma\max\{ 200L^4 \varsigma^*,\mr{C}_{\mr{ini}} \hat{\varsigma}_{\mr{ini}}^{\mr{u}}\}\text{ and }\|\vecbeta^*-\hat{\vecbeta}_2\|_{\veclambda} \leq   \|\vecbeta^*-\hat{\vecbeta}_2\|_\Sigma \times \mr{R}_{\veclambda}.
	\end{align}

\noindent
\underline{Estimation of $\hat{\varsigma}_2$}

From \eqref{ine:FI-4} and Theorem  \ref{t:abs}, we have
\begin{align}
	\varsigma^* \leq \hat{\varsigma}_2& \leq 8\max\{ 3\mr{C}_{\mr{PHR}}\mr{C}_{\mr{H}}\max\{200L^4 \varsigma^*, \mr{C}_{\mr{ini}}\hat{\varsigma}_{\mr{ini}}^{\mr{u}}\}\mr{R}_\Sigma , 25 L^4 \varsigma^*\}\nonumber\\
	&= 8\max\{ 600\mr{C}_{\mr{PHR}}\mr{C}_{\mr{H}}L^4 \varsigma^*\mr{R}_\Sigma,3\mr{C}_{\mr{PHR}}\mr{C}_{\mr{H}}\mr{C}_{\mr{ini}}\hat{\varsigma}_{\mr{ini}}^{\mr{u}}\mr{R}_\Sigma , 25 L^4 \varsigma^*\}\nonumber \\
	&\leq  \max\{200L^4 \varsigma^*,   \mr{C}_{\mr{ini}}(\hat{\varsigma}_{\mr{ini}}^{\mr{u}})^2     \},
\end{align}
where the last line follows from the assumption $\mr{C}_{\mr{PHR}}\mr{C}_{\mr{H}} \mr{R}_\Sigma< 1/24$ and the definition of $\hat{\varsigma}_{\mr{ini}}^{\mr{u}}$.

\subsubsection{Analysis the $k(\geq 2)$-th  iteration}

We now state and prove the following lemma.
\begin{lemma} 
	\label{l:induction}
	In the $k$-th iteration, we have
	\begin{align}
		\|\vecbeta^*-\hat{\vecbeta}_k\|_\Sigma&\leq  \mr{C}_{\mr{PHR}}\mr{C}_{\mr{H}}\mr{R}_\Sigma\max\{ 200L^4 \varsigma^*,\mr{C}_{\mr{ini}} (\hat{\varsigma}_{\mr{ini}}^{\mr{u}})^{k-1}\}\text{ and }\|\vecbeta^*-\hat{\vecbeta}_k\|_{\veclambda} \leq   \|\vecbeta^*-\hat{\vecbeta}_k\|_\Sigma \times \mr{R}_{\veclambda}\nonumber\\
		\varsigma^* \leq \hat{\varsigma}_{k}&\leq  \max\{200L^4 \varsigma^*,   \mr{C}_{\mr{ini}}(\hat{\varsigma}_{\mr{ini}}^{\mr{u}})^{k}  \}.
	\end{align}
\end{lemma}
\begin{proof}
	The proof proceeds by induction. Assume that in $k$-th iteration, we have
	\begin{align}
		\|\vecbeta^*-\hat{\vecbeta}_k\|_\Sigma&\leq  \mr{C}_{\mr{PHR}}\mr{C}_{\mr{H}}\mr{R}_\Sigma\max\{ 200L^4 \varsigma^*,\mr{C}_{\mr{ini}} (\hat{\varsigma}_{\mr{ini}}^{\mr{u}})^{k-1}\}\text{ and }\|\vecbeta^*-\hat{\vecbeta}_k\|_{\veclambda} \leq   \|\vecbeta^*-\hat{\vecbeta}_k\|_\Sigma \times \mr{R}_{\veclambda},\nonumber \\
			\varsigma \leq \hat{\varsigma}_{k}&\leq  \max\{200L^4 \varsigma^*,   \mr{C}_{\mr{ini}}(\hat{\varsigma}_{\mr{ini}}^{\mr{u}})^{k}  \}.
	\end{align}
	Then, we consider the $(k+1)$-th iteration.

	\noindent
	\underline{Case I:  $\hat{\varsigma}_{k}\leq 200L^4 \varsigma^*$ in the $k$-th iteration}

	From Theorem \ref{t:Huber}, we have
	\begin{align}
		\|\vecbeta^*-\hat{\vecbeta}_{k+1}\|_\Sigma \leq   \mr{C}_{\mr{PHR}}\mr{C}_{\mr{H}}  \mr{R}_\Sigma 200L^4\varsigma^*\text{ and }\|\vecbeta^*-\hat{\vecbeta}_{k+1}\|_{\veclambda} \leq  \|\vecbeta^*-\hat{\vecbeta}_2\|_\Sigma \times \mr{R}_{\veclambda},
\end{align}
and  Theorem  \ref{t:abs}, we have
\begin{align}
	\varsigma^* \leq \hat{\varsigma}_{k+1} \leq 8 \max\{ 600 L^4 \mr{C}_{\mr{PHR}}\mr{C}_{\mr{H}}\mr{R}_\Sigma \varsigma^*, 25L^4 \varsigma\}\stackrel{(a)}{\leq}   200L^4 \varsigma^*,
\end{align}
where (a) follows from $\mr{C}_{\mr{PHR}}\mr{C}_{\mr{H}} \mr{R}_\Sigma< 1/24$.

\noindent
\underline{Case II:  $\hat{\varsigma}_{k}\leq  \mr{C}_{\mr{ini}}(\hat{\varsigma}_{\mr{ini}}^{\mr{u}})^{k} $ in the $k$-th iteration}

From Theorem \ref{t:Huber}, we have
\begin{align}
	\|\vecbeta^*-\hat{\vecbeta}_{k+1}\|_\Sigma \leq   \mr{C}_{\mr{PHR}}\mr{C}_{\mr{H}}  \mr{R}_\Sigma \mr{C}_{\mr{ini}}(\hat{\varsigma}_{\mr{ini}}^{\mr{u}})^{k} \text{ and }\|\vecbeta^*-\hat{\vecbeta}_{k+1}\|_{\veclambda} \leq  \|\vecbeta^*-\hat{\vecbeta}_2\|_\Sigma \times \mr{R}_{\veclambda},
\end{align}
and  Theorem  \ref{t:abs}, we have
\begin{align}
	\varsigma^* \leq \hat{\varsigma}_{k+1} \leq 8\max\{ 3\mr{C}_{\mr{PHR}}\mr{C}_{\mr{H}}\mr{C}_{\mr{ini}}(\hat{\varsigma}_{\mr{ini}}^{\mr{u}})^{k}\mr{R}_\Sigma , 25L^4 \varsigma^*\}=\max\{ \mr{C}_{\mr{ini}}(\hat{\varsigma}_{\mr{ini}}^{\mr{u}})^{k+1} , 200L^4 \varsigma^*\}.
\end{align}

\end{proof}

\subsubsection{Conclusion}
From Lemma \ref{l:induction}, at $\mr{N}_{\mr{Iter}}$-th iteration, we have 
\begin{align}
	\|\vecbeta^*-\hat{\vecbeta}_{\mr{N}_{\mr{Iter}}}\|_\Sigma&\leq  \mr{C}_{\mr{PHR}}\mr{C}_{\mr{H}}\mr{R}_\Sigma\max\{ 200L^4 \varsigma^*,\mr{C}_{\mr{ini}} (\hat{\varsigma}_{\mr{ini}}^{\mr{u}})^{\mr{N}_{\mr{Iter}}-1}\},\nonumber\\
	\|\vecbeta^*-\hat{\vecbeta}_{\mr{N}_{\mr{Iter}}}\|_{\veclambda} &\leq   \|\vecbeta^*-\hat{\vecbeta}_{\mr{N}_{\mr{Iter}}}\|_\Sigma \times \mr{R}_{\veclambda}.
\end{align}
From the assumption on $\mr{N}_{\mr{Iter}}$, we have 
\begin{align}
	\max\{ 200L^4 \varsigma^*,\mr{C}_{\mr{ini}} (\hat{\varsigma}_{\mr{ini}}^{\mr{u}})^{\mr{N}_{\mr{Iter}}-1}\}=  \max\{ 200L^4 \varsigma^*,\mr{C}_{\mr{ini}} (24  \mr{C}_{\mr{PHR}}\mr{C}_{\mr{H}} \mr{R}_\Sigma)^{\mr{N}_{\mr{Iter}}-1}\}\leq 200L^4\varsigma^*,
\end{align}
and the proof is complete.

\subsection{Proof of Theorem \ref{t:intercept}}
We present the proof in the setting with heavy-tailed noise. The result for sub-Gaussian noise follows similarly by replacing Lemma \ref{l:noiseconcentration} with Lemma \ref{l:noiseconcentration2}. A closely related argument is found in \cite{MinNdaWan2024Robust}; however, our proof is simplified by avoiding the estimation of $\vecbeta^*$, at the expense of requiring more  empirical process analysis.

Following the approach in the proof of Theorem 2.5 in \cite{MinNdaWan2024Robust}, we treat $\xi_1^\sharp, \dots, \xi_{\max\{o, \log(1/\delta)\}}^\sharp$ as outliers. As a result, the vector $\vectheta^*$ is $(o + \max\{o, \log(1/\delta)\})$-sparse. The vector $(\xi_1, \dots, \xi_n)$ is constructed by retaining the $n - \max\{o, \log(1/\delta)\}$ entries with the smallest magnitudes and replacing the others with zero. Let $\mr{C}_{\veciota} = 80$.

To facilitate the union bound, we assume that the following lemmas and inequalities hold with probability at least $1 - 9\delta$: Lemmas \ref{l:gw2}, \ref{l:Che}, \ref{l:noiseconcentration}, and the two bounds from \cite{MinNdaWan2024Robust}, stated below:
\begin{align}
	\label{ine:noiseconcentration1}
	\max_{i\geq \max\{o,\log(1/\delta)\}} \frac{|\xi_i^\sharp|}{\sqrt{n}\iota_i}&\geq \frac{\mr{c}_{\mr{M}}\sigma^*}{20},\\
	\label{ine:noiseconcentration2}
	\frac{\mr{c}_{\mr{M}}^2(\sigma^*)^2n}{10}\leq \sum_{i=\max\{o,\log(1/\delta)\}}|\xi_i^\sharp|^2&\leq 2\mr{c}_{\mr{M}}^2(\sigma^*)^2 n.
\end{align}

Throughout the proof, we define the following quantities:
\begin{align}
	Q(\mu,\vectheta)&:=\sqrt{\frac{1}{2n}\sum_{i=1}^n (y_i-\vecx_i^\top\hat{\vecbeta}-\hat{\mu}-\sqrt{n} \hat{\theta}_i)^2},\nonumber\\
	L(\mu,\vectheta)&:=Q(\mu,\vectheta)+\mr{C}_{\veciota}\|\vectheta\|_{\veciota},\quad \Delta_{\vectheta} := \hat{\vectheta}-\vectheta^*,\quad \Delta_\mu :=\hat{\mu}-\mu^*,\quad \Delta_{\vecbeta} := \hat{\vecbeta}-\vecbeta^*.
\end{align}
Also , define 
\begin{align}
	r_{\Sigma,\lambda}&:= L(r_\Sigma+\rho \mr{R}_{\veclambda}r_\Sigma)\stackrel{(a)}{\leq} 2 Lr_\Sigma,
\end{align}
where (a) follows from the assumption.

\noindent
\underline{Step 1: Bounding  $\|\Delta_{\vectheta}\|_{\veciota}$ via $\|\Delta_{\vectheta}\|_2$ and $\Delta_\mu$}

	From the optimality of $\hat{\mu}$ and $\hat{\vectheta}$, we have
	\begin{align}
		\label{ine:l:intercept-01}
		Q(\hat{\mu},\hat{\vectheta})-Q(\mu^*,\vectheta^*)\leq \mr{C}_{\veciota}\|\vectheta^*\|_{\veciota}-\mr{C}_{\veciota}\|\hat{\vectheta}^*\|_{\veciota}.
	\end{align}
	To bound he right-hand side of \eqref{ine:l:intercept-01}, we apply Lemma A.1 of \cite{BelLecTsy2018Slope}:
	\begin{align}
		\|\vectheta^*\|_{\veciota}-\|\hat{\vectheta}^*\|_{\veciota} \leq 2 \sqrt{\sum_{i=1}^{o+\max\{\log(1/\delta),o\}} \iota_i^2}\|\Delta_{\vectheta}\|_2-\|\Delta_{\vectheta}\|_{\veciota}.
	\end{align}
	To lower bound the left-hand side of \eqref{ine:l:intercept-01},  from the convexity of $L(\mu,\vectheta)$, and the  sub-gradients of $Q(\mu^*,\vectheta^*)$ at $\mu^*$ and $\vectheta^*$, we have
	\begin{align}
		&Q(\hat{\mu},\hat{\vectheta})-Q(\mu^*,\vectheta^*)\nonumber\\
		&\geq -\frac{\sum_{i=1}^n (y_i-\vecx_i^\top\hat{\vecbeta}-\mu^*-\sqrt{n} \theta^*_i)}{2nQ(\mu^*,\vectheta^*)}(\hat{\mu}-\mu^*)-\frac{\sum_{i=1}^n (y_i-\vecx_i^\top\hat{\vecbeta}-\mu^*-\sqrt{n} \theta^*_i)(\hat{\theta}_i-\theta^*_i)}{2\sqrt{n}Q(\mu^*,\vectheta^*)}.
	\end{align}
	We evaluate the right-hand side at each term.  We have
	\begin{align}
		\label{ine:step1-1}
		\left|\frac{1}{n}\sum_{i=1}^n \xi_i(\hat{\mu}-\mu^*)\right|&\stackrel{(a)}{\lesssim} c_{\mr{M}}\sigma^* \left(r_\delta+r_{\mr{M}}\right)|\Delta_\mu|\\
		\label{ine:step1-2}
		\left|\frac{1}{n}\sum_{i=1}^n\vecx_i^\top \Delta_\beta\Delta_\mu\right|&\stackrel{(b)}{\lesssim} \left\{L\left(\sigma^* \rho\mr{R}_\lambda+\sigma^*\frac{r_\Sigma}{\sqrt{n}}\right)+\sigma^* r_\Sigma r_\delta \right\} |\Delta_\mu|\stackrel{(b')}\lesssim  c_{\mr{M}}\sigma^* r_{\Sigma,\lambda}|\Delta_\mu|,\\
		\label{ine:step1-3}
		\left|\frac{1}{\sqrt{n}}\sum_{i=1}^n\xi_i(\hat{\theta}_i-\theta^*_i)\right|&\leq \max_{j\geq \max\{o,\log(1/\delta)\}}\left(\frac{|\xi^\sharp|_j}{\mr{C}_{\veciota}\iota_j}\right)\mr{C}_{\veciota}\|\Delta_{\vectheta}\|_{\veciota}\stackrel{(c)}{\leq}  c_{\mr{M}}\sigma^*\frac{\mr{C}_{\veciota}}{20}\|\Delta_{\vectheta}\|_{\veciota},\\
		\label{ine:step1-4}
		\left|\frac{1}{\sqrt{n}}\sum_{i=1}^n (\vecx_i^\top \Delta_\beta)(\hat{\theta}_i-\theta^*_i)\right|&\stackrel{(d)}{\lesssim} L\sigma^*\left( \rho \mr{R}_\lambda +r_\Sigma+r_\delta r_\Sigma\right)\|\vectheta\|_2\stackrel{(d')}{\leq } c_{\mr{M}}\sigma^* r_{\Sigma,\lambda}\|\vectheta\|_2\\
		\label{ine:step1-5}
		Q(\mu^*,\vectheta^*)&=\sqrt{\frac{1}{n}\sum_{i=\max\{o,\log(1/\delta)\}}^n\xi_i^2}\stackrel{(e)}{\geq} c_{\mr{M}}\sigma^*\sqrt{\frac{1}{10}},
	\end{align}
	where (a) follows from Lemma \ref{l:noiseconcentration}, (b) follows from Lemma \ref{l:gw2}, (b') follows from $1\leq \sqrt{\log(1/\delta)}$, $r_\delta \leq 1$ and $c_{\mr{M}}\geq 1$, (c) follows from \eqref{ine:noiseconcentration1}, (d) follows from Lemma \ref{l:Che}, (d') follows from $r_\delta\leq 1$ and $c_{\mr{M}}\geq 1$, and (e) follows from \eqref{ine:noiseconcentration2}.

	Consequently, we have
 	\begin{align}
		&-\frac{\mr{C}_{\veciota}\sqrt{10}}{20}\|\Delta_{\vectheta}\|_{\veciota}-C\left\{\left(r_\delta+r_{\mr{M}}+r_{\Sigma,\lambda}\right)|\Delta_\mu|+ r_{\Sigma,\lambda}\|\vectheta\|_2\right\}\leq 2\mr{C}_{\veciota} \sqrt{\sum_{i=1}^{o+\max\{\log(1/\delta),o\}} \iota_i^2}\|\Delta_{\vectheta}\|_2-\mr{C}_{\veciota}\|\Delta_{\vectheta}\|_{\veciota},
	\end{align}
	and  rearranging terms, we have
	\begin{align}
		\label{ine:interceptpre}
		\|\Delta_{\vectheta}\|_{\veciota} &\lesssim  \left(r_\iota+r_{\Sigma,\lambda}\right)\|\Delta_{\vectheta}\|_2 + \left(r_\delta+r_{\mr{M}}+r_{\Sigma,\lambda} \right)|\Delta_\mu|,
	\end{align}
	and the proof is complete.

	\noindent
	\underline{Step 2: Bounding  $|\Delta_\mu|$ via $\|\Delta_{\vectheta}\|_2$}

	From the convexity of $L(\mu,\vectheta)$ and  the first-order optimality condition of $L(\mu,\vectheta)$, we have
	\begin{align}
		\label{ine:firstmu}
		0&= -\frac{\sum_{i=1}^n (y_i-\vecx_i^\top\hat{\vecbeta}-\hat{\mu}-\sqrt{n} \hat{\theta}_i)}{2n}.
	\end{align}
	From $y_i=\vecx_i^\top \vecbeta^*+\xi_i+\sqrt{n}\vectheta_i^*$, we have
	\begin{align}
		\label{ine:muestimatepre}
		|\Delta_\mu|&= \left|\frac{\sum_{i=1}^n \vecx_i^\top\Delta_{\vecbeta}-\sum_{i=1}^n\sqrt{n}( \hat{\theta}_i-\theta^*_i)}{n}\right|\nonumber\\
		&\stackrel{(a)}{\lesssim}  c_{\mr{M}}\sigma^* r_{\Sigma,\lambda}  +\|\Delta_{\vectheta}\|_{\veciota} \nonumber\\
		&\stackrel{(b)}{\lesssim} c_{\mr{M}}\sigma^* r_{\Sigma,\lambda} + \left(r_\iota+r_{\Sigma,\lambda}\right)\|\Delta_{\vectheta}\|_2 + \left(r_\delta+r_{\mr{M}}+r_{\Sigma,\lambda} \right)|\Delta_\mu|,
	\end{align}	
	where (a) follows from the same argument of \eqref{ine:step1-2}, and (b) follows from \eqref{ine:interceptpre}. Then, from the assumption $r_\delta$, $r_{\mr{M}}$ and $r_\Sigma$ are sufficiently small, we have
	\begin{align}
		\label{ine:firstmu2}
		|\Delta_\mu|\lesssim  c_{\mr{M}}\sigma^* r_{\Sigma,\lambda} +( r_\iota+r_{\Sigma,\lambda})\|\Delta_{\vectheta}\|_2  .
	\end{align}	

	\noindent
	\underline{Step 3: Bounding $\|\Delta_{\vectheta}\|_{\veciota}$ via $\|\Delta_{\vectheta}\|_2$}

	From \eqref{ine:firstmu2} and \eqref{ine:interceptpre}, we have
	\begin{align}
		\label{ine:l:intercept2}
		\|\Delta_{\vectheta}\|_{\veciota} &\lesssim  \left(r_\iota+r_{\Sigma,\lambda}\right)\|\Delta_{\vectheta}\|_2 + \left(r_\delta+r_{\mr{M}}+r_{\Sigma,\lambda} \right)\left\{ c_{\mr{M}}\sigma^* r_{\Sigma,\lambda} +( r_\iota+r_{\Sigma,\lambda})\|\Delta_{\vectheta}\|_2 \right\}\nonumber\\
		&\stackrel{(a)}{\lesssim}  \left(r_\iota+r_{\Sigma,\lambda}\right)\|\Delta_{\vectheta}\|_2 +  c_{\mr{M}}\sigma^* r_{\Sigma,\lambda},
	\end{align}
	where (a) follows from  the assumption that $r_\iota,  r_\Sigma, r_\delta$ and $r_{\mr{M}}$ are sufficiently small.

	\noindent
	\underline{Step 4: Bound for $\|\Delta_{\vectheta}\|_2$}

	From the convexity of $L(\mu,\vectheta)$ and  the first-order optimality condition of $L(\mu,\vectheta)$, 
	there exist a vector $\hat{\vecu} \in \partial \|\hat{\vectheta}\|_{\veciota}$, where $\partial$  is the sub-differential, such that $\langle \hat{\vecu},\hat{\vectheta}\rangle=\|\hat{\vectheta}\|_{\veciota}$,
	and we have
	\begin{align}
		\label{ine:firsttheta}
		0&= -\frac{\vecy-\vecv_{\vecx,\hat{\vecbeta}}-\hat{\vecmu}-\sqrt{n} \hat{\vectheta}}{2\sqrt{n}}+\mr{C}_{\veciota} Q(\hat{\mu},\hat{\vectheta})\hat{\vecu},
	\end{align}
	where $\vecy=(y_1,\cdots,y_n)^\top$, $\vecv_{\vecx,\hat{\vecbeta}}=(\vecx_1^\top \hat{\vecbeta},\cdots,\vecx_n^\top \hat{\vecbeta})^\top$, and $\hat{\vecmu}=(\mu,\cdots,\mu)^\top$.
	From $y_i=\vecx_i^\top \vecbeta^*+\xi_i+\sqrt{n}\vectheta_i^*$, and multiplying $\Delta_{\vectheta}$ to the both sides, we have
	\begin{align}
		\label{ine:firsttheta2}
		0&= -\frac{\sum_{i=1}^n (\xi_i-\vecx_i^\top \Delta_{\vecbeta}-\Delta_\mu+\sqrt{n}( \theta^*_i-\hat{\theta}_i))(\hat{\theta}_i-\theta^*_i)}{2\sqrt{n}}+\mr{C}_{\veciota} Q(\hat{\mu},\hat{\vectheta})\langle \hat{\vecu}, \Delta_{\vectheta}\rangle.
	\end{align}
	
	For the each term of the right-hand side of \eqref{ine:firsttheta2}, we have
	\begin{align}
		\label{ine:firsttheta3}
		\frac{\sum_{i=1}^n (\xi_i-\vecx_i^\top \Delta_{\vecbeta})(\hat{\theta}_i-\theta^*_i)}{2\sqrt{n}}&\stackrel{(a)}{\lesssim} c_{\mr{M}}\sigma^*\frac{\mr{C}_{\veciota}}{20}\|\Delta_{\vectheta}\|_{\veciota}+c_{\mr{M}}\sigma^* r_{\Sigma,\lambda}\|\vectheta\|_2,\nonumber\\
		\frac{\sum_{i=1}^n (\mu^*-\hat{\mu})(\hat{\theta}_i-\theta^*_i)}{2\sqrt{n}}&\stackrel{(b)}{\leq}\frac{\|\Delta_{\vectheta}\|_{\veciota}}{2}|\Delta_\mu|,\nonumber\\
		\frac{\sum_{i=1}^n (\sqrt{n} \theta^*_i-\sqrt{n}\hat{\theta}_i)(\hat{\theta}_i-\theta^*_i)}{2\sqrt{n}}&=-\frac{1}{\sqrt{2}}\|\Delta_{\vectheta}\|_2^2,\nonumber\\
		-\langle \hat{\vecu}, \Delta_{\vectheta}\rangle&\stackrel{(c)}{\leq} \|\vectheta^*\|_{\veciota}-\|\hat{\vectheta}\|_{\veciota}\stackrel{(d)}{\leq}  \|\Delta_{\vectheta}\|_{\veciota},\nonumber\\
		Q(\hat{\mu},\hat{\vectheta})&\stackrel{(e)}{\leq} Q(\mu^*,\vectheta^*)+ \|\Delta\|_{\veciota}+\|\Delta_{\vectheta}\|_2\stackrel{(f)}\leq \mr{c}_{\mr{M}}\sigma^*+ \|\Delta\|_{\veciota}+\|\Delta_{\vectheta}\|_2, 
	\end{align}
	where (a) follows from \eqref{ine:noiseconcentration1} and Lemma \ref{l:gw2}, (b) follows from the definition of $\|\cdot\|_{\veciota}$, (c) follows from $-\Delta_{\vectheta}=\hat{\vecu}^\top\vectheta^*-\|\hat{\vectheta}\|_1= \|\vectheta^*\|_{\veclambda}-\|\hat{\vectheta}\|_{\veclambda}$, (d) follows from triangular inequality, (e) follows from the fact that from triangular inequality
	\begin{align}
	|	Q(\hat{\mu},\hat{\vectheta})-Q(\mu^*,\vectheta^*)|&=\frac{1}{\sqrt{2n}}\left|\sqrt{\sum_{i=1}^n (y_i-\vecx_i^\top\hat{\vecbeta}-\hat{\mu}-\sqrt{n} \hat{\theta}_i)^2}-\sqrt{\sum_{i=1}^n (y_i-\vecx_i^\top\hat{\vecbeta}-\mu^*-\sqrt{n} \theta^*_i)^2}\right|\nonumber\\
	&\leq \frac{1}{\sqrt{2n}}\left|\sqrt{\sum_{i=1}^n (-\Delta_\mu-\sqrt{n} (\hat{\theta}_i-\theta^*_i))^2}\right|\nonumber\\
	&\leq \frac{1}{\sqrt{n}}\left|\sqrt{\sum_{i=1}^n (\Delta_\mu^2-n (\hat{\theta}_i-\theta^*_i)^2)}\right|\leq |\Delta_\mu|+\|\vectheta\|_2,
	\end{align}
	and (f) follows from \eqref{ine:noiseconcentration2}.
	Combining these inequalities  and \eqref{ine:firsttheta2}, we have
\begin{align}
	\label{ine:firsttheta4}
	\|\Delta_{\vectheta}\|_2^2&\lesssim  c_{\mr{M}}\sigma^*\|\Delta_{\vectheta}\|_{\veciota}+c_{\mr{M}}\sigma^* r_{\Sigma,\lambda}\|\Delta_{\vectheta}\|_2+\|\Delta_{\vectheta}\|_{\veciota}|\Delta_\mu|+(\mr{c}_{\mr{M}}\sigma^*+\|\Delta_{\vectheta}\|_2)\|\Delta_{\vectheta}\|_{\veciota}.
\end{align}
We evaluate the R.H.S of \eqref{ine:firsttheta4} at each term: For the first term, we have 
\begin{align}
	 \|\Delta_{\vectheta}\|_{\veciota}&\stackrel{(a)}{\lesssim}  \left(r_\iota+r_{\Sigma,\lambda}\right)\|\Delta_{\vectheta}\|_2 + \left(r_\delta+r_{\mr{M}}+r_{\Sigma,\lambda} \right)\{ c_{\mr{M}}\sigma^* r_{\Sigma,\lambda} +( r_\iota+r_{\Sigma,\lambda})\|\Delta_{\vectheta}\|_2 \}\nonumber \\
	 &\stackrel{(b)} \lesssim  \left(r_\iota+r_{\Sigma,\lambda}\right)\|\Delta_{\vectheta}\|_2 + c_{\mr{M}}\sigma^* \left(r_\delta+r_{\mr{M}}+r_{\Sigma,\lambda} \right)^2,
\end{align}
where (a) follows from  \eqref{ine:interceptpre} and \eqref{ine:firstmu2}, and (b) follows from  $r_\iota$ and $r_{\Sigma,\lambda}$ are sufficiently small  and $r_{\mr{M}},r_{\delta}\geq 0$.
For the third term, from \eqref{ine:interceptpre}, we have
\begin{align}
	\|\Delta_{\vectheta}\|_{\veciota} |\Delta_\mu|&\lesssim   \left\{\left(r_\iota+r_{\Sigma,\lambda}\right)\|\Delta_{\vectheta}\|_2 + \left(r_\delta+r_{\mr{M}}+r_{\Sigma,\lambda} \right)|\Delta_\mu|\right\} |\Delta_\mu|,
\end{align}
and from the assumption  $r_\iota$ and $r_{\Sigma,\lambda}$ are sufficiently small, we have
\begin{align}
	\|\Delta_{\vectheta}\|_{\veciota} |\Delta_\mu|&\lesssim   \left(r_\delta+r_{\mr{M}}+r_{\Sigma,\lambda} \right)|\Delta_\mu|^2.
\end{align}
From \eqref{ine:firstmu2}, we have
\begin{align}
	\|\Delta_{\vectheta}\|_{\veciota} |\Delta_\mu|&\lesssim  \left(r_\delta+r_{\mr{M}}+r_{\Sigma,\lambda} \right)\{c_{\mr{M}}\sigma^* r_{\Sigma,\lambda} +( r_\iota+r_{\Sigma,\lambda})\|\Delta_{\vectheta}\|_2 \}^2 \nonumber\\
	&\stackrel{(a)}{\lesssim}  (c_{\mr{M}}\sigma^* r_{\Sigma,\lambda})^2 +( r_\iota+r_{\Sigma,\lambda})^2\|\Delta_{\vectheta}\|_2^2,
\end{align}
where (a) follows from the assumption that $r_\delta, r_{\mr{M}}$ and $r_{\Sigma,\lambda} $ are sufficiently small. For the last term, from  \eqref{ine:interceptpre} and \eqref{ine:firstmu2}, we have
\begin{align}
	&(\mr{c}_{\mr{M}}\sigma^*+\|\Delta_{\vectheta}\|_2)\|\Delta_{\vectheta}\|_{\veciota}\nonumber \\
	&\quad \lesssim (\mr{c}_{\mr{M}}\sigma^*+\|\Delta_{\vectheta}\|_2)\left\{\left(r_\iota+r_{\Sigma,\lambda}\right)\|\Delta_{\vectheta}\|_2 + \left(r_\delta+r_{\mr{M}}+r_{\Sigma,\lambda} \right)\left(c_{\mr{M}}\sigma^* r_{\Sigma,\lambda} +( r_\iota+r_{\Sigma,\lambda})\|\Delta_{\vectheta}\|_2  \right)\right\}\nonumber\\
	&\quad\stackrel{(a)}{ \lesssim} (\mr{c}_{\mr{M}}\sigma^*+\|\Delta_{\vectheta}\|_2)\left\{\left(r_\iota+r_{\Sigma,\lambda}\right)\|\Delta_{\vectheta}\|_2 + c_{\mr{M}}\sigma^* r_{\Sigma,\lambda}   \left(r_\delta+r_{\mr{M}}+r_{\Sigma,\lambda} \right) \right\}\nonumber\\
	&\quad \stackrel{(b)}{ \lesssim}  (r_\iota+r_{\Sigma,\lambda})\|\Delta_{\vectheta}\|_2^2+\mr{c}_{\mr{M}}\sigma^* \left(r_\iota+r_\delta+r_{\mr{M}}+r_{\Sigma,\lambda}\right)\|\Delta_{\vectheta}\|_2+(\mr{c}_{\mr{M}}\sigma^* )^2  \left(r_\delta+r_{\mr{M}}+r_{\Sigma,\lambda} \right)^2,
\end{align}
where (a) follows from the assumption that $r_\delta, r_{\mr{M}},r_\iota$ and $r_{\Sigma,\lambda}$ are sufficiently small, and (b) follows from  the assumption that $r_\delta, r_{\mr{M}}$ and $r_{\Sigma,\lambda}$  are sufficiently small and $r_\delta, r_{\mr{M}}>0$. Then, combining the inequalities above and  \eqref{ine:firsttheta4}, we have
\begin{align}
	\|\Delta_{\vectheta}\|_2^2&\lesssim c_{\mr{M}}\sigma^* \left(r_\iota+r_{\Sigma,\lambda}\right)\|\Delta_{\vectheta}\|_2 + (c_{\mr{M}}\sigma^*)^2 \left(r_\delta+r_{\mr{M}}+r_{\Sigma,\lambda} \right)^2\nonumber\\
	&\quad  +c_{\mr{M}}\sigma^* r_{\Sigma,\lambda}\|\Delta_{\vectheta}\|_2+ (c_{\mr{M}}\sigma^* r_{\Sigma,\lambda})^2 +( r_\iota+r_{\Sigma,\lambda})^2\|\Delta_{\vectheta}\|_2^2+(r_\iota+r_{\Sigma,\lambda})\|\Delta_{\vectheta}\|_2^2\nonumber\\
	&\quad +\mr{c}_{\mr{M}}\sigma^* \left(r_\iota+r_\delta+r_{\mr{M}}+r_{\Sigma,\lambda}\right)\|\Delta_{\vectheta}\|_2+(\mr{c}_{\mr{M}}\sigma^* )^2  \left(r_\delta+r_{\mr{M}}+r_{\Sigma,\lambda} \right)^2.
\end{align}
From the assumption that $r_\iota$ and $r_{\Sigma,\lambda}$ are small, we have
\begin{align}
	\|\Delta_{\vectheta}\|_2^2&\lesssim \mr{c}_{\mr{M}}\sigma^* \left(r_\iota+r_\delta+r_{\mr{M}}+r_{\Sigma,\lambda}\right)\|\Delta_{\vectheta}\|_2+(\mr{c}_{\mr{M}}\sigma^* )^2  \left(r_\iota+r_\delta+r_{\mr{M}}+r_{\Sigma,\lambda} \right)^2.
\end{align}
Consequently, we have
\begin{align}
	\label{ine:thetal2}
	\|\Delta_{\vectheta}\|_2\lesssim  c_{\mr{M}}\sigma^* \left( r_\iota+r_\delta+r_{\mr{M}}+r_{\Sigma,\lambda} \right).
\end{align}

\noindent
\underline{Step 5: Bounding $\|\Delta_{\vectheta}\|_{\veciota}$ without $\|\Delta_{\vectheta}\|_2$}

From \eqref{ine:l:intercept2} and \eqref{ine:thetal2}, we have
\begin{align}
	\label{ine:thetal22}
	\|\Delta_{\vectheta}\|_{\veciota} \lesssim  c_{\mr{M}}\sigma^* \left(r_\iota+r_\delta+r_{\mr{M}}+r_{\Sigma,\lambda} \right)^2.
\end{align}

\noindent
\underline{Step 6: Bounding $|\Delta_\mu|$ without $\|\Delta_{\vectheta}\|_2$}

Remember \eqref{ine:muestimatepre} and from Lemma \ref{l:gw2}, we have
\begin{align}
	|\Delta_\mu|&\lesssim c_{\mr{M}}\sigma^* r_{\Sigma,\lambda}  +\left|\frac{1}{\sqrt{n}} \sum_{i=1}^n(\hat{\theta}_i-\theta^*_i)\right|.
\end{align}	
We evaluate $\left|\frac{1}{\sqrt{n}} \sum_{i=1}^n(\hat{\theta}_i-\theta^*_i)\right|$:
\begin{align}
	\left|\frac{1}{\sqrt{n}} \sum_{i=1}^n(\hat{\theta}_i-\theta^*_i)\right|& \leq \sqrt{\frac{\max\{\log(1/\delta),o\}}{n}}\sqrt{\sum_{i=1}^{\max\{\log(1/\delta),o\}}(\hat{\theta}_i-\theta^*_i) ^2}+\frac{\|\Delta_{\vectheta}\|_{\veciota}}{\sqrt{n}\iota_n}\nonumber\\
	&\leq \sqrt{\frac{\max\{\log(1/\delta),o\}}{n}}\|\Delta_{\vectheta}\|_2+\frac{\|\Delta_{\vectheta}\|_{\veciota}}{\sqrt{n} \iota_n}\nonumber\\
	&\stackrel{(a)}{\leq} \sqrt{\frac{\max\{\log(1/\delta),o\}}{n}} c_{\mr{M}}\sigma^* \left(r_\iota+r_\delta+r_{\mr{M}}+r_{\Sigma,\lambda} \right)+\frac{ c_{\mr{M}}\sigma^* \left(r_\iota+r_\delta+r_{\mr{M}}+r_{\Sigma,\lambda} \right)^2}{\sqrt{n}\iota_n}\nonumber\\
	&\stackrel{(a)}{\leq}c_{\mr{M}}\sigma^*  \left(\sqrt{\frac{\max\{\log(1/\delta),o\}}{n}} \left(r_\iota+r_\delta+r_{\mr{M}}+r_{\Sigma,\lambda} \right)+r_{\Sigma,\lambda}  +\frac{   \left(r_\iota+r_\delta+r_{\mr{M}}+r_{\Sigma,\lambda} \right)^2}{\sqrt{n}\iota_n}\right),
\end{align}
where (a) follows from  from \eqref{ine:l:intercept2} and \eqref{ine:thetal22}, and (b) follows from the assumption that  $r_{\mr{M}}, r_\iota$ and $r_\Sigma$ are sufficiently small
and the definition of $\{\iota_i\}_{i=1}^n$.

\subsection{Proof of Theorem \ref{t:intercept2}}
\label{sec:PTI2}

As in the proof of Theorem \ref{t:abs}, define
\begin{align}
	\mc{O}_\mr{B}&= \{k=1,\cdots,\mr{B}\,\mid\, \mc{O} \cap \{(k-1)\mr{N},(k-1)\mr{N}+1,\cdots,k\mr{N}\} \neq \emptyset\},\nonumber\\
	\Delta &= \hat{\vecbeta}-\vecbeta^*.
\end{align} 
Then, for any positive value $\eta$ and vector $\vecv\in \mbb{R}^d$ satisfying $\|\Delta\|_\Sigma/\eta=1$ and $\|\Delta\|_\Sigma/\eta\leq \mr{R}_{\veclambda}$,  with probability at least $1-\delta$, the following holds:
\begin{align}
	&\sum_{k=1}^{\mr{B}}\mathbb{I}_{\frac{|\hat{\mu}_k-\mbb{E}\frac{1}{n}\sum_{i=1}^n(\xi_i+\vecx_i^\top \vecv+\mu^*)|}{48(\sigma^*+\eta)/\sqrt{N}} \geq 1}\nonumber\\
	&\stackrel{(a)}{=}  \sum_{k=1}^{\mr{B}}\mathbb{I}_{\frac{|\tilde{\mu}_k-\mbb{E}\frac{1}{n}\sum_{i=1}^n(\xi_i+\vecx_i^\top \vecv+\mu^*)|}{48(\sigma^*+\eta)/\sqrt{N}} \geq 1}-\sum_{k\in\mc{O}_{\mr{B}}}\mathbb{I}_{\frac{|\hat{\mu}_k-\mbb{E}\frac{1}{n}\sum_{i=1}^n(\xi_i+\vecx_i^\top \vecv+\mu^*)|}{48(\sigma^*+\eta)/\sqrt{N}} \geq 1}+\sum_{k\in\mc{O}_{\mr{B}}}\mathbb{I}_{\frac{|\tilde{\mu}_k-\mbb{E}\frac{1}{n}\sum_{i=1}^n(\xi_i+\vecx_i^\top \vecv + \mu^*)|}{248(\sigma^*+\eta)/\sqrt{N}} \geq 1}\nonumber\\
	&\leq \sum_{k=1}^{\mr{B}}\mathbb{I}_{\frac{|\tilde{\mu}_k-\mbb{E}\frac{1}{n}\sum_{i=1}^n(\xi_i+\vecx_i^\top \vecv+\mu^*)|}{48(\sigma^*+\eta)/\sqrt{N}} \geq 1}+\sum_{k\in\mc{O}_{\mr{B}}}\mathbb{I}_{\frac{|\tilde{\mu}_k-\mbb{E}\frac{1}{n}\sum_{i=1}^n(\xi_i+\vecx_i^\top \vecv+\mu^*)|}{48(\sigma^*+\eta)/\sqrt{N}} \geq 1}\nonumber\\
	&\stackrel{(b)}{\leq} \sum_{k=1}^{\mr{B}}\mathbb{I}_{\frac{|\tilde{\mu}_k-\mbb{E}\frac{1}{n}\sum_{i=1}^n(\xi_i+\vecx_i^\top \vecv+\mu^*)|}{48(\sigma^*+\eta)/\sqrt{N}} \geq 1}+O\nonumber\\
	&\stackrel{(c)}{\leq}    \frac{18}{48}\mr{B}+\frac{6}{48}\mr{B}\nonumber\\
	&=\frac{1}{2}\mr{B},
\end{align}
where (a) follows from the definition of $\mc{O}_{\mr{B}}$, (b) follows from the fact the upper bound of the number of elements of $\mc{O}_{\mr{B}}$ is $o$, and (c) follows from  Corollary \ref{c:main-mean} and the assumption $O\leq 8\mr{B}$.

By setting, set $\eta= \|\Delta\|_\Sigma $, we have
\begin{align}
	&\left|\hat{\mu}-\mbb{E}\frac{1}{n}\sum_{i=1}^n(\xi_i+\vecx_i^\top \Delta+\mu^*)\right| \leq 48\frac{\sigma^*+\|\Delta\|_\Sigma}{\mr{N}}\lesssim (\sigma^*+\|\Delta\|_\Sigma)\left(\sqrt{\frac{O}{n}}+\sqrt{\frac{\log(1/\delta)}{n}}\right).
\end{align}
From Lemma \ref{l:gw2},  we have
\begin{align}
	\left|\mbb{E}\frac{1}{n}\sum_{i=1}^n\vecx_i^\top \Delta\right|\lesssim \frac{L(\rho r_\Sigma \mr{R}_{\veclambda}\sqrt{n}+r_\Sigma)}{\sqrt{n}} \leq L\rho r_\Sigma \mr{R}_{\veclambda} +Lr_\Sigma\leq 2Lr_\Sigma,
\end{align}
where we use $\rho \mr{R}_{\veclambda}\leq 1$.
Consequently, we have
\begin{align}
	\mbb{P}\left(\left|\hat{\mu}-\mu^*\right| \lesssim (\sigma^*+r_\Sigma )\left(\sqrt{\frac{O}{n}}+\sqrt{\frac{\log(1/\delta)}{n}}\right)+Lr_\Sigma\right)\geq 1-\delta.
\end{align}

\section{Proofs of Propositions in Section \ref{sec:AESC}}
\label{sec:keyproposition}
In Section \ref{sec:keyproposition}, we prove Propositions~\ref{p:main}--\ref{p:main-abs}, and Corollary~\ref{c:main-mean}.
	
	\subsection{Proof of Proposition \ref{p:main}}
	\label{sec:PPM}
	This proof consists of two steps. In the first step, we derive an error bound for $\|\hat{\vecbeta}-\vecbeta^*\|_\Sigma$. In the second step, we derive an error bound for $\|\hat{\vecbeta}-\vecbeta^*\|_{\veclambda}$.

	\noindent
	\underline{First step: Error bound for $\|\hat{\vecbeta}-\vecbeta^*\|_\Sigma$}

	Define $\Delta = \hat{\vecbeta}-\vecbeta^*$. Suppose  $\|\Delta\|_\Sigma > r_\Sigma(\hat{\varsigma}_{\mr{cur}})$, and we will derive a contradiction.  
  By the intermediate value theorem, we can find $\eta \in (0,1)$ such that $\|\eta \Delta\|_\Sigma = r_\Sigma(\hat{\varsigma}_{\mr{cur}})$.
	Define $\Delta_\eta$ as $\eta \Delta\|_\Sigma$.
	First, we present the following lemma, which majorizes  $\|\Delta_\eta\|_{\veclambda}$ by $\|\Delta_\eta\|_\Sigma$:
	\begin{lemma}
		\label{l:l1l2-relation}
		Suppose  $\|\Delta_\eta\|_\Sigma = r_\Sigma(\hat{\varsigma}_{\mr{cur}})$, and assume \eqref{ine:p:main-03}. Then, for any fixed $\eta\in(0,1)$, we have
		\begin{align}
			\|\Delta_\eta\|_{\veclambda} \leq \frac{\mr{C}_{\veclambda}}{\mr{C}_{\veclambda}/2-r_{\mr{l},\lambda } -r_{\mr{u},\lambda}} \sqrt{\frac{\sum_{i=1}^s \lambda_i^2}{\kappa}}\| \Delta_\eta\|_\Sigma.
		\end{align}
	\end{lemma}
	\begin{proof}
		Let
		\begin{align}
			Q'(\eta) =\mr{C}_{\mr{H}} \hat{\varsigma}_{\mr{cur}} \left(-\frac{1}{n}\sum_{i=1}^n h\left(\frac{y_i- \vecx_i^\top(\vecbeta^*+\Delta_\eta )}{\mr{C}_{\mr{H}}\hat{\varsigma}_{\mr{cur}}}\right) \vecx_i^\top \Delta+\frac{1}{n}\sum_{i=1}^n h\left(\frac{y_i- \vecx_i^\top \vecbeta^* }{\mr{C}_{\mr{H}} \hat{\varsigma}_{\mr{cur}}}\right) \vecx_i^\top \Delta \right).
		\end{align}
		From the proof of Lemma F.2 of \cite{FanLiuSunZha2018Lamm}, we have $\eta Q'(\eta)\leq \eta Q'(1)$, and this means 
		\begin{align}
			\label{ine:K.K.T.}
			&\mr{C}_{\mr{H}} \hat{\varsigma}_{\mr{cur}} \left(-\frac{1}{n}\sum_{i=1}^n h\left(\frac{y_i- \vecx_i^\top(\vecbeta^*+\Delta_\eta )}{\mr{C}_{\mr{H}}\hat{\varsigma}_{\mr{cur}}}\right) \vecx_i^\top \Delta_\eta+\frac{1}{n}\sum_{i=1}^n h\left(\frac{y_i- \vecx_i^\top \vecbeta^* }{\mr{C}_{\mr{H}} \hat{\varsigma}_{\mr{cur}}}\right) \vecx_i^\top \Delta_\eta \right)\nonumber\\
			&\quad \quad \leq \mr{C}_{\mr{H}} \hat{\varsigma}_{\mr{cur}} \left(-\frac{1}{n}\sum_{i=1}^n h\left(\frac{y_i- \vecx_i^\top \hat{\vecbeta} }{\mr{C}_{\mr{H}}\hat{\varsigma}_{\mr{cur}}}\right) \vecx_i^\top \Delta_\eta+\frac{1}{n}\sum_{i=1}^n h\left(\frac{y_i- \vecx_i^\top \vecbeta^* }{\mr{C}_{\mr{H}} \hat{\varsigma}_{\mr{cur}}}\right) \vecx_i^\top \Delta_\eta \right).
		\end{align}
		In this proof, for a vector $\vecv$,  define $\partial \vecv$ as a sub-differential of $\|\vecv\|_{\veclambda}$. Note that, from the definition of the sub-differential, we have 
		\begin{align}
			\|\hat{\vecbeta}\|_{\lambda}-\|\vecbeta^*\|_{\veclambda}\leq \langle \partial \hat{\vecbeta},\Delta\rangle,
		\end{align}
		and from the optimality of $\hat{\vecbeta}$, we have
		\begin{align}
		 -\mr{C}_{\mr{H}} \hat{\varsigma}_{\mr{cur}}\frac{1}{n}\sum_{i=1}^n h\left(\frac{y_i- \vecx_i^\top \hat{\vecbeta} }{\mr{C}_{\mr{H}} \hat{\varsigma}_{\mr{cur}}}\right) \vecx_i^\top \Delta+\mr{C}_{\mr{H}} \hat{\varsigma}_{\mr{cur}}\mr{C}_{\veclambda}\langle \partial \hat{\vecbeta},\Delta\rangle=0.
		\end{align}
		Adding $\eta \mr{C}_{\mr{H}} \hat{\varsigma}_{\mr{cur}}\mr{C}_{\veclambda}(\|\hat{\vecbeta}\|_{\veclambda}-\|\vecbeta^*\|_{\veclambda})$ to both sides of \eqref{ine:K.K.T.}, from the two aligns above, we have
		\begin{align}
			&\mr{C}_{\mr{H}}\hat{\varsigma}_{\mr{cur}}\left(-\frac{1}{n}\sum_{i=1}^n h\left(\frac{y_i- \vecx_i^\top(\vecbeta^*+\Delta_\eta )}{\mr{C}_{\mr{H}} \hat{\varsigma}_{\mr{cur}}}\right) \vecx_i^\top \Delta_\eta+\frac{1}{n}\sum_{i=1}^n h\left(\frac{y_i- \vecx_i^\top \vecbeta^* }{\mr{C}_{\mr{H}} \hat{\varsigma}_{\mr{cur}}}\right) \vecx_i^\top \Delta_\eta \right)+\eta \mr{C}_{\mr{H}} \hat{\varsigma}_{\mr{cur}}\mr{C}_{\veclambda}(\|\hat{\vecbeta}\|_{\veclambda}-\|\vecbeta^*\|_{\veclambda})\nonumber\\
			&\quad \quad \leq \mr{C}_{\mr{H}} \hat{\varsigma}_{\mr{cur}} \left(\frac{1}{n}\sum_{i=1}^n h\left(\frac{y_i- \vecx_i^\top \vecbeta^* }{\mr{C}_{\mr{H}} \hat{\varsigma}_{\mr{cur}}}\right) \vecx_i^\top \Delta_\eta \right),
		\end{align}
		and rearranging the terms, we have
		\begin{align}
			\label{ine:KKT}
			&\mr{C}_{\mr{H}} \hat{\varsigma}_{\mr{cur}} \left(-\frac{1}{n}\sum_{i=1}^n h\left(\frac{y_i- \vecx_i^\top(\vecbeta^*+\Delta_\eta )}{\mr{C}_{\mr{H}} \hat{\varsigma}_{\mr{cur}}}\right) \vecx_i^\top \Delta_\eta+\frac{1}{n}\sum_{i=1}^n h\left(\frac{y_i- \vecx_i^\top \vecbeta^* }{\mr{C}_{\mr{H}} \hat{\varsigma}_{\mr{cur}}}\right) \vecx_i^\top \Delta_\eta \right)\nonumber\\
			&\quad \quad \leq \mr{C}_{\mr{H}} \hat{\varsigma}_{\mr{cur}} \left(\frac{1}{n}\sum_{i=1}^n h\left(\frac{y_i- \vecx_i^\top \vecbeta^* }{\mr{C}_{\mr{H}}\hat{\varsigma}_{\mr{cur}}}\right) \vecx_i^\top \Delta_\eta \right)+\eta \mr{C}_{\mr{H}}\hat{\varsigma}_{\mr{cur}}\mr{C}_{\veclambda}(\|\vecbeta^*\|_{\veclambda}-\|\hat{\vecbeta}\|_{\veclambda}),
		\end{align}
	Then, from \eqref{ine:p:main-01} and \eqref{ine:p:main-02}, we have
	\begin{align}
		\label{ine:l:l1l2-relation-03}
		&\frac{\|\Delta_\eta\|_\Sigma^2}{3}- \mr{C}_{\mr{H}} \hat{\varsigma}_{\mr{cur}}r_{\mr{l},\Sigma }\|\Delta_\eta\|_\Sigma- \mr{C}_{\mr{H}} \hat{\varsigma}_{\mr{cur}}r_{\mr{l},\lambda} \|\Delta_\eta\|_{\veclambda}\nonumber\\
		&\quad \quad \leq  \mr{C}_{\mr{H}} \hat{\varsigma}_{\mr{cur}}r_{\mr{u},\Sigma}\|\Delta_\eta\|_\Sigma+\mr{C}_{\mr{H}} \hat{\varsigma}_{\mr{cur}}r_{\mr{u},\lambda}\|\Delta_\eta\|_{\veclambda}+	\mr{C}_{\mr{H}} \hat{\varsigma}_{\mr{cur}}\mr{C}_{\veclambda}\eta (\|\vecbeta^*\|_{\veclambda}-\|\hat{\vecbeta}\|_{\veclambda}).
	\end{align}
	From Lemma A.1 of \cite{BelLecTsy2018Slope}, we have
	\begin{align}
		\label{ine:l:l1l2-relation-04}
		\eta (\|\vecbeta^*\|_{\veclambda}-\|\hat{\vecbeta}\|_{\veclambda}) \leq 2 \sqrt{\sum_{i=1}^s \lambda_i^2}\|\Delta_\eta\|_2-\|\Delta_\eta\|_{\veclambda}.
	\end{align}
	If $\Delta_\eta$ satisfies \eqref{a:RE-2}, from Assumption \ref{a:cov-RE}, we have
	\begin{align}
		\eta (\|\vecbeta^*\|_{\veclambda}-\|\hat{\vecbeta}\|_{\veclambda}) \leq 2 \sqrt{\frac{\sum_{i=1}^s \lambda_i^2}{\kappa}}\|\Delta_\eta\|_\Sigma-\|\Delta_\eta\|_{\veclambda},
	\end{align}
	otherwise,  we have
	\begin{align}
		\eta (\|\vecbeta^*\|_{\veclambda}-\|\hat{\vecbeta}\|_{\veclambda}) \leq \frac{1}{c_{\mr{RE}}}\|\Delta_\eta\|_{\veclambda}-\|\Delta_\eta\|_{\veclambda}= \frac{1-c_{\mr{RE}}}{c_{\mr{RE}}}\|\Delta_\eta\|_{\veclambda}.
	\end{align}
	Combining both cases, we have
	\begin{align}
		\label{ine:l:l1l2-relation-05}
		\eta (\|\vecbeta^*\|_{\veclambda}-\|\hat{\vecbeta}\|_{\veclambda})\leq  \sqrt{\frac{\sum_{i=1}^s \lambda_i^2}{\kappa}}\|\Delta_\eta\|_\Sigma+\frac{1-2c_{\mr{RE}}}{2c_{\mr{RE}}}\|\Delta_\eta\|_{\veclambda}\leq  \sqrt{\frac{\sum_{i=1}^s \lambda_i^2}{\kappa}}\|\Delta_\eta\|_\Sigma-\frac{1}{2}\|\Delta_\eta\|_{\veclambda},
	\end{align}
	and from \eqref{ine:l:l1l2-relation-03}, we have
	\begin{align}
		&\frac{\|\Delta_\eta\|_\Sigma^2}{3}- \mr{C}_{\mr{H}} \hat{\varsigma}_{\mr{cur}}r_{\mr{l},\Sigma }\|\Delta_\eta\|_\Sigma- \mr{C}_{\mr{H}} \hat{\varsigma}_{\mr{cur}}r_{\mr{l},\lambda } \|\Delta_\eta\|_{\veclambda}\nonumber\\
		&\quad \quad \leq  \mr{C}_{\mr{H}} \hat{\varsigma}_{\mr{cur}}r_{\mr{u},\Sigma}\|\Delta_\eta\|_\Sigma+\mr{C}_{\mr{H}} \hat{\varsigma}_{\mr{cur}}r_{\mr{u},\lambda}\|\Delta_\eta\|_{\veclambda}+	 \mr{C}_{\mr{H}} \hat{\varsigma}_{\mr{cur}}\mr{C}_{\veclambda}\sqrt{\frac{\sum_{i=1}^s \lambda_i^2}{\kappa}}\| \Delta_\eta\|_\Sigma -\frac{\mr{C}_{\mr{H}} \hat{\varsigma}_{\mr{cur}}\mr{C}_{\veclambda}}{2}\|\Delta_\eta\|_{\veclambda}.
	\end{align}
	If $\frac{1}{3}\|\Delta_\eta\|_\Sigma^2- \mr{C}_{\mr{H}}\hat{\varsigma}_{\mr{cur}}r_{\mr{u},\Sigma}\|\Delta_\eta\|_\Sigma-\mr{C}_{\mr{H}} \hat{\varsigma}_{\mr{cur}}r_{\mr{l},\Sigma }\|\Delta_\eta\|_\Sigma\leq 0$, then we have $\|\Delta_\eta\|_\Sigma \leq 3\mr{C}_{\mr{H}} \hat{\varsigma}_{\mr{cur}}r_{\mr{u},\Sigma}+3\mr{C}_{\mr{H}} \hat{\varsigma}_{\mr{cur}}r_{\mr{l},\Sigma }$, and this is a contradiction because we assume $\|\Delta_\eta\|_\Sigma = r_\Sigma(\hat{\varsigma}_{\mr{cur}})$, where 
	\begin{align}
		&r_\Sigma(\hat{\varsigma}_{\mr{cur}}) \nonumber\\
		&\quad \quad = 6\mr{C}_{\mr{H}} \hat{\varsigma}_{\mr{cur}} \left(r_{\mr{l},\Sigma } + r_{\mr{u},\Sigma}+\left(\mr{C}_{\veclambda}+\frac{r_{\mr{l},\lambda} +r_{\mr{u},\lambda}}{\mr{C}_{\veclambda}/2-r_{\mr{l},\lambda } -r_{\mr{u},\lambda}}\right)\sqrt{\frac{\sum_{i=1}^s \lambda_i^2}{\kappa}} \right)(> 3\mr{C}_{\mr{H}} \hat{\varsigma}_{\mr{cur}}r_{\mr{u},\Sigma}+3\mr{C}_{\mr{H}} \hat{\varsigma}r_{\mr{l},\Sigma }),
	\end{align}
	and noting that from \eqref{ine:p:main-03}, we see $\mr{C}_{\veclambda}/2-r_{\mr{l},\lambda } -r_{\mr{u},\lambda}>0$. Therefore, we can assume $\frac{1}{3}\|\Delta_\eta\|_\Sigma^2- \mr{C}_{\mr{H}} \hat{\varsigma}_{\mr{cur}}r_{\mr{u},\Sigma}\|\Delta_\eta\|_\Sigma-\mr{C}_{\mr{H}} \hat{\varsigma}_{\mr{cur}}r_{\mr{l},\Sigma }\|\Delta_\eta\|_\Sigma\geq 0$, and we have
	\begin{align}
		\left(\frac{\mr{C}_{\veclambda}}{2} -r_{\mr{u},\lambda}-r_{\mr{l},\lambda}\right)\|\Delta_\eta\|_{\veclambda} \leq \mr{C}_{\veclambda}\sqrt{\frac{\sum_{i=1}^s \lambda_i^2}{\kappa}}\| \Delta_\eta\|_\Sigma.
	\end{align}
	Lastly, from \eqref{ine:p:main-03}, we see $\mr{C}_{\veclambda}/2-r_{\mr{l},\lambda } -r_{\mr{u},\lambda}$ is positive, and the proof is complete.
	\end{proof}
	Then, we proceed with the proof of Proposition \ref{p:main}. From \eqref{ine:l:l1l2-relation-03} and \eqref{ine:l:l1l2-relation-05}, we have
	\begin{align}
		\label{ine:l:main-05}
		&\frac{\|\Delta_\eta\|_\Sigma^2}{3}- \mr{C}_{\mr{H}} \hat{\varsigma}_{\mr{cur}}r_{\mr{l},\Sigma }\|\Delta_\eta\|_\Sigma- \mr{C}_{\mr{H}} \hat{\varsigma}_{\mr{cur}}r_{\mr{l},\lambda} \|\Delta_\eta\|_{\veclambda}\nonumber\\
		&\quad \quad \leq  \mr{C}_{\mr{H}} \hat{\varsigma}_{\mr{cur}}r_{\mr{u},\Sigma}\|\Delta_\eta\|_\Sigma+\mr{C}_{\mr{H}} \hat{\varsigma}_{\mr{cur}}r_{\mr{u},\lambda}\|\Delta_\eta\|_{\veclambda}+	 \mr{C}_{\mr{H}} \hat{\varsigma}_{\mr{cur}}\mr{C}_{\veclambda} \sqrt{\frac{\sum_{i=1}^s \lambda_i^2}{\kappa}}\|\Delta_\eta\|_\Sigma-\frac{\mr{C}_{\mr{H}} \hat{\varsigma}_{\mr{cur}}\mr{C}_{\veclambda} }{2}\|\Delta_\eta\|_{\veclambda}\nonumber\\
		&\quad \quad \leq  \mr{C}_{\mr{H}} \hat{\varsigma}_{\mr{cur}}r_{\mr{u},\Sigma}\|\Delta_\eta\|_\Sigma+\mr{C}_{\mr{H}}\hat{\varsigma}_{\mr{cur}}r_{\mr{u},\lambda}\|\Delta_\eta\|_{\veclambda}+	 \mr{C}_{\mr{H}} \hat{\varsigma}_{\mr{cur}}\mr{C}_{\veclambda} \sqrt{\frac{\sum_{i=1}^s \lambda_i^2}{\kappa}}\|\Delta_\eta\|_\Sigma,
	\end{align}
	and from Lemma \ref{l:l1l2-relation}, we have
	\begin{align}
		\label{ine:l:main-06}
		\|\Delta_\eta\|_\Sigma\leq  3\mr{C}_{\mr{H}} \hat{\varsigma}_{\mr{cur}}\left(r_{\mr{l},\Sigma }+r_{\mr{u},\Sigma}+ \mr{C}_{\veclambda}\sqrt{\frac{\sum_{i=1}^s \lambda_i^2}{\kappa}}\right)+3\mr{C}_{\mr{H}} \hat{\varsigma}_{\mr{cur}}\mr{C}_{\veclambda}\frac{r_{\mr{l},\lambda}+r_{\mr{u},\lambda}}{\mr{C}_{\veclambda}/2-r_{\mr{l},\lambda } -r_{\mr{u},\lambda}} \sqrt{\frac{\sum_{i=1}^s \lambda_i^2}{\kappa}}(<r_\Sigma(\hat{\varsigma}_{\mr{cur}})),
	\end{align}
	However, this is a contradiction. Then, we have $\|\Delta_\eta\|_\Sigma\leq r_\Sigma(\hat{\varsigma}_{\mr{cur}})$.

	\noindent
	\underline{Second step: Error bound for $\|\hat{\vecbeta}-\vecbeta^*\|_{\veclambda}$}

	From \eqref{ine:KKT} and the convexity of the Huber loss, we have
	\begin{align}
		0\leq  \frac{1}{n}\sum_{i=1}^n \mr{C}_{\mr{H}} \hat{\varsigma}_{\mr{cur}} h\left(\frac{y_i- \vecx_i^\top \vecbeta^* }{\mr{C}_{\mr{H}} \hat{\varsigma}_{\mr{cur}}}\right) \vecx_i^\top (\hat{\vecbeta}-\vecbeta^*)+\mr{C}_{\mr{H}} \hat{\varsigma}_{\mr{cur}}\mr{C}_{\veclambda} (\|\vecbeta^*\|_{\veclambda}-\|\hat{\vecbeta}\|_{\veclambda}).
	 \end{align}
	 Then, from \eqref{ine:p:main-01} and \eqref{ine:l:l1l2-relation-05}, we have
	 \begin{align}
		\left(\mr{C}_{\veclambda}/2-r_{\mr{u},\lambda}\right)\|\Delta\|_{\veclambda} \leq \left(r_{\mr{u},\Sigma}+ \mr{C}_{\veclambda}\sqrt{\frac{\sum_{i=1}^s \lambda_i^2}{\kappa}}\right)\| \Delta\|_\Sigma.
	\end{align}
	From \eqref{ine:p:main-03}, we see $\mr{C}_{\veclambda}/2-r_{\mr{u},\lambda}$ is positive, then, we have
	\begin{align}
		\|\Delta\|_{\veclambda} \leq \frac{r_{\mr{u},\Sigma}+ \mr{C}_{\veclambda}\sqrt{\frac{\sum_{i=1}^s \lambda_i^2}{\kappa}}}{\mr{C}_{\veclambda}/2-r_{\mr{u},\lambda}}\| \Delta\|_\Sigma.
	\end{align}

	\subsection{Proof of Proposition \ref{p:ub}}
	\label{sec:PPU}
	We define the following five sets:
	\begin{align}
		\mr{V}_{\mr{up}}&:=\left\{\vecv\in\mbb{R}^{d}, v\in \mbb{R}_+\,\mid\, \|\vecv\|_\Sigma =1,0< v\right\},\nonumber\\
		\mr{V}'_{\mr{up}}&:=\left\{\vecv\in\mbb{R}^{d}, v\in \mbb{R}_+\,\mid\, \|\vecv\|_\Sigma \leq 1,0< v\right\},\nonumber\\
		\mr{U}_{\mr{up}}&:=\left\{\vecv\in\mbb{R}^{d}, v\in \mbb{R}_+\,\mid\, \|\vecv\|_{\veclambda} \leq r_{\veclambda},\,1/v\leq r_\tau',\, v\leq r_\tau\right\},\nonumber\\
		\hat{\mr{V}}_{\mr{up}}&:=\left\{\vecv\in\mbb{R}^{d}\,\mid\, \|\vecv\|_\Sigma \leq 1,\,\|\vecv\|_{\veclambda} \leq r_{\veclambda}\right\},\nonumber\\
		\tilde{\mr{V}}_{\mr{up}}&:=\left\{v\in \mbb{R}_+\,\mid\, 1/v\leq r_\tau',\, v\leq r_\tau\right\}.
	\end{align}
	Our goal is twofold: First, to derive a concentration inequality for the supremum of $\frac{1}{n}\sum_{i=1}^n h\left(\frac{\xi_i}{\mr{C}_{\mr{H}} \tau}\right)\langle \vecx_i,\vecv\rangle$ over $(\vecv,\tau) \in \mr{V}_{\mr{up}} \cap \mr{U}_{\mr{up}}$, and second, to remove the constraints using the peeling device.

	\noindent
	\underline{Application of  Theorem 3.2 in \cite{Dir2015Tail}}

	To apply Theorem 3.2 in \cite{Dir2015Tail}, we verify that $\frac{1}{n}\sum_{i=1}^n h\left(\frac{\xi_i}{\mr{C}_{\mr{H}} \tau}\right)\langle \vecx_i,\vecv\rangle$ satisfies the \textit{the $\psi_\alpha$ condition} with $\alpha = 2$ as defined therein. That is, for any $u > 0$ and any fixed $(\vecv_1, v_1), (\vecv_2, v_2) \in \mr{V}_{\mr{up}}$, the following inequality holds with probability at most $2 \exp(-u)$:
	\begin{align}
		\label{ine:dist}
		\left|\sum_{i=1}^n \left\{ h\left(\frac{\xi_i}{\mr{C}_{\mr{H}} v_1}\right)\langle \vecx_i,\vecv_1\rangle-h\left(\frac{\xi_i}{\mr{C}_{\mr{H}} v_2}\right)\langle \vecx_i,\vecv_2\rangle \right\}\right|\geq  \sqrt{u} (d(\vecv_1,\vecv_2)+d(v_1,v_2)),
	\end{align}
	where $d(\vecv_1,\vecv_2)$ and $d(v_1,v_2)$ are  distances between $\vecv_1$ to $\vecv_2$ and $v_1$ to $v_2$, respectively.

	First, we evaluate the $\psi_2$ norm of $\frac{1}{n}\sum_{i=1}^n h\left(\frac{\xi_i}{\mr{C}_{\mr{H}} v_1}\right)\langle \vecx_i,\vecv_1\rangle- h\left(\frac{\xi_i}{\mr{C}_{\mr{H}} v_2}\right)\langle \vecx_i,\vecv_2\rangle$, and then we use  Theorem 2.6.2 of \cite{Ver2018High} and get an evaluation in the form of \eqref{ine:dist}.
	For  any fixed $(\vecv_1, v_1), (\vecv_2,v_2) \in \mr{V}_{\mr{up}}$,  we have
	\begin{align}
		\label{ine:sec:PPU-01}
		&\left\|h\left(\frac{\xi_i}{\mr{C}_{\mr{H}} v_1}\right)\langle \vecx_i,\vecv_1\rangle- h\left(\frac{\xi_i}{\mr{C}_{\mr{H}} v_2}\right)\langle \vecx_i,\vecv_2\rangle\right\|_{\psi_2}\nonumber\\
		& \quad \quad  \leq\left\|  h\left(\frac{\xi_i}{\mr{C}_{\mr{H}} v_1}\right)\langle \vecx_i,\vecv_1-\vecv_2\rangle\right\|_{\psi_2}+\left\|\left(h\left(\frac{\xi_i}{\mr{C}_{\mr{H}} v_1}\right)-h\left(\frac{\xi_i}{\mr{C}_{\mr{H}} v_2}\right)\right)\langle \vecx_i,\vecv_2\rangle\right\|_{\psi_2}.
	\end{align}
	For the first term of the right-hand side of \eqref{ine:sec:PPU-01}, from $|h(\cdot)|\leq 1$ and the  $L$-subGaussian property of $\{\vecx_i\}_{i=1}^n$, we have
	\begin{align}
		 \left\|  h\left(\frac{\xi_i}{\mr{C}_{\mr{H}} v_1}\right)\langle \vecx_i,\vecv_1-\vecv_2\rangle\right\|_{\psi_2}\leq L \|\vecv_1-\vecv_2\|_\Sigma.
	\end{align}
	For the second term of the right-hand side of  \eqref{ine:sec:PPU-01}, we have
	\begin{align}
		\left\|\left(h\left(\frac{\xi_i}{\mr{C}_{\mr{H}} v_1}\right)-h\left(\frac{\xi_i}{\mr{C}_{\mr{H}} v_2}\right)\right)\langle \vecx_i,\vecv_2\rangle\right\|_{\psi_2}&=\left\|\left|h\left(\frac{\xi_i}{\mr{C}_{\mr{H}} v_1}\right)-h\left(\frac{\xi_i}{\mr{C}_{\mr{H}} v_2}\right)\right||\langle \vecx_i,\vecv_2\rangle|\right\|_{\psi_2}\nonumber\\
		&\stackrel{(a)}{\leq} \left|\frac{v_1-v_2}{v_1}\right|\left\|\langle \vecx_i,\vecv_2\rangle\right\|_{\psi_2}\nonumber\\
		&\stackrel{(b)}{\leq} L \frac{|v_1-v_2|}{v_1}\|\vecv_2\|_\Sigma\leq Lr_\tau'|v_1-v_2|,
	\end{align}
	where (a) follows from Lemma \ref{l:huber_smooth}, and (b) follows from the  $L$-subGaussian property of $\{\vecx_i\}_{i=1}^n$.
	Consequently, we have
	\begin{align}
	\left\|  h\left(\frac{\xi_i}{\mr{C}_{\mr{H}} v_1}\right)\langle \vecx_i,\vecv_1\rangle- h\left(\frac{\xi_i}{\mr{C}_{\mr{H}} v_2}\right)\langle \vecx_i,\vecv_2\rangle\right\|_{\psi_2}\lesssim L \left(\|\vecv_1-\vecv_2\|_\Sigma+r'_\tau|v_1-v_2|\right).
	\end{align}
	Then,  from Theorem 2.6.2 of \cite{Ver2018High}, we have
	\begin{align}
		\mbb{P}\left(	\left| \frac{1}{n}\sum_{i=1}^n h\left(\frac{\xi_i}{\mr{C}_{\mr{H}} v_1}\right)\langle \vecx_i,\vecv_1\rangle- h\left(\frac{\xi_i}{\mr{C}_{\mr{H}} v_2}\right)\langle \vecx_i,\vecv_2\rangle\right|\gtrsim \frac{u}{\sqrt{n}} L \left(\|\vecv_1-\vecv_2\|_\Sigma+r'_\tau|v_1-v_2|\right)\right) \leq e^{-u^2},
		\end{align}
		and we see that $ \frac{1}{n}\sum_{i=1}^n h\left(\frac{\xi_i}{\mr{C}_{\mr{H}} v_1}\right)\langle \vecx_i,\vecv_1\rangle- h\left(\frac{\xi_i}{\mr{C}_{\mr{H}} v_2}\right)\langle \vecx_i,\vecv_2\rangle$ satisfies the $\psi_2$ condition.

		Note that 
		\begin{align}
			\label{ine:sec:PPU-01-dev}
			\sup_{(\vecv_1,v_1),(\vecv_2,v_2)\in \mr{V}'_{\mr{up}}\cap \mr{U}_{\mr{up}}}\left\|h\left(\frac{\xi_i}{\mr{C}_{\mr{H}} v_1}\right)\langle \vecx_i,\vecv_1\rangle- h\left(\frac{\xi_i}{\mr{C}_{\mr{H}} v_2}\right)\langle \vecx_i,\vecv_2\rangle\right\|_{\psi_2} &\leq2 \sup_{(\vecv,v) \in \mr{V}'_{\mr{up}}\cap \mr{U}_{\mr{up}}} \left\|h\left(\frac{\xi_i}{\mr{C}_{\mr{H}} v}\right)\langle \vecx_i,\vecv\rangle\right\|_{\psi_2}\nonumber\\
			&  \leq2 \sup_{(\vecv,v) \in \mr{V}'_{\mr{up}}\cap \mr{U}_{\mr{up}}} \left\|\langle \vecx_i,\vecv\rangle\right\|_{\psi_2}\nonumber\\
			&\leq 2L,
		\end{align}
		where the last inequality follows from the $L$-subGaussian property.
		Also, note that
		\begin{align}
			\sup_{\vecv\in \mr{V}_{\mr{up}}\cap \mr{U}_{\mr{up}}} \left|\frac{1}{n}\sum_{i=1}^nh\left(\frac{\xi_i}{\mr{C}_{\mr{H}} v}\right) \vecx_i^\top\vecv\right|&\leq \sup_{\vecv\in \mr{V}'_{\mr{up}}\cap \mr{U}_{\mr{up}}} \left|\frac{1}{n}\sum_{i=1}^nh\left(\frac{\xi_i}{\mr{C}_{\mr{H}} v}\right) \vecx_i^\top\vecv\right|.
		\end{align}
	Then, from Theorem 3.2 of \cite{Dir2015Tail} with 
	\begin{align}
		t_0 = ({\bf{0}},r_\tau)(\in \mr{V}'_{\mr{up}}\cap \mr{U}_{\mr{up}}),\quad \Delta_d(\mr{V}'_{\mr{up}}\cap \mr{U}_{\mr{up}})&\stackrel{(a)}{\lesssim} \frac{L}{\sqrt{n}},
	\end{align}
	where (a) follows from \eqref{ine:sec:PPU-01-dev}, with probability at least $1-\delta$,  we have 
	\begin{align}
		\sup_{\vecv\in \mr{V}'_{\mr{up}}\cap \mr{U}_{\mr{up}}} \left|\frac{1}{n}\sum_{i=1}^nh\left(\frac{\xi_i}{\mr{C}_{\mr{H}} v}\right) \vecx_i^\top\vecv\right|\leq  L\left( \frac{\gamma_2(\mr{V}'_{\mr{up}}\cap \mr{U}_{\mr{up}}, \|\cdot\|_\Sigma+r'_\tau|\cdot|)}{\sqrt{n}} + \sqrt{\frac{\log(1/\delta)}{n}}\right),
	\end{align}
	noting $\gamma_2(\cdot,\cdot)$ is defined in Definition \ref{d:tg2}. Define $g$ as a standard Gaussian  random variable.
	For $\gamma_2(\mr{V}'_{\mr{up}}\cap \mr{U}_{\mr{up}}, \|\cdot\|_\Sigma+r'_\tau|\cdot|)$, we have
	\begin{align}
		\gamma_2(\mr{V}'_{\mr{up}}\cap \mr{U}_{\mr{up}}, \|\cdot\|_\Sigma+r'_\tau|\cdot|)&\stackrel{(a)}{\leq }\gamma_2(\hat{\mr{V}}_{\mr{up}}, \|\cdot\|_\Sigma)+r'_\tau\gamma_2( \tilde{\mr{V}}_{\mr{up}},|\cdot|)\nonumber\\
		&\stackrel{(b)}{\lesssim}\mbb{E}\sup_{\vecv \in \hat{\mr{V}}_{\mr{up}}}\langle\vecv,\Sigma^\frac{1}{2}\vecg^d\rangle+r'_\tau \mbb{E}\sup_{v \in \tilde{\mr{V}}_{\mr{up}}}(vg)\nonumber\\
		&\stackrel{(c)}\lesssim \rho\sqrt{n} r_{\veclambda}+1+r_\tau'r_\tau,
	\end{align}
	where (a) follows from the definition of $\gamma_2(\cdot,\cdot)$, (b) follows from  majorizing measure theorem (Theorem 2.4.1 \cite{Tal2014Upper}), and (c) follows from Corollary \ref{c:gw}. Combining the arguments, we have 
	\begin{align}
		\mbb{P}\left(\sup_{\vecv\in \mr{V}_{\mr{up}}\cap \mr{U}_{\mr{up}}} \left|\frac{1}{n}\sum_{i=1}^nh\left(\frac{\xi_i}{\mr{C}_{\mr{H}} v}\right) \vecx_i^\top\vecv\right|\lesssim L\left( \frac{ \rho\sqrt{n}r_{\veclambda}+1+r_\tau r_\tau'}{\sqrt{n}} + \sqrt{\frac{\log(1/\delta)}{n}} \right)\right)\geq 1-\delta .
	\end{align}
	\noindent
	\underline{Applying  peeling device}

	For a vector $\vecv = (v_1,\cdots,v_d,v_{d+1}) \in \mbb{R}^{d+1}$, define $\vecv_{1:d} = (v_1,\cdots,v_d)$ and $\vecv_{d+1} = v_{d+1}$.
	To apply Lemma \ref{l:peeling1}, for a vector $\vecv \in \mbb{R}^{d+1}$, define
	\begin{align}
		h_1(\vecv)&:=\|\vecv_{1:d}\|_{\veclambda},\quad h_2(\vecv):= \vecv_{d+1},\quad h_3(\vecv):=\frac{1}{\vecv_{d+1}},\nonumber\\
		g_1(x) &:=x,\quad g_2(x) = x,\quad \mr{a}_1 = \mr{a}_2= CL\frac{1}{\sqrt{n}}.
	\end{align}
	Then, from Lemma \ref{l:peeling1} and $\log(1/\delta) \geq \log 4$, for any $(\vecv,\tau) \in \mr{V}_{\mr{up}}$, we have
	\begin{align}
		\mbb{P}\left(\left|\frac{1}{n}\sum_{i=1}^n h\left(\frac{\xi_i}{\mr{C}_{\mr{H}} \tau}\right) \vecx_i^\top\vecv\right|\lesssim L\left(\rho\|\vecv\|_{\veclambda} + \sqrt{\frac{\log(1/\delta)}{n}} \right)\right)\geq 1-\delta.
	\end{align}
	After normalization $\vecv = \vecv/\|\vecv\|_\Sigma$, we complete the proof.

	\subsection{Proof of Proposition \ref{p:sc}}
	\label{sec:PPS}
	This proof borrows an idea that also appears in \cite{SunZhoFan2020Adaptive}.
	Define 
	\begin{align}
		\psi(x) =\begin{cases}
			x^2 & \mbox{ if } |x| \leq 1/4\\
			(x-1/2)^2 & \mbox{ if } 1/4\leq x \leq 1/2 \\
			(x+1/2)^2 & \mbox{ if } -1/2\leq x \leq -1/4 \\
			0 & \mbox{ if } |x| >1/2
		\end{cases}.
	\end{align}
	From the convexity of the Huber loss, for any $(\vecv,\tau) \in \mr{V}_{\mr{sc}}$, we have
	\begin{align}
		&\frac{1}{n}\sum_{i=1}^n \mr{C}_{\mr{H}}\tau\left(-h\left(\frac{\xi_i}{\mr{C}_{\mr{H}}\tau}-\frac{\langle \vecx_{i}, \vecv\rangle}{\mr{C}_{\mr{H}}\tau}\right)+h\left(\frac{\xi_i}{\tau}\right)  \right) \vecx_i^\top \vecv \nonumber\\
		&\quad \quad \geq  \frac{1}{n}\sum_{i=1}^n \mr{C}_{\mr{H}}\tau\left(-h\left(\frac{\xi_i}{\mr{C}_{\mr{H}}\tau}-\frac{\langle \vecx_{i}, \vecv\rangle}{\mr{C}_{\mr{H}}\tau}\right)+h\left(\frac{\xi_i}{\mr{C}_{\mr{H}}\tau}\right)  \right) \vecx_i^\top \vecv\times \mathbb{I}_{ \left( \left|\frac{\xi_i}{\mr{C}_{\mr{H}}\tau}\right| \leq \frac{1}{2}\right) \cap \left( \left|\frac{\vecx_i^\top \vecv}{\mr{C}_{\mr{H}}\tau}\right| \leq \frac{1}{2}\right)}\nonumber\\
		&\quad \quad =  \frac{1}{n}\sum_{i=1}^n (\vecx_i^\top \vecv)^2\times \mathbb{I}_{ \left( \left|\frac{\xi_i}{\mr{C}_{\mr{H}}\tau}\right| \leq \frac{1}{2}\right) \cap \left( \left|\frac{\vecx_i^\top \vecv}{\mr{C}_{\mr{H}}\tau}\right| \leq \frac{1}{2}\right)}\nonumber\\
		&\quad \quad \geq (\mr{C}_{\mr{H}} \tau)^2\frac{1}{n}\sum_{i=1}^n \psi\left(\frac{\vecx_i^\top \vecv}{\mr{C}_{\mr{H}} \tau}\right) \times \mathbb{I}_{ \left|\frac{\xi_i}{\mr{C}_{\mr{H}} \varsigma^*}\right| \leq \frac{1}{2}}.
 	\end{align}
 	Decompose  
	\begin{align}
		&\frac{1}{n}\sum_{i=1}^n \psi\left(\frac{\vecx_i^\top \vecv}{\mr{C}_{\mr{H}} \tau}\right) \times \mathbb{I}_{ \left|\frac{\xi_i}{\mr{C}_{\mr{H}} \varsigma^*}\right| \leq \frac{1}{2}}\geq -\sup_{\vecv,\tau \in \mr{V}_{\mr{sc}}}\Delta_{\vecv,\tau}+  \mbb{E}\psi\left(\frac{\vecx^\top \vecv'}{\mr{C}_{\mr{H}} }\right) \times \mathbb{I}_{ \left|\frac{\xi}{\mr{C}_{\mr{H}} \varsigma^*}\right| \leq \frac{1}{2}},
	\end{align}
	where $\vecv'\in \mbb{R}^d$ is a fixed vector such that $\|\vecv'\|_\Sigma = \mr{C}_{\mr{H}} \mr{r}$, 
	\begin{align}
		 \Delta_{\vecv,\tau}:=\frac{1}{n}\sum_{i=1}^n \psi\left(\frac{\vecx_i^\top \vecv}{\mr{C}_{\mr{H}} \tau}\right) \times \mathbb{I}_{ \left|\frac{\xi_i}{\mr{C}_{\mr{H}} \varsigma^*}\right| \leq \frac{1}{2}}-\mbb{E}\psi\left(\frac{\vecx_i^\top \vecv}{\mr{C}_{\mr{H}} \tau}\right) \times \mathbb{I}_{ \left|\frac{\xi_i}{\mr{C}_{\mr{H}} \varsigma^*}\right| \leq \frac{1}{2}},
	\end{align}
	and $\vecx$ and $\xi$ are random vector and random variable, which follow the same distribution as ones of $\{\vecx_i\}_{i=1}^n$ and $\{\xi_i\}_{i=1}^n$, respectively.

	\noindent
	\underline{Bounding $\mbb{E}\psi\left(\frac{\vecx^\top \vecv'}{\mr{C}_{\mr{H}} }\right) \times \mathbb{I}_{ \left|\frac{\xi}{\mr{C}_{\mr{H}} \varsigma^*}\right| \leq \frac{1}{2}} $}

	Here, we write $\vecv'$ as $\vecv$ for simplicity.
	We have
	\begin{align}
		\label{ine:PPS-01}
		\mbb{E}\psi\left(\frac{\vecx^\top \vecv}{\mr{C_{\mr{H}}}}\right)\times \mathbb{I}_{ \left|\frac{\xi}{\mr{C}_{\mr{H}} \varsigma^*}\right| \leq \frac{1}{2}}\geq \mbb{E}\left(\frac{\vecx^\top \vecv}{\mr{C_{\mr{H}}}}\right)^2-\mbb{E}\left(\frac{\vecx^\top \vecv}{\mr{C_{\mr{H}}}}\right)^2\mathbb{I}_{ \left|\frac{2\xi}{\mr{C}_{\mr{H}} \varsigma^*}\right| \geq \frac{1}{2}}-\mbb{E}\left(\frac{\vecx^\top \vecv}{\mr{C_{\mr{H}}}}\right)^2\mathbb{I}_{\left|\frac{\langle \vecx, \vecv\rangle}{\mr{C}_{\mr{H}} \varsigma^*}\right|\geq \frac{1}{4}}.
	\end{align}
	We evaluate the R.H.S. of \eqref{ine:PPS-01} at each term. First, for any fixed $\vecv$ such that $\|\vecv\|_\Sigma = \mr{C}_{\mr{H}} \mr{r}$, we have
	\begin{align}
	\label{ine:PPS-02}
		\mbb{E}\left(\frac{\vecx^\top \vecv}{\mr{C_{\mr{H}}}}\right)^2\mathbb{I}_{\left|\frac{\langle \vecx_{i}, \vecv\rangle}{\mr{C}_{\mr{H}} }\right|\geq \frac{1}{4}}&\stackrel{(a)}{\leq} \sqrt{\mbb{E}\left(\frac{\vecx^\top \vecv}{\mr{C_{\mr{H}}}}\right)^4} \sqrt{\mbb{P}\left(\left|\frac{\vecx^\top \vecv}{\mr{C}_{\mr{H}} }\right|\geq \frac{1}{4}\right)}\nonumber\\
		&\stackrel{(b)}{\leq} \sqrt{  \mbb{E}\left(\frac{\vecx^\top \vecv}{\mr{C_{\mr{H}}}}\right)^4} \sqrt{16^2 \mbb{E}\left(\frac{\vecx^\top \vecv}{\mr{C}_{\mr{H}} }\right)^4}\nonumber\\
		&\stackrel{(c)}{\leq} 16\sqrt{\frac{L^4}{\mr{C}_{\mr{H}}^4} \|\vecv\|_\Sigma^4} \sqrt{\frac{L^4}{\mr{C}_{\mr{H}}^4} \|\vecv\|_\Sigma^4}\nonumber\\
		&=  16\frac{L^4}{\mr{C}_{\mr{H}}^4} \|\vecv\|_\Sigma^4\nonumber\\
		&\stackrel{(d)}{=}  16\frac{L^4}{\mr{C}_{\mr{H}}^2}\mr{r}^2 \|\vecv\|_\Sigma^2\nonumber\\
		&\stackrel{(e)}{\leq}\frac{\|\vecv\|_\Sigma^2 }{3\mr{C}^2_{\mr{H}} } =\frac{\mr{r}^2}{3},
	\end{align}
	where (a) follows from H{\"o}lder's inequality, (b) follows from the relation between indicator function and expectation, and  Markov's inequality,  (c) follows from \eqref{ine:d:l-02}, (d) follows from the definition of $\mr{V}_{\mr{sc}}$, and (e) follows from the assumption $\mr{r}^2\leq \frac{1}{48L^4}$.
	Second, for any fixed $\vecv$ such that $\|\vecv\|_\Sigma = \mr{C}_{\mr{H}} \mr{r}$, we have
	\begin{align}
	\label{ine:PPS-03}
		\mbb{E}\left(\frac{\vecx^\top \vecv}{\mr{C_{\mr{H}}}}\right)^2\mathbb{I}_{ \left|\frac{\xi}{\mr{C}_{\mr{H}} \varsigma^*}\right| \geq \frac{1}{2}}
		&\stackrel{(a)}{\leq} \sqrt{\mbb{E}\frac{(\vecx^\top \vecv)^4}{\mr{C}_{\mr{H}}^4 }}\sqrt{\frac{2}{\mr{C}_{\mr{H}}\varsigma^*}\mbb{E}|\xi|}\nonumber\\
		&\stackrel{(b)}{\leq} 2\frac{L^2 \|\vecv\|_\Sigma^2}{\mr{C}_{\mr{H}}^2 }\sqrt{\frac{2}{\mr{C}_{\mr{H}}\varsigma^*}\mbb{E}|\xi|}\nonumber\\
		&\stackrel{(c)}{=} 2\frac{L^2 \|\vecv\|_\Sigma^2}{\mr{C}_{\mr{H}}^2 }\sqrt{\frac{2}{\mr{C}_{\mr{H}}}}\nonumber\\
		&\stackrel{(d)}{\leq} \frac{\|\vecv\|_\Sigma^2}{3\mr{C}_{\mr{H}}^2 } 	=\frac{\mr{r}^2}{3} ,
	\end{align}
	where (a) follows from H{\"o}lder's inequality, relation between indicator function and expectation, and Markov's inequality, and (b) follows from \eqref{ine:d:l-02}, (c) follows from the  assumption on $\xi$, and (d) follows the assumption  $\mr{C}_{\mr{H}}\geq 72L^4$.
	
	Consequently, we have 
	\begin{align}
		(\mr{C}_{\mr{H}} \tau)^2\mbb{E}\psi\left(\frac{\vecx^\top \vecv'}{\mr{C}_{\mr{H}}}\right) \times \mathbb{I}_{ \left|\frac{\xi}{\mr{C}_{\mr{H}} \varsigma^*}\right| \leq \frac{1}{2}} \geq (\mr{C}_{\mr{H}} \tau)^2\frac{\mr{r}^2}{3}=\frac{\|\vecv\|_\Sigma^2}{3}.
	\end{align}
	\noindent
	\underline{Bounding $\sup_{\vecv,\tau \in \mr{V}_{\mr{sc}}}\Delta_{\vecv,\tau}$} 

	Define
	\begin{align}
		\mr{V}_{\mr{sc}}':=\left\{\vecv \in \mbb{R}^d\,\mid\,\|\vecv\|_\Sigma = \mr{C}_{\mr{H}} \mr{r} \right\}.
	\end{align}
	We will prove that:  for any $\vecv \in \mr{V}_{\mr{sc}}'$, 	with probability at least $1-\delta$,
	and we have 
	\begin{align}
		\label{ine:PPS-04}
	\frac{1}{n}\sum_{i=1}^n \psi\left(\frac{\vecx_i^\top \vecv}{\mr{C}_{\mr{H}} }\right) \times \mathbb{I}_{ \left|\frac{\xi_i}{\mr{C}_{\mr{H}} \varsigma^*}\right| \leq \frac{1}{2}}-\mbb{E}\psi\left(\frac{\vecx_i^\top \vecv}{\mr{C}_{\mr{H}} }\right) \times \mathbb{I}_{ \left|\frac{\xi_i}{\mr{C}_{\mr{H}} \varsigma^*}\right| \leq \frac{1}{2}} \lesssim L \left(\frac{\|\vecv\|_{\veclambda}}{\mr{C}_{\mr{H}}}+\sqrt{\frac{\log\left(1/\delta\right)}{n}}\mr{r}+\frac{\log\left(1/\delta\right)}{n}\right).
	\end{align}
	Then, from the fact that $\vecv/\tau \in \mr{V}_{\mr{sc}}'$  for any $\vecv,\tau \in \mr{V}_{\mr{sc}}$, and  then, for any $\vecv, \tau \in \mr{V}_{\mr{sc}}$, with probability at least $1-\delta$,  we have
	\begin{align}
		\Delta_{\vecv,\tau}  \lesssim L \left(\frac{\|\vecv\|_{\veclambda}}{\mr{C}_{\mr{H}}\tau}+\sqrt{\frac{\log\left(1/\delta\right)}{n}}\mr{r}+\frac{\log\left(1/\delta\right)}{n}\right).
	\end{align}
	Therefore, in the remaining part, we prove \eqref{ine:PPS-04}.

	Define
	\begin{align}
		\mr{U}_{\mr{sc}}:=\{\vecv \in \mbb{R}^{d}\,\,|\,\,\|\vecv\|_{\veclambda} \leq r_{\veclambda}\}.
	\end{align}
	First, we evaluate $\sup_{\vecv \in \mr{V}'_{\mr{sc}}\cap \mr{U}_{\mr{sc}}} \Delta_{\vecv,\tau}$,
	and after that, we will use peeling lemma to get rid of the constraint.
	From Theorem 2.3 of \cite{Bou2002Bennett} and symmetrization (Lemma 6.4.2 of \cite{Ver2018High}), with probability at least $1-\delta$, 
	\begin{align}
		&\sup_{\vecv \in  \mr{V}'_{\mr{sc}}\cap \mr{U}_{\mr{sc}}}\left| \Delta_{\vecv,\tau}\right|\leq 2\mbb{E}\sup_{\vecv \in \mr{V}'_{\mr{sc}}\cap \mr{U}_{\mr{sc}}}\left| \frac{1}{n}\sum_{i=1}^n \epsilon_i\psi\left(\frac{\vecx_i^\top \vecv}{\mr{C}_{\mr{H}}}\right) \times \mathbb{I}_{ \left|\frac{\xi_i}{\mr{C}_{\mr{H}} \varsigma^*}\right| \leq \frac{1}{2}}\right| +\sup_{\vecv \in \mr{V}'_{\mr{sc}}\cap \mr{U}_{\mr{sc}}}\sqrt{2\log(1/\delta)v(\vecv)}+2\frac{\log(1/\delta)}{n},
	\end{align}
	where $v(\vecv)=\mbb{E}\Delta_{\vecv,\tau}^2$.
	Note that 
	\begin{align}
		\sup_{\vecv \in  \mr{V}'_{\mr{sc}}\cap \mr{U}_{\mr{sc}}}v(\vecv)\leq \frac{1}{n}\sup_{\vecv \in  \mr{V}'_{\mr{sc}}\cap \mr{U}_{\mr{sc}}}\mbb{E}\psi\left(\frac{\vecx_i^\top \vecv}{\mr{C}_{\mr{H}} }\right)^2\lesssim  L^2 \mr{r}^2.
	\end{align}
	For the first term,  define $\hat{\mr{V}}_{\mr{sc}}:=\left\{\vecv \in \mbb{R}^d\,\mid\,\|\vecv\|_\Sigma \leq \mr{C}_{\mr{H}} \mr{r} \right\}$, and  we have
	\begin{align}
		\mbb{E}\sup_{\vecv \in  \mr{V}'_{\mr{sc}}\cap \mr{U}_{\mr{sc}}}\left| \frac{1}{n}\sum_{i=1}^n \epsilon_i\psi\left(\frac{\vecx_i^\top \vecv}{\mr{C}_{\mr{H}}}\right) \times \mathbb{I}_{ \left|\frac{\xi_i}{\mr{C}_{\mr{H}} \varsigma^*}\right| \leq \frac{1}{2}}\right|&\leq \mbb{E}\sup_{\vecv \in  \hat{\mr{V}}_{\mr{sc}}\cap \mr{U}_{\mr{sc}}}\left| \frac{1}{n}\sum_{i=1}^n \epsilon_i\psi\left(\frac{\vecx_i^\top \vecv}{\mr{C}_{\mr{H}}}\right) \times \mathbb{I}_{ \left|\frac{\xi_i}{\mr{C}_{\mr{H}} \varsigma^*}\right| \leq \frac{1}{2}}\right|\nonumber\\
		&\stackrel{(a)}{\lesssim} \mbb{E}\sup_{\vecv \in  \hat{\mr{V}}_{\mr{sc}}\cap \mr{U}_{\mr{sc}}}\frac{1}{n}\sum_{i=1}^n \epsilon_i\psi\left(\frac{\vecx_i^\top \vecv}{\mr{C}_{\mr{H}}}\right) \times \mathbb{I}_{ \left|\frac{\xi_i}{\mr{C}_{\mr{H}} \varsigma^*}\right| \leq \frac{1}{2}}\nonumber\\
		&\stackrel{(b)}{\leq} \mbb{E}\sup_{\vecv \in \mr{V}'_{\mr{sc}}\cap \mr{U}_{\mr{sc}}} \frac{1}{n}\sum_{i=1}^n \epsilon_i\psi\left(\frac{\vecx_i^\top \vecv}{\mr{C}_{\mr{H}}}\right) \nonumber\\
		&\stackrel{(c)}{\leq} \mbb{E}\sup_{\vecv \in \mr{V}'_{\mr{sc}}\cap \mr{U}_{\mr{sc}}}\frac{1}{n}\sum_{i=1}^n \epsilon_i\frac{\vecx_i^\top \vecv}{\mr{C}_{\mr{H}} }\nonumber\\
		&\stackrel{(d)}{\lesssim} L\frac{\rho r_{\veclambda}\sqrt{n}+\mr{C}_{\mr{H}}\mr{r}}{\mr{C}_{\mr{H}}\sqrt{n}},
	\end{align}
	where (a) follows from Exercise 2.2.2 of \cite{Tal2014Upper}, (b) follows from Theorem 11.5 of \cite{BouLugMas2013concentration}, (c) follows from Theorem 11.6 of  \cite{BouLugMas2013concentration}, and (d) follows from Lemma \ref{l:gw2}.
	Combining the arguments and from the fact $L\geq 1$, we have
	\begin{align}
		\sup_{\vecv \in  \mr{V}'_{\mr{sc}}\cap \mr{U}_{\mr{sc}}}\left| \Delta_{\vecv,\tau}\right|\leq C\left(L\frac{\rho r_{\veclambda}\sqrt{n}+\mr{C}_{\mr{H}}\mr{r}}{\mr{C}_{\mr{H}}\sqrt{n}}+L\mr{r}\sqrt{\frac{\log(1/\delta)}{n}}+L\frac{\log(1/\delta)}{n}\right).
	\end{align} 
	To apply Lemma \ref{l:peeling2}, we define
	\begin{align}
		h(\vecv):=\|\vecv\|_{\veclambda}, \quad g(x) = \frac{\rho x\sqrt{n}+1}{\mr{C}_{\mr{H}}},\quad \mr{a}=\mr{r},\quad\mr{a}_1 = \mr{a}_2= LC/\sqrt{n},\quad \mr{a}_3 =LC/n,
	\end{align}
	and we see $\mr{a}_2/\mr{a}_1\leq 1$ and $\mr{a}_3/\mr{a}_1\leq 1/\sqrt{n}\leq \mr{r}$. Then, from Lemma \ref{l:peeling2} and $\log(1/\delta)\geq 1$, for any $\vecv\in \mr{V}_{\mr{sc}}'$, we have
	\begin{align}
		\frac{1}{n}\sum_{i=1}^n \psi\left(\frac{\vecx_i^\top \vecv}{\mr{C}_{\mr{H}} }\right) \times \mathbb{I}_{ \left|\frac{\xi_i}{\mr{C}_{\mr{H}} \varsigma^*}\right| \leq \frac{1}{2}}-\mbb{E}\psi\left(\frac{\vecx_i^\top \vecv}{\mr{C}_{\mr{H}} }\right) \times \mathbb{I}_{ \left|\frac{\xi_i}{\mr{C}_{\mr{H}} \varsigma^*}\right| \leq \frac{1}{2}}
		&\lesssim L \left(\frac{\rho\|\vecv\|_{\veclambda}}{\mr{C}_{\mr{H}}}+\sqrt{\frac{\log\left(1/\delta\right)}{n}}\mr{r}+\frac{\log\left(1/\delta\right)}{n}\right).
	\end{align}

	\noindent
	\underline{Combining the arguments} 

	Combining the arguments above, for any $(\vecv,\tau)\in \mr{V}_{\mr{sc}}$, we have
	\begin{align}
		&\frac{1}{n}\sum_{i=1}^n \mr{C}_{\mr{H}}\tau\left(-h\left(\frac{\xi_i}{\mr{C}_{\mr{H}}\tau}-\frac{\langle \vecx_{i}, \vecv\rangle}{\mr{C}_{\mr{H}}\tau}\right)+h\left(\frac{\xi_i}{\tau}\right)  \right) \vecx_i^\top \vecv \nonumber\\
		&\quad \quad \geq \frac{\|\vecv\|_\Sigma}{3}-C(\mr{C}_{\mr{H}} \tau)^2L \left(\rho\frac{\|\vecv\|_{\veclambda}}{\mr{C}_{\mr{H}}\tau}+\sqrt{\frac{\log\left(1/\delta\right)}{n}}\mr{r}+\frac{\log\left(1/\delta\right)}{n}\right)\nonumber\\
		&\quad \quad \geq \frac{\|\vecv\|_\Sigma^2}{3 }- CL\mr{C}_{\mr{H}}\tau\rho  \|\vecv\|_{\veclambda}-2L \mr{C}_{\mr{H}} \tau \sqrt{\frac{\log(1/\delta)}{n}}\|\vecv\|_\Sigma,
 	\end{align}
	where the last line follows from $\mr{C}_{\mr{H}}\tau \mr{r}=\|\vecv\|_\Sigma$ for $(\vecv,\tau) \in \mr{V}_{\mr{sc}}$,  $L\geq 1$ and $\mr{r}\geq \sqrt{\frac{\log(1/\delta)}{n}}$.

	\subsection{Proof of Proposition \ref{p:oq}}
		This proof is heavily rely on Theorem 5.5 of \cite{Dir2015Tail}. 		Define $r_{\veclambda}$ as a positive constant. For the first step, we derive a concentration inequality for constrained $\vecv \in r_{\veclambda}\mbb{B}_\Sigma^d \cap \mbb{S}_\Sigma^d$ using Theorem 5.5 of \cite{Dir2015Tail}, and the second step, we use a peeling device to get rid of the constraint. 
		
		\noindent
		\underline{First step: Concentration}

		First, we derive a concentration inequality for full sample, and next, we will derive a oncentration inequality for the partial sample.
		To derive a concentration inequality for full sample, we note that, from \eqref{d:l}, we have the following four inequalities:
		\begin{align}
		\label{ine:p:oq-01}
			&\max_{i=1,\cdots,n} \|\vecx_i^\top (\vecv-\vecv')\|_{\psi_2}\leq L \|\Sigma^\frac{1}{2}(\vecv-\vecv')\|_2,\\
		\label{ine:p:oq-02}
			&\sup_{\vecv \in r_{\veclambda} \mbb{B}_\Sigma^d \cap \mbb{S}_\Sigma^d} \max_{i=1,\cdots,n}\|\vecx_i^\top \vecv\|_{\psi_2}\leq \sup_{\vecv \in r_{\veclambda}\mbb{B}_\Sigma^d \cap \mbb{S}_\Sigma^d} L \|\Sigma^\frac{1}{2}\vecv\|_2\leq L,\\
		\label{ine:p:oq-03}
			&\sup_{\vecv \in r_{\veclambda} \mbb{B}_\Sigma^d \cap \mbb{S}_\Sigma^d} \left(\frac{1}{n}\sum_{i=1}^n \|\vecx_i^\top \vecv\|_{\psi_2}^4 \right)^{1/2}\leq \sup_{\vecv \in r_{\veclambda} \mbb{B}_\Sigma^d \cap \mbb{S}_\Sigma^d} \left(L^4 \|\Sigma^\frac{1}{2}\vecv\|_2^4 \right)^{1/2}\leq L^2 ,\\
		\label{ine:p:oq-04}
			&\sup_{\vecv \in r_{\veclambda} \mbb{B}_\Sigma^d \cap \mbb{S}_\Sigma^d} \max_{i=1,\cdots,n}\|\vecx_i^\top \vecv\|_{\psi_2}^2\leq \sup_{\vecv \in r_{\veclambda} \mbb{B}_\Sigma^d \cap \mbb{S}_\Sigma^d} \|\Sigma^\frac{1}{2}\vecv\|_2^2\leq L^2 .
	\end{align}
	Then, from the inequalities above, applying Theorem 5.5 of \cite{Dir2015Tail}, with probability at most $e^{-t}$,  we have 
	\begin{align}
		\label{ine:p:oq-05}
		\sup_{\vecv \in r_{\veclambda} \mbb{B}_\Sigma^d \cap \mbb{S}_\Sigma^d}\left|\frac{1}{n}\sum_{i=1}^n (\vecx_i^\top \vecv)^2-\mbb{E}(\vecx_i^\top \vecv)^2\right|&\gtrsim L^2\left( \frac{\gamma_2(r_{\veclambda} \mbb{B}_\Sigma^d \cap \mbb{S}_\Sigma^d , \|\cdot\|_\Sigma)^2}{n}+\frac{\gamma_2(r_{\veclambda}\mbb{B}_\Sigma^d \cap \mbb{S}_\Sigma^d , \|\cdot\|_\Sigma)}{\sqrt{n}} + \sqrt{\frac{t}{n}} + \frac{t}{n}\right)\nonumber \\
		& \stackrel{(a)}{\gtrsim}L^2\left( \frac{\left(\sup_{\vecv\in r_{\veclambda} \mbb{B}_\Sigma^d \cap \mbb{S}_\Sigma^d}\langle\vecv,\Sigma^\frac{1}{2}\vecg^d \rangle\right)^2}{n}+\frac{\sup_{\vecv\in r_{\veclambda} \mbb{B}_\Sigma^d \cap \mbb{S}_\Sigma^d }\langle\vecv,\Sigma^\frac{1}{2}\vecg^d \rangle}{\sqrt{n}} + \sqrt{\frac{t}{n}} + \frac{t}{n}\right)\nonumber \\
		&\stackrel{(b)}{\gtrsim} L^2\left( \frac{(\rho\sqrt{n}r_{\veclambda}+1)^2}{n}+\frac{\rho\sqrt{n}r_{\veclambda}+1}{\sqrt{n}}  + \sqrt{\frac{t}{n}} + \frac{t}{n}\right),
	\end{align}
	where (a) from the Majorizing Measure theorem (Theorem 2.4.1. of \cite{Tal2014Upper}), and (b) follows from Corollary \ref{c:gw}.
	Then, we consider for the partial sample. From \eqref{ine:p:oq-05}, we have
	\begin{align}
	\label{ine:p:oq-06}
	 	\sup_{\vecv \in r_{\veclambda}\mbb{B}_\Sigma^d \cap \mbb{S}_\Sigma^d}\left| \frac{1}{o} \sum_{i\in \mc{O}} (\vecx_i^\top\vecv)^2  -\mbb{E}(\vecx_i^\top \vecv)^2\right|\gtrsim  L^2\left( \frac{(\rho\sqrt{n}r_{\veclambda}+1)^2}{o}+\frac{\rho\sqrt{n}r_{\veclambda}+1}{\sqrt{o}} + \sqrt{\frac{t}{o}} + \frac{t}{o}\right),
	\end{align}
	with probability at least 
	\begin{align}
		&\leq \binom{n}{o}\times \mbb{P}\left[\sup_{\vecv \in r_{\veclambda}\mbb{B}_\Sigma^d \cap \mbb{S}_\Sigma^d}\left|\frac{1}{o}\sum_{i=1}^o (\vecz_i^\top\vecv)^2  -\mbb{E}(\vecz_i^\top \vecv)^2 \right|\gtrsim   L^2\left(\frac{(\rho\sqrt{n}r_{\veclambda}+1)^2}{o}+\frac{\rho\sqrt{n}r_{\veclambda}+1}{\sqrt{n}} + \sqrt{\frac{t}{o}}+ \frac{t}{o}\right)\right]\nonumber \\
		&\leq \binom{n}{o} e^{-t},
	\end{align}
	where $\{\vecz_i\}_{i=1}^n$ is a sequence of i.i.d. random vectors sampled from the same distribution as $\{\vecx_i\}_{i=1}^n$. 
	Let $t = o\log (ne/o)+\log(1/\delta)$. From Stirling's formula, we have $\binom{n}{o} e^{-t}  \leq \delta$. From \eqref{ine:p:oq-06}, with probability at least $1-\delta$, we have
	\begin{align}
	\label{ine:p:oq-07}
		\frac{1}{n}\times \lefteqn{ \sup_{\vecv \in r_{\veclambda} \mbb{B}_\Sigma^d \cap \mbb{S}_\Sigma^d}\left| \sum_{i \in \mc{O}} (\vecx_i^\top\vecv)^2  -\mbb{E}(\vecx_i^\top \vecv)^2\right| } \nonumber \\
		& \quad\lesssim  \frac{o}{n}L^2 \left(\frac{(\rho\sqrt{n}r_{\veclambda}+1)^2}{o}+\frac{\rho\sqrt{n}r_{\veclambda}+1}{\sqrt{o}}+\sqrt{\frac{o\log \frac{ne}{o}+\log(1/\delta)}{o}}r_\Sigma+\frac{o\log \frac{ne}{o}+\log(1/\delta)}{o} \right)\nonumber\\
		& \quad\stackrel{(a)}{\lesssim}  L^2 \left(  \frac{(\rho\sqrt{n}r_{\veclambda}+1)^2}{n}+\sqrt{\frac{o}{n}}\frac{\rho\sqrt{n}r_{\veclambda}+1}{\sqrt{n}}+\frac{\log(1/\delta)}{n}+\frac{o}{n}\log\frac{n}{o}\right)\nonumber\\
		& \quad \stackrel{(b)}{\lesssim}  L^2 \left( \frac{(\rho\sqrt{n}r_{\veclambda}+1)^2}{n}+\frac{\log(1/\delta)}{n}+\frac{o}{n}\log\frac{n}{o}\right),
	\end{align}
	where (a) follows from  $e \le n/o$ and $r_\delta \le 1$, and (b) follows from Young's inequality and $e \le n/o$. 
	Then, from the triangular inequality, we have
	\begin{align}
		\label{ine:p:oq-08}
			\sup_{\vecv \in r_{\veclambda}\mbb{B}_\Sigma^d \cap \mbb{S}_\Sigma^d}\sqrt{\frac{1}{n}\times \left| \sum_{i\in \mc{O}} (\vecx_i^\top\vecv)^2  -\mbb{E}(\vecx_i^\top \vecv)^2\right|  }\lesssim  L \left(\frac{\rho\sqrt{n}r_{\veclambda}+1}{\sqrt{n}}+\sqrt{\frac{o}{n}\log\frac{n}{o}}+\sqrt{\frac{\log(1/\delta)}{n}}\right).
	\end{align}

	\noindent
	\underline{Second step: Applying peeling device}

	Lastly, we apply Lemma 5 of \cite{DalTho2019Outlier} and $\log(1/\delta)\geq \log 9$, for any $\vecv \in r_{\veclambda}\mbb{B}_\Sigma^d \cap \mbb{S}_\Sigma^d$, with probability at least $1-\delta$, we have
	\begin{align}
		\label{ine:p:oq-09}
		\sqrt{\frac{1}{n}\times \left| \sum_{i  \in\mc{O}} (\vecx_i^\top\vecv)^2  -\mbb{E}(\vecx_i^\top \vecv)^2\right|  }\lesssim  L \left(\rho\|\vecv\|_{\veclambda}+\sqrt{\frac{o}{n}\log\frac{n}{o}}+\sqrt{\frac{\log(1/\delta)}{n}}\right),
	\end{align}
	and 
	\begin{align}
		\label{ine:p:oq-10}
		\sqrt{\frac{o}{n}}\sqrt{\frac{1}{n}\times \left| \sum_{i  \in\mc{O}} (\vecx_i^\top\vecv)^2  -\mbb{E}(\vecx_i^\top \vecv)^2\right|  }\lesssim  L \left(\rho\|\vecv\|_{\veclambda}+\frac{o}{n}\sqrt{\log\frac{n}{o}}+\sqrt{\frac{\log(1/\delta)}{n}}\right).
	\end{align}
	After normalizing $\vecv = \vecv/\|\vecv\|_\Sigma$, the proof is complete.

	\subsection{Proof of Proposition \ref{p:main-abs}}
		Define 
		\begin{align}
			\chi(x) :=\begin{cases}
			1 & \mbox{ if } x> 1\\
			2x & \mbox{ if } 1/2\leq x \leq 1 \\
			0 & \mbox{ if } x <1/2
			\end{cases}\text{ and }\mr{c}_4 = 24. 
		\end{align}
		Then, we have
		\begin{align}
		\sum_{k=1}^{\mr{B}}\mathbb{I}_{\left|\frac{\tilde{\varsigma}_k-\mbb{E}\frac{1}{n}\sum_{i=1}^n|\xi_i + \vecx_i^\top \vecv|}{\mr{c}_4(\sigma^*+\eta)/\sqrt{N}} \geq 1\right|} \leq \sum_{k=1}^{\mr{B}} \chi\left(\frac{|\tilde{\varsigma}_k-\mbb{E}\frac{1}{n}\sum_{i=1}^n|\xi_i + \vecx_i^\top \vecv||}{\mr{c}_4(\sigma^*+\eta)/\sqrt{N}}\right),
		\end{align}
		and decomposing  this, we have
		\begin{align}
			&\sum_{k=1}^{\mr{B}}\mathbb{I}_{\left|\frac{\tilde{\varsigma}_k-\mbb{E}\frac{1}{n}\sum_{i=1}^n|\xi_i + \vecx_i^\top \vecv|}{mr{c}_4(\sigma^*+\eta)/\sqrt{N}} \geq 1\right|}\nonumber\\
			&\quad\quad\leq \sup_{\vecv,\eta \in \tilde{\mr{V}}_{\mr{abs}}}\left(\sum_{k=1}^{\mr{B}}\chi\left(\frac{|\tilde{\varsigma}_k-\mbb{E}\frac{1}{n}\sum_{i=1}^n|\xi_i + \vecx_i^\top \vecv||}{\mr{c}_4(\sigma^*+\eta)/\sqrt{N}}\right)-\mbb{E}\sum_{k=1}^{\mr{B}} \chi\left(\frac{|\tilde{\varsigma}_k-\mbb{E}\frac{1}{n}\sum_{i=1}^n|\xi_i + \vecx_i^\top \vecv||}{\mr{c}_4(\sigma^*+\eta)/\sqrt{N}}\right)\right)\nonumber\\
			&\quad \quad\quad +\sup_{\vecv,\eta \in \tilde{\mr{V}}_{\mr{abs}}} \mbb{E}\sum_{k=1}^{\mr{B}} \chi\left(\frac{|\tilde{\varsigma}_k-\mbb{E}\frac{1}{n}\sum_{i=1}^n|\xi_i + \vecx_i^\top \vecv||}{\mr{c}_4(\sigma^*+\eta)/\sqrt{N}}\right)\nonumber\\
			&\quad\quad= \sup_{\vecv,\eta \in \tilde{\mr{V}}_{\mr{abs}}}\left(\sum_{k=1}^{\mr{B}}\chi\left(\frac{|\tilde{\varsigma}_k-|\xi_i + \vecx_i^\top \vecv||}{\mr{c}_4(\sigma^*+\eta)/\sqrt{N}}\right)-\mbb{E}\sum_{k=1}^{\mr{B}} \chi\left(\frac{|\tilde{\varsigma}_k-|\xi_i + \vecx_i^\top \vecv||}{\mr{c}_4(\sigma^*+\eta)/\sqrt{N}}\right)\right)+\sup_{\vecv,\eta \in \tilde{\mr{V}}_{\mr{abs}} } \mbb{E}\sum_{k=1}^{\mr{B}} \chi\left(\frac{|\tilde{\varsigma}_k-\mbb{E}|\xi_i + \vecx_i^\top \vecv||}{\mr{c}_4(\sigma^*+\eta)/\sqrt{N}}\right)\nonumber\\
		\end{align}

	\noindent
	\underline{Evaluation of $\sup_{\vecv,\eta \in \tilde{\mr{V}}_{\mr{abs}} }  \mbb{E}\sum_{k=1}^{\mr{B}} \chi\left(\frac{|\tilde{\varsigma}_k-\mbb{E}|\xi_i + \vecx_i^\top \vecv||}{\mr{c}_4(\sigma^*+\eta)/\sqrt{N}}\right)$} 

	First, from definition of $\chi(\cdot)$ and Markov's inequality, we have
	\begin{align}
		\mbb{E}\sum_{k=1}^{\mr{B}} \chi\left(\frac{|\tilde{\varsigma}_k-\mbb{E}|\xi_i + \vecx_i^\top \vecv||}{\mr{c}_4(\sigma^*+\eta)/\sqrt{N}}\right)\leq \mbb{E}\sum_{k=1}^{\mr{B}}\mathbb{I}_{\frac{2|\tilde{\varsigma}_k-\mbb{E}|\xi_i + \vecx_i^\top \vecv||}{ \mr{c}_4(\sigma^*+\eta)/\sqrt{N}}\geq 1}  \leq 4\mr{B}\mbb{E}\left|\frac{\tilde{\varsigma}_k-\mbb{E}|\xi_i + \vecx_i^\top \vecv|}{\mr{c}_4(\sigma^*+\eta)/\sqrt{N}}\right|^2.
	\end{align}
	Then, for any fixed $\vecv \in \tilde{\mr{V}}_{\mr{abs}}$, we have
	\begin{align}
		\left(\mbb{E}\left|\frac{\tilde{\varsigma}_k-\mbb{E}|\xi_i + \vecx_i^\top \vecv|}{\mr{c}_4(\sigma^*+\eta)/\sqrt{N}}\right|^2\right)^\frac{1}{2}&=\left(\mbb{E}\left|\frac{\frac{1}{N}\sum_{i=(k-1)N+1}^{kN}(|\xi_i+\vecx_i^\top \vecv|-\mbb{E}|\xi_i + \vecx_i^\top \vecv|)}{\mr{c}_4(\sigma^*+\eta)/\sqrt{N}}\right|^2\right)^\frac{1}{2}\nonumber\\
		&=\frac{1}{N}\left(\mbb{E}\frac{\sum_{i=(k-1)N+1}^{kN}(|\xi_i+\vecx_i^\top \vecv|-\mbb{E}|\xi_i + \vecx_i^\top \vecv|)^2}{(\mr{c}_4(\sigma^*+\eta)/\sqrt{N})^2}\right)^\frac{1}{2}\nonumber\\
		&=\frac{1}{N}\left(N\mbb{E}\frac{(|\xi+\vecx^\top \vecv|-\mbb{E}|\xi + \vecx^\top \vecv|)^2}{(\mr{c}_4(\sigma^*+\eta)/\sqrt{N})^2}\right)^\frac{1}{2}\nonumber\\
		&\leq \frac{1}{\sqrt{N}}\left(\mbb{E}\left|\frac{\xi + \vecx^\top \vecv}{\mr{c}_4(\sigma^*+\eta)/\sqrt{N}}\right|^2\right)^\frac{1}{2}\nonumber\\
		&\leq \frac{\left(\mbb{E}|\vecx^\top \vecv|^2\right)^\frac{1}{2}}{\mr{c}_4\eta}+\frac{\sigma^*}{\mr{c}_4\sigma^*}= \frac{\|\vecv\|_\Sigma}{\mr{c}_4\eta}+\frac{1}{\mr{c}_4}\leq \frac{2}{\mr{c}_4} .
	\end{align}

		\noindent
		\underline{Evaluation of $\sup_{\vecv,\eta \in \tilde{\mr{V}}_{\mr{abs}}}\left(\sum_{k=1}^{\mr{B}}\chi\left(\frac{|\tilde{\varsigma}_k-\mbb{E}|\xi_i + \vecx_i^\top \vecv||}{\mr{c}_4(\sigma^*+\eta)/\sqrt{N}}\right)-\mbb{E}\sum_{k=1}^{\mr{B}} \chi\left(\frac{|\tilde{\varsigma}_k-\mbb{E}|\xi_i + \vecx_i^\top \vecv||}{\mr{c}_4(\sigma^*+\eta)/\sqrt{N}}\right)\right)$} 

		From Theorem 2.3 of \cite{Bou2002Bennett} and symmetrization (Lemma 6.4.2 of \cite{Ver2018High}), with probability at least $1-\delta$, 
		\begin{align}
			&\sup_{\vecv,\eta \in \tilde{\mr{V}}_{\mr{abs}}}\left| \sum_{k=1}^{\mr{B}}\chi\left(\frac{|\tilde{\varsigma}_k-\mbb{E}|\xi_i + \vecx_i^\top \vecv||}{\mr{c}_4(\sigma^*+\eta)/\sqrt{N}}\right)-\mbb{E}\sum_{k=1}^{\mr{B}} \chi\left(\frac{|\tilde{\varsigma}_k-\mbb{E}|\xi_i + \vecx_i^\top \vecv||}{\mr{c}_4(\sigma^*+\eta)/\sqrt{N}}\right)\right|\nonumber\\
			&\quad\quad\leq 2\mbb{E}\sup_{\vecv,\eta \in\tilde{\mr{V}}_{\mr{abs}}}\left|\sum_{k=1}^{\mr{B}} \epsilon_i\chi\left(\frac{|\tilde{\varsigma}_k-\mbb{E}|\xi_i + \vecx_i^\top \vecv||}{\mr{c}_4(\sigma^*+\eta)/\sqrt{N}}\right)\right| +\sqrt{2\log(1/\delta)\mr{B}}+2\log(1/\delta).
		\end{align}
		For $ 2\mbb{E}\sup_{\vecv,\eta \in \tilde{\mr{V}}_{\mr{abs}}}\left|\sum_{k=1}^{\mr{B}} \epsilon_i\chi\left(\frac{|\tilde{\varsigma}_k-\mbb{E}|\xi_i + \vecx_i^\top \vecv||}{\mr{c}_4(\sigma^*+\eta)/\sqrt{N}}\right)\right|$, we have
		\begin{align}
			2\mbb{E}\sup_{\vecv,\eta \in\tilde{\mr{V}}_{\mr{abs}}}\left|\sum_{k=1}^{\mr{B}} \epsilon_i \chi\left(\frac{|\tilde{\varsigma}_k-\mbb{E}|\xi_i + \vecx_i^\top \vecv||}{\mr{c}_4(\sigma^*+\eta)/\sqrt{N}}\right)\right|
			&\stackrel{(a)}{\leq} 4\mbb{E}\sup_{\vecv,\eta \in \tilde{\mr{V}}_{\mr{abs}}}\left|\sum_{k=1}^{\mr{B}}\epsilon_i \frac{\tilde{\varsigma}_k-\mbb{E}|\xi_i + \vecx_i^\top \vecv|}{\mr{c}_4(\sigma^*+\eta)/\sqrt{N}}\right|\nonumber\\
			&\stackrel{(b)}{\leq} 8\mbb{E}\sup_{\vecv,\eta \in \tilde{\mr{V}}_{\mr{abs}}}\left|\sum_{k=1}^{\mr{B}}\frac{\tilde{\varsigma}_k-\mbb{E}|\xi_i + \vecx_i^\top \vecv|}{\mr{c}_4(\sigma^*+\eta)/\sqrt{N}}-\mbb{E} \sum_{k=1}^{\mr{B}}\frac{\tilde{\varsigma}_k-\mbb{E}|\xi_i + \vecx_i^\top \vecv|}{\mr{c}_4(\sigma^*+\eta)/\sqrt{N}}\right|\nonumber\\
			&= 8\mbb{E}\sup_{\vecv,\eta \in \tilde{\mr{V}}_{\mr{abs}}}\left|\sum_{k=1}^{\mr{B}}\frac{\tilde{\varsigma}_k}{\mr{c}_4(\sigma^*+\eta)/\sqrt{N}}-\mbb{E} \sum_{k=1}^{\mr{B}}\frac{\tilde{\varsigma}_k}{\mr{c}_4(\sigma^*+\eta)/\sqrt{N}}\right|\nonumber\\
			&= \frac{8}{N}\mbb{E}\sup_{\vecv,\eta \in \tilde{\mr{V}}_{\mr{abs}}}\left|\sum_{i=1}^n\frac{|\xi_i + \vecx_i^\top \vecv|}{\mr{c}_4(\sigma^*+\eta)/\sqrt{N}}-\mbb{E} \sum_{i=1}^n\frac{|\xi_i + \vecx_i^\top \vecv|}{\mr{c}_4(\sigma^*+\eta)/\sqrt{N}}\right|\nonumber\\
			&\stackrel{(c)}{\leq } \frac{16}{N}\mbb{E}\sup_{\vecv,\eta \in \tilde{\mr{V}}_{\mr{abs}}}\left|\sum_{i=1}^n \epsilon_i\frac{|\xi_i + \vecx_i^\top \vecv|}{\mr{c}_4(\sigma^*+\eta)/\sqrt{N}}\right|\nonumber\\
			&\stackrel{(d)}{\leq } \frac{16}{\sqrt{N}\mr{c}_4(\sigma^*+\eta)}\left(2\mbb{E}\sup_{\vecv,\eta \in \tilde{\mr{V}}_{\mr{abs}}}\sum_{i=1}^n \epsilon_i(\xi_i + \vecx_i^\top \vecv)+\mbb{E}\left|\sum_{i=1}^n\epsilon_i\xi_i\right|\right)\nonumber\\
			&\leq \frac{32}{\sqrt{N}\mr{c}_4}\mbb{E}\sup_{\vecv,\eta \in \tilde{\mr{V}}_{\mr{abs}}}\sum_{i=1}^n \epsilon_i\vecx_i^\top \frac{\vecv}{\eta}+ \sqrt{\frac{n}{N}}\frac{48}{\mr{c}_4}\nonumber\\
			&\stackrel{(e)}{\leq} \sqrt{4e}\sqrt{\frac{n}{N}}\frac{32}{\mr{c}_4} L(\sqrt{n}\mr{R}_{\veclambda}+\sqrt{s\log(ed/s)})+ \sqrt{\frac{n}{N}}\frac{48}{\mr{c}_4}\nonumber\\
			&\stackrel{(f)}{\leq} \mr{B}\frac{\sqrt{n}}{\mr{c}_4\mr{c}_\xi^2}+ \mr{B}\frac{48}{\mr{c}_4},
		\end{align}
		where (a) follow from Theorem 11.6 of \cite{BouLugMas2013concentration}, (b) follows from 6.4.2 of \cite{Ver2018High}, (c) follows from Exercise 6.4.4 of \cite{Ver2018High},  (d) follows from Exercise 2.2.2 of \cite{Tal2014Upper}, (e) follows from Lemma \ref{l:gw2}, and (f) follows from $L\left(\mr{R}_{\veclambda}+\sqrt{\frac{s\log(ed/s)}{n}}\right)\leq \frac{1}{64\sqrt{e} \mr{C}_{\mr{B}}\mr{c}_\xi}$.

		\noindent
		\underline{Combining the arguments} 

		Combining the arguments above, with probability at least $1-\delta$, we have
		\begin{align}
			\sum_{k=1}^{\mr{B}}\mathbb{I}_{\left|\frac{\tilde{\varsigma}_k-\mbb{E}\frac{1}{n}\sum_{i=1}^n|\xi_i + \vecx_i^\top \vecv|}{\mr{c}_4(\sigma^*+\eta)/\sqrt{N}} \geq 1\right|}&\leq \frac{8}{\mr{c}_4}\mr{B}+ \sqrt{2\log(1/\delta)\mr{B}}+2\log(1/\delta)+\sqrt{\mr{B}}\frac{\sqrt{n}}{\mr{c}_4\mr{c}_\xi\mr{C}_{\mr{B}}}+ \sqrt{\mr{B}}\frac{48}{\mr{c}_4}\nonumber\\
			&\stackrel{(a)}{\leq}\frac{8}{\mr{c}_4}\mr{B}+ \frac{\sqrt{2n\mr{B}}}{\mr{c}_\xi\mr{C}_{\mr{B}}}+2\frac{n}{\mr{c}_\xi^2\mr{C}_{\mr{B}}^2}+\sqrt{\mr{B}}\frac{\sqrt{n}}{\mr{c}_4\mr{c}_\xi\mr{C}_{\mr{B}}}+ \sqrt{\mr{B}}\frac{48}{\mr{c}_4}\nonumber\\
			&\stackrel{(b)}{\leq}\left(\frac{8}{\mr{c}_4}+ \frac{\sqrt{2}}{\sqrt{\mr{C}_{\mr{B}}}}+\frac{2}{\mr{C}_{\mr{B}}}+\frac{1}{\mr{c}_4\mr{C}_{\mr{B}}}+ \frac{1}{\mr{c}_4}\right)\mr{B}\nonumber\\
			&\stackrel{(c)}{\leq}\left(\frac{8}{24}+ \frac{1}{24}+\frac{1}{24^2}+\frac{1}{2\times 24^3}+ \frac{1}{24}\right)\mr{B}= \frac{11}{24}\mr{B}.
		\end{align}
		where (a) follows from $\frac{\log(1/\delta)}{n}\leq \frac{1}{\mr{c}_\xi \mr{C}^2_{\mr{B}}}$,  (b) follows from $\mr{B}=\frac{n}{\mr{c}_\xi^2 \mr{C}_{\mr{B}}}$ and $\sqrt{\mr{B}}\geq 48$, and (c) follows from $\mr{c}_4=24$ and $\mr{C}_{\mr{B}}\geq 2\times 24^2$. Consequently, we have
		\begin{align}
			\mbb{P}\left(\sum_{k=1}^{\mr{B}}\mathbb{I}_{\left|\frac{\tilde{\varsigma}_k-\mbb{E}\frac{1}{n}\sum_{i=1}^n|\xi_i + \vecx_i^\top \vecv|}{24(\sigma^*+\eta)/\sqrt{\mr{c}_\xi \mr{C}_{\mr{B}}}} \geq 1\right|}\leq \frac{11}{24}\mr{B}\right)\geq 1-\delta.
		\end{align}

	\subsection{Proof of Corollary \ref{c:main-mean}}
	This proof is almost the same as the one of Proposition \ref{p:main-abs}. In this proof, set $\mr{c}_4=48$.
	We have
		\begin{align}
			&\sum_{k=1}^{\mr{B}}\mathbb{I}_{\left|\frac{\tilde{\mu}_k-\mbb{E}\frac{1}{n}\sum_{i=1}^n(\xi_i + \vecx_i^\top \vecv+\mu^*)}{\mr{c}_4(\sigma^*+\eta)/\sqrt{N}} \geq 1\right|}\nonumber\\
			&\quad\quad= \sup_{\vecv,\eta \in \tilde{\mr{V}}_{\mr{abs}}}\left(\sum_{k=1}^{\mr{B}}\chi\left(\frac{|\tilde{\mu}_k-(\xi_i + \vecx_i^\top \vecv+\mu^*)|}{\mr{c}_4(\sigma^*+\eta)/\sqrt{N}}\right)-\mbb{E}\sum_{k=1}^{\mr{B}} \chi\left(\frac{|\tilde{\mu}_k-(\xi_i + \vecx_i^\top \vecv+\mu^*)|}{\mr{c}_4(\sigma^*+\eta)/\sqrt{N}}\right)\right)\nonumber\\
			&\quad\quad \quad\quad+\sup_{\vecv,\eta \in \tilde{\mr{V}}_{\mr{abs}} } \mbb{E}\sum_{k=1}^{\mr{B}} \chi\left(\frac{|\tilde{\mu}_k-\mbb{E}(\xi_i + \vecx_i^\top \vecv+\mu^*)|}{\mr{c}_4(\sigma^*+\eta)/\sqrt{N}}\right).
		\end{align}

	\noindent
	\underline{Bounding $\sup_{\vecv,\eta \in \tilde{\mr{V}}_{\mr{abs}} }  \mbb{E}\sum_{k=1}^{\mr{B}} \chi\left(\frac{|\tilde{\mu}_k-\mbb{E}(\xi_i + \vecx_i^\top \vecv+\mu^*)|}{\mr{c}_4(\sigma^*+\eta)/\sqrt{N}}\right)$} 

	First, from definition of $\chi(\cdot)$ and Markov's inequality, we have
	\begin{align}
		\mbb{E}\sum_{k=1}^{\mr{B}} \chi\left(\frac{|\tilde{\mu}_k-\mbb{E}(\xi_i + \vecx_i^\top \vecv+\mu^*)|}{\mr{c}_4(\sigma^*+\eta)/\sqrt{N}}\right)\leq \mbb{E}\sum_{k=1}^{\mr{B}}\mathbb{I}_{\frac{2|\tilde{\mu}_k-\mbb{E}(\xi_i + \vecx_i^\top \vecv+\mu^*)|}{ \mr{c}_4(\sigma^*+\eta)/\sqrt{N}}\geq 1}  \leq 4\mr{B}\mbb{E}\left|\frac{\tilde{\mu}_k-\mbb{E}(\xi_i + \vecx_i^\top \vecv+\mu^*)}{\mr{c}_4(\sigma^*+\eta)/\sqrt{N}}\right|^2.
	\end{align}
	Then, for any fixed $\vecv \in \tilde{\mr{V}}_{\mr{abs}}$, we have
	\begin{align}
		\left(\mbb{E}\left|\frac{\tilde{\mu}_k-\mbb{E}(\xi_i + \vecx_i^\top \vecv+\mu^*)}{\mr{c}_4(\sigma^*+\eta)/\sqrt{N}}\right|^2\right)^\frac{1}{2}&=\left(\mbb{E}\left|\frac{\frac{1}{N}\sum_{i=(k-1)N+1}^{kN}((\xi_i+\vecx_i^\top \vecv+\mu^*)-\mbb{E}(\xi_i + \vecx_i^\top \vecv+\mu^*))}{\mr{c}_4(\sigma^*+\eta)/\sqrt{N}}\right|^2\right)^\frac{1}{2}\nonumber\\
		&=\frac{1}{N}\left(\mbb{E}\frac{\sum_{i=(k-1)N+1}^{kN}(\xi_i+\vecx_i^\top \vecv)^2}{(\mr{c}_4(\sigma^*+\eta)/\sqrt{N})^2}\right)^\frac{1}{2}\nonumber\\
		&\leq \frac{\left(\mbb{E}|\vecx^\top \vecv|^2\right)^\frac{1}{2}}{\mr{c}_4\eta}+\frac{\sigma^*}{\mr{c}_4\sigma^*}\leq \frac{2}{\mr{c}_4} .
	\end{align}

		\noindent
		\underline{Bounding $\sup_{\vecv,\eta \in \tilde{\mr{V}}_{\mr{abs}}}\left(\sum_{k=1}^{\mr{B}}\chi\left(\frac{|\tilde{\mu}_k-(\xi_i + \vecx_i^\top \vecv+\mu^*)|}{\mr{c}_4(\sigma^*+\eta)/\sqrt{N}}\right)-\mbb{E}\sum_{k=1}^{\mr{B}} \chi\left(\frac{|\tilde{\mu}_k-(\xi_i + \vecx_i^\top \vecv+\mu^*)|}{\mr{c}_4(\sigma^*+\eta)/\sqrt{N}}\right)\right)$} 

		From Theorem 2.3 of \cite{Bou2002Bennett} and symmetrization (Lemma 6.4.2 of \cite{Ver2018High}), with probability at least $1-\delta$, 
		\begin{align}
			&\sup_{\vecv,\eta \in \tilde{\mr{V}}_{\mr{abs}}}\left| \sum_{k=1}^{\mr{B}}\chi\left(\frac{|\tilde{\mu}_k-\mbb{E}(\xi_i + \vecx_i^\top \vecv+\mu^*)|}{\mr{c}_4(\sigma^*+\eta)/\sqrt{N}}\right)-\mbb{E}\sum_{k=1}^{\mr{B}} \chi\left(\frac{|\tilde{\mu}_k-\mbb{E}(\xi_i + \vecx_i^\top \vecv+\mu^*)|}{\mr{c}_4(\sigma^*+\eta)/\sqrt{N}}\right)\right|\nonumber\\
			&\quad\quad\leq 2\mbb{E}\sup_{\vecv,\eta \in\tilde{\mr{V}}_{\mr{abs}}}\left|\sum_{k=1}^{\mr{B}} \epsilon_i\chi\left(\frac{|\tilde{\mu}_k-\mbb{E}|(\xi_i + \vecx_i^\top \vecv +\mu^*|}{\mr{c}_4(\sigma^*+\eta)/\sqrt{N}}\right)\right| +\sqrt{2\log(1/\delta)\mr{B}}+2\log(1/\delta).
		\end{align}
		For $ 2\mbb{E}\sup_{\vecv,\eta \in \tilde{\mr{V}}_{\mr{abs}}}\left|\sum_{k=1}^{\mr{B}} \epsilon_i\chi\left(\frac{|\tilde{\mu}_k-\mbb{E}(\xi_i + \vecx_i^\top \vecv+\mu^*)|}{\mr{c}_4(\sigma^*+\eta)/\sqrt{N}}\right)\right|$, we have
		\begin{align}
			2\mbb{E}\sup_{\vecv,\eta \in\tilde{\mr{V}}_{\mr{abs}}}\left|\sum_{k=1}^{\mr{B}} \epsilon_i \chi\left(\frac{|\tilde{\mu}_k-\mbb{E}(\xi_i + \vecx_i^\top \vecv+\mu^*)|}{\mr{c}_4(\sigma^*+\eta)/\sqrt{N}}\right)\right|
			&\stackrel{(a)}{\leq} 4\mbb{E}\sup_{\vecv,\eta \in \tilde{\mr{V}}_{\mr{abs}}}\left|\sum_{k=1}^{\mr{B}}\epsilon_i \frac{\tilde{\mu}_k-\mbb{E}(\xi_i + \vecx_i^\top \vecv+\mu^*)}{\mr{c}_4(\sigma^*+\eta)/\sqrt{N}}\right|\nonumber\\
			&=4\mbb{E}\sup_{\vecv,\eta \in \tilde{\mr{V}}_{\mr{abs}}}\left|\sum_{k=1}^{\mr{B}}\epsilon_i \frac{\tilde{\mu}_k-\mu^*}{\mr{c}_4(\sigma^*+\eta)/\sqrt{N}}\right|\nonumber\\
			&\stackrel{(b)}{\leq } \frac{4}{N}\mbb{E}\sup_{\vecv,\eta \in \tilde{\mr{V}}_{\mr{abs}}}\left|\sum_{i=1}^n \epsilon_i\frac{\xi_i + \vecx_i^\top \vecv}{\mr{c}_4(\sigma^*+\eta)/\sqrt{N}}\right|\nonumber\\
			&\stackrel{(c)}{\leq } \frac{4}{\sqrt{N}\mr{c}_4(\sigma^*+\eta)}\left(2\mbb{E}\sup_{\vecv,\eta \in \tilde{\mr{V}}_{\mr{abs}}}\sum_{i=1}^n \epsilon_i(\xi_i + \vecx_i^\top \vecv)+\mbb{E}\left|\sum_{i=1}^n\epsilon_i\xi_i\right|\right)\nonumber\\
			&\leq \frac{8}{\sqrt{N}\mr{c}_4}\mbb{E}\sup_{\vecv,\eta \in \tilde{\mr{V}}_{\mr{abs}}}\sum_{i=1}^n \epsilon_i\vecx_i^\top \frac{\vecv}{\eta}+ \sqrt{\frac{n}{N}}\frac{6}{\mr{c}_4}\nonumber\\
			&\stackrel{(d)}{\leq} \sqrt{\frac{n}{N}}\frac{8}{\mr{c}_4} L(\sqrt{n}\mr{R}_{\veclambda}+\sqrt{s\log(ed/s)})+ \sqrt{\frac{n}{N}}\frac{6}{\mr{c}_4}\nonumber\\
			&\stackrel{(f)}{\leq} \mr{B}\frac{\sqrt{n}}{\mr{c}_4}+ \mr{B}\frac{6}{\mr{c}_4},
		\end{align}
		where (a) follows from Theorem 11.6 of \cite{BouLugMas2013concentration},  (b) follows from Exercise 6.4.4 of \cite{Ver2018High},  (c) follows from Exercise 2.2.2 of \cite{Tal2014Upper}, (d) follows from Lemma \ref{l:gw2}, and (f) follows from the assumption $L\left(\mr{R}_{\veclambda}+\sqrt{\frac{s\log(ed/s)}{n}}\right)\leq 1/8$.

		\noindent
		\underline{Combining the arguments} 

		Combining the arguments above, with probability at least $1-\delta$, we have
		\begin{align}
			\sum_{k=1}^{\mr{B}}\mathbb{I}_{\left|\frac{\tilde{\varsigma}_k-\mbb{E}\frac{1}{n}\sum_{i=1}^n|\xi_i + \vecx_i^\top \vecv|}{\mr{c}_4(\sigma^*+\eta)/\sqrt{N}} \geq 1\right|}
			&\leq \frac{8}{\mr{c}_4}\mr{B}+ \sqrt{2\log(1/\delta)\mr{B}}+2\log(1/\delta)+\sqrt{\mr{B}}\frac{\sqrt{n}}{\mr{c}_4}+ \sqrt{\mr{B}}\frac{6}{\mr{c}_4}\nonumber\\
			&\stackrel{(a)}{\leq}\frac{8}{\mr{c}_4}\mr{B}+ \frac{\mr{B}}{8}+\frac{\mr{B}}{64}+\sqrt{\mr{B}}\frac{\sqrt{n}}{\mr{c}_4}+ \sqrt{\mr{B}}\frac{6}{\mr{c}_4}\nonumber\\
			&\stackrel{(b)}{\leq} \left(\frac{8}{\mr{c}_4}+ \frac{1}{8}+\frac{1}{64}+\frac{2}{\mr{c}_4}+ \frac{1}{\mr{c}_4}\right)\mr{B} \nonumber\\
			&\stackrel{(c)}{\leq}\frac{18}{48}\mr{B}.
		\end{align}
		where (a) follows from the assumption that $\log(1/\delta)\leq \mr{B}/128$,  (b) follows from the assumptions that  $\mr{B}\geq n/4$, $\sqrt{\mr{B}}\geq 6$ and  $\mr{c}_4=24$,  and  (c) follows from $\mr{c}_4=48$. Consequently, we have
		\begin{align}
			\mbb{P}\left(\sum_{k=1}^{\mr{B}}\mathbb{I}_{\left|\frac{\tilde{\mu}_k-\mbb{E}\frac{1}{n}\sum_{i=1}^n(\xi_i + \vecx_i^\top \vecv)}{48(\sigma^*+\eta)/\mr{N}} \geq 1\right|}\leq \frac{18}{48}\mr{B}\right)\geq 1-\delta.
		\end{align}

\section{Auxiliary Lemmas}
\label{sec:AT}
Before presenting the lemmas, we introduce a definition, which is an important quantity to evaluate empirical processes and frequently used in Section \ref{sec:AT}.
\begin{definition}[$\gamma_{\alpha}(T,d)$-functional, Section $2$ of \cite{Dir2015Tail}]
	\label{d:tg2}
		Let $(T, d)$ be a semi-metric space with $d(x, z) \leq d(x, y) +
		d(y, z)$ and $d(x,y) = d(y,x)$ for $x, y, z \in T$.
		A sequence $\mc{T} = \{T_m\}_{m\geq 0}$ with subsets of $T$ is said to be admissible if $|T_0|=1$ and $|T_m|\leq 2^{2^m}$ for all $m\geq 1$.
		For any $\alpha \in (0,\infty)$, the $\gamma_{\alpha}$-functional of $(T,d)$ is defined by
		\begin{align}
			\gamma_{\alpha} (T,d) = \inf_{\mc{T}} \sup_{t\in T}\sum_{m=0}^\infty2^{\frac{m}{\alpha}} \inf_{s\in T_m} d(t,s),
		\end{align}
		where infimum is taken over all admissible sequences $\mc{T} = \{T_m\}_{m\geq 0}$.
\end{definition}

For any vector $\vecv  = (\vecv_1,\cdots,\vecv_d)$, define $(\vecv_1^\sharp,\cdots,\vecv_d^\sharp)$ be a non-increasing rearrangement of $|\vecv_1|,\cdots,|\vecv_d|$.
\begin{corollary}[Lemma G.1 of \cite{BelLecTsy2018Slope}]
	\label{c:gw}
	Suppose that $\max_{1\leq j\leq d}\Sigma_{jj}\leq \rho^2$.
	Then, for any constant $\mr{c}$, we have
	\begin{align}
		\mbb{E}\sup_{\vecv \in \mr{c} \mbb{B}_\lambda^d \cap \mr{c} '\mbb{B}^{d}_\Sigma}\vecv^\top \Sigma^\frac{1}{2}\vecg^d \lesssim \rho\sqrt{n}\mr{c}+\mr{c}'.
	\end{align}
\end{corollary}
\begin{proof}
	This proof is almost the same to the one of Lemma G.1 of \cite{BelLecTsy2018Slope}. Define $\vecg'_j = (\Sigma^\frac{1}{2}\vecg^d)_j$, for $j=1,\cdots,d$. 
	Then, from the assumption of $\Sigma$, we see that the variance of $\vecg'_j$ is at most $\rho^2$. Then, from Proposition E.2 of \cite{BelLecTsy2018Slope}, it holds that the event 
	\begin{align}
		\mr{E}_0:= \left\{\max_{j=1,\cdots,d} \frac{\vecg'_j}{\rho\sqrt{\log(2d/j)}}\leq 4\right\}
	\end{align}
	holds with probability at least $1/2$. Under the event $\mr{E}_0$, for any $\vecv  \in \mr{c} \mbb{B}_\lambda^d \cap \mr{c} '\mbb{B}^{d}_\Sigma$,  we have
	\begin{align}
		\vecv^\top \Sigma^\frac{1}{2}\vecg^d = \sum_{j=1}^d \vecg'_j \vecv\leq  \sum_{j=1}^d \vecg'^{\sharp}_j\vecv^{\sharp}_j\leq 4\rho \sqrt{n}\sum_{i=1}^d \vecv^{\sharp}_j\sqrt{\frac{\log(2d/j)}{n}}\leq 4\rho\sqrt{n} \mr{c}.
	\end{align}
	Note that $\sup_{\vecv \in \mr{c} \mbb{B}_\lambda^d \cap \mr{c} '\mbb{B}^{d}_\Sigma} (\vecv^\top \Sigma^\frac{1}{2}\vecg)\leq 4\rho\sqrt{n}\mr{c}$ with probability at least $1/2$, and therefore, median of $\sup_{\vecv \in \mr{c} \mbb{B}_\lambda^d \cap \mr{c} '\mbb{B}^{d}_\Sigma} (\vecv^\top \Sigma^\frac{1}{2}\vecg)$ is at most $4\rho\sqrt{n}\mr{c}$. The restriction $\mr{c} '\mbb{B}^{d}_\Sigma$ implies $1$-Lipschitz continuity of $\sup_{\vecv \in \mr{c} \mbb{B}_\lambda^d \cap \mr{c} '\mbb{B}^{d}_\Sigma}(\vecv^\top \Sigma^\frac{1}{2}(\cdot)) /\mr{c}'$, and from Lemma A.3 of \cite{BelLecTsy2018Slope}, we have the conclusion.
\end{proof}
\begin{lemma}
	\label{l:gw2}
	Suppose that Assumptions \ref{a:cov} holds. Then, we have
	\begin{align}
		&\mbb{E}\sup_{\vecv \in \mr{c}\mbb{B}_\lambda^d\cap\mr{c} '\mbb{B}_\Sigma^{d}}\sum_{i=1}^n\epsilon_i\vecv^\top \vecx_i \lesssim \sqrt{n}L(\rho\mr{c}\sqrt{n}+\mr{c}'),\nonumber\\
		&\mbb{P}\left(\sup_{\vecv \in \mr{c}\mbb{B}_\lambda^d\cap\mr{c} '\mbb{B}_\Sigma^{d}}\sum_{i=1}^n\vecv^\top \vecx_i \lesssim \sqrt{n}L(\rho\mr{c}\sqrt{n}+\mr{c}')+\mr{c}'\sqrt{n\log(1/\delta)}\right)\geq 1-\delta,\nonumber\\
		&\mbb{E}\sup_{\vecv \in \mr{c}\mbb{B}_\lambda^d\cap\mr{c} '\mbb{B}_\Sigma^{d}}\sum_{i=1}^n\epsilon_i\vecv^\top \vecx_i \leq \sqrt{4en}\left(\mr{c}\sqrt{n} L+\mr{c}'L\sqrt{s\log(ed/s)}\right),\nonumber\\
		&\left|\mbb{E}\sup_{\vecv \in \mr{c}\mbb{B}_\lambda^d\cap\mr{c} '\mbb{B}_\Sigma^{d}}\sum_{i=1}^n\vecv^\top \vecx_i\right| \lesssim \sqrt{n}L(\rho\mr{c}\sqrt{n}+\mr{c}').
	\end{align}
\end{lemma}
\begin{proof}
	
	\noindent\underline{First and fourth inequality}

	We only prove the first inequality. The proof of the fourth inequality is almost the same as the first one in terms of Exercise 2.2.2 of \cite{Tal2014Upper}.
	Note that 
	\begin{align}
		\mbb{E}\sup_{\vecv \in \mr{c}\mbb{B}_\lambda^d\cap\mr{c} '\mbb{B}_\Sigma^{d}}\sum_{i=1}^n\epsilon_i(\vecv^\top \vecx_i) \leq 		\mbb{E}\sup_{\vecv \in \mr{c}\mbb{B}_\lambda^d\cap\mr{c} '\mbb{B}_\Sigma^{d}} \left|\sum_{i=1}^n\epsilon_i(\vecv^\top \vecx_i)\right|.
	\end{align}
	For any $\vecv_1,\vecv_2 \in \mr{c}\mbb{B}_\lambda^d \cap \mr{c}'\mbb{B}^d_\Sigma$, and $i \in \{1,\cdots,n\}$, from the fact $|\varepsilon_i|\leq 1$ and the $L$-subGaussian assumption, we have
	\begin{align}
		\left\|\epsilon_i \langle \vecx_i,\vecv_1\rangle-\epsilon_i \langle \vecx_i,\vecv_2\rangle\right\|_{\psi_2}\leq  \left\|\langle \vecx_i,\vecv_1-\vecv_2\rangle\right\|_{\psi_2}\leq L \|\vecv_1-\vecv_2\|_\Sigma,
	\end{align}
	and from the i.i.d. assumption on $\{\epsilon_i,\vecx_i\}_{i=1}^n$ and Theorem 2.6.2 of  \cite{Ver2018High}, we have
	\begin{align}
		\mbb{P}\left(\left|\sum_{i=1}^n\left(\epsilon_i \langle \vecx_i,\vecv_1\rangle-\epsilon_i \langle \vecx_i,\vecv_2\rangle\right)\right|\gtrsim L \sqrt{nu} \|\vecv_1-\vecv_2\|_\Sigma \right)\lesssim e^{-u}.
	\end{align}
	Also, we have
	\begin{align}
		\sup_{\vecv \in \mr{c}\mbb{B}_\lambda^d \cap\mr{c}' \mbb{B}^d_\Sigma}\left\|\epsilon_i \langle \vecx_i,\vecv\rangle\right\|_{\psi_2}\leq  L\mr{c}'.
	\end{align}
	Then, from Theorem 3.2 of \cite{Dir2015Tail} with $t_0= \bf{0}$, we have 
	\begin{align}
		\mbb{E}\sup_{\vecv \in \mr{c}\mbb{B}_\lambda^d \cap\mr{c}' \mbb{B}^d_\Sigma}\left|\sum_{i=1}^n\epsilon_i(\vecv^\top \vecx_i)\right|&\lesssim \gamma_2 (\mr{c}\mbb{B}_\lambda^d \cap \mr{c}'\mbb{B}^d_\Sigma,L\sqrt{n}\|\cdot\|_\Sigma)+L\mr{c}'\sqrt{n},
	\end{align}
	and from majorizing measure theorem (Theorem 2.4.1 \cite{Tal2014Upper}), we have
	\begin{align}
		\gamma_2 (\mr{c}\mbb{B}_\lambda^d \cap \mr{c}'\mbb{B}^d_\Sigma,L\sqrt{n}\|\cdot\|_\Sigma)\lesssim L\sqrt{n}\mbb{E}\sup_{\vecv \in \mr{c}\mbb{B}_\lambda^d \cap \mr{c}'\mbb{B}^d_\Sigma}\vecv^\top \Sigma^\frac{1}{2}\vecg^d \stackrel{(a)}{\leq}  L\sqrt{n} (\rho \sqrt{n} \mr{c}+\mr{c}'),
	\end{align}
	where (a) follows from Corollary \ref{c:gw}.
	Combining the arguments, the proof is complete.

	\noindent\underline{Second inequality}
	
	For any $\vecv_1,\vecv_2 \in \mr{c}\mbb{B}_\lambda^d \cap \mr{c}'\mbb{B}^d_\Sigma$, and $i \in \{1,\cdots,n\}$, from the $L$-subGaussian assumption, we have
	\begin{align}
		\left\|\langle \vecx_i,\vecv_1\rangle-\langle \vecx_i,\vecv_2\rangle\right\|_{\psi_2}\leq L \|\vecv_1-\vecv_2\|_\Sigma,
	\end{align}
	and from the i.i.d. assumption on $\{\epsilon_i,\vecx_i\}_{i=1}^n$ and Proposition 2.6.2 of  \cite{Ver2018High}, we have
	\begin{align}
		\mbb{P}\left(\left|\sum_{i=1}^n\left(\langle \vecx_i,\vecv_1\rangle-\langle \vecx_i,\vecv_2\rangle\right)\right|\gtrsim L \sqrt{nu} \|\vecv_1-\vecv_2\|_\Sigma \right)\lesssim e^{-u}.
	\end{align}
	Also, we have
	\begin{align}
		\left\|\langle \vecx_i,\vecv_1\rangle-\langle \vecx_i,\vecv_2\rangle\right\|_{\psi_2}\leq  2L\mr{c}',
	\end{align}
	and we have
	\begin{align}
		\mbb{P}\left(\left|\sum_{i=1}^n\left( \langle \vecx_i,\vecv_1\rangle-\langle \vecx_i,\vecv_2\rangle\right)\right|\gtrsim L \sqrt{nu} \mr{c}'\right)\lesssim e^{-u}.
	\end{align}
	Then, applying Theorem 3.2 of \cite{Dir2015Tail} with $t_0= \bf{0}$ and $\Delta_d(T) = 2L\mr{c}'$, the proof is complete.

	\noindent\underline{Third inequality}

	Define $\sum_{i=1}^n\epsilon_i \vecx_i := \vecX =(X_1,\cdots,X_d)^\top$.
	\begin{align}
		\sum_{i=1}^n\epsilon_i\vecv^\top \vecx_i \leq \sum_{j=1}^d |X_j^\sharp||v_j^\sharp|&\leq \sqrt{\sum_{j=1}^{s}|X_j^\sharp|^2|v_j^\sharp|^2} +\sum_{j=s+1}^{d}|X_j^\sharp||v_j^\sharp|\nonumber\\
		&\leq \sqrt{\sum_{j=1}^{s}|X_j^\sharp|^2}\|\vecv\|_2+  n \sum_{j=s+1}^{d} \frac{|X_j^\sharp|}{\sqrt{n\log(ed/s)}} \times \sqrt{\frac{\log(ed/s)}{n}}|v_j^\sharp|,
	\end{align}
	and from the definition of $\|\cdot\|_{\veclambda}$ and Jensen's inequality, we have
	\begin{align}
		\mbb{E} \sup_{\vecv \in \mr{c}\mbb{B}_\lambda^d \cap\mr{c}' \mbb{B}^d_\Sigma} \sum_{i=1}^n\epsilon_i\vecv^\top \vecx_i \leq \mr{c}'\sqrt{\mbb{E}\sum_{j=1}^{s}|X_j^\sharp|^2}+ \mr{c}n \max_{i=s+1,\cdots, d}\mbb{E} \frac{|X_i^\sharp|}{\sqrt{n\log(ed/k)}} ,
	\end{align}
	For a vector $\vecv$, define $\vecv|_j$ as its $j$-th coordinate. For any $j \in \{1,\cdots,d\}$ and $\lambda \in \mbb{R}$, we have
	We note that from the i.i.d. assumption on $\{\epsilon_i, \vecx_i\}_{i=1}^n$, Proposition 2.6.1 \cite{Ver2018High} and  \eqref{ine:d:l-02}, we have
	\begin{align}
		\mbb{E}\exp\left(\lambda \frac{X_i}{\sqrt{n}L}\right)&\stackrel{(a)}{=}\prod_{i=1}^n \exp\left(\lambda \frac{\epsilon_i \vecx_i |_j}{\sqrt{n}L}\right)\nonumber\\
		&\stackrel{(b)}{\leq} \prod_{i=1}^n e^{\lambda^2/n}= e^{\lambda^2},
	\end{align}
	where (a) follows from the i.i.d. assumption on $\{\epsilon_i, \vecx_i\}_{i=1}^n$, and (b) follows from \eqref{ine:d:l-02}.
	Then, from the last part of the argument of the proof of Proposition 2.5.2 of \cite{Ver2018High}, we have
	\begin{align}
		\mbb{P}\left(\left| \frac{X_j}{\sqrt{n}}\right|\geq t\right)\leq 2 e^{-t^2/4},
	\end{align}
	and from the first part of the  argument of the proof of Proposition 2.5.2 of \cite{Ver2018High}, we have
	\begin{align}
		\left\|\frac{X_j}{\sqrt{n}L}\right\|_{L_p}\leq 8\sqrt{p},
	\end{align}
	and from the second part of the  argument of the proof of Proposition 2.5.2 of \cite{Ver2018High}, we have
	\begin{align}
		\mbb{E}\exp\left(\frac{1}{4e} \left|\frac{X_j}{\sqrt{n}L}\right|^2\right) \leq e.
	\end{align}
	We note that 
	\begin{align}
		|X_k^\sharp|^2 \leq \frac{1}{k}\sum_{j=1}^k|X_j^\sharp|^2.
	\end{align}
	From almost the same procedure of the proof of Proposition E.1 of \cite{BelLecTsy2018Slope}, we have
	\begin{align}
		 \mbb{E}\frac{1}{k}\sum_{j=1}^k \frac{|X_j^\sharp|^2}{4enL^2} &\stackrel{(a)}{\leq} 4e \log \mbb{E}\exp\left(\frac{1}{k}\sum_{j=1}^k \frac{|X_j^\sharp|^2}{4enL^2}\right)\nonumber\\
		&\stackrel{(b)}{\leq} 4log \frac{1}{k}\sum_{j=1}^k\mbb{E}\exp\left(\frac{|X_j^\sharp|^2}{4enL^2}\right)\nonumber\\
		&\leq  \log \frac{1}{k}\sum_{j=1}^d\mbb{E}\exp\left(\frac{|X_j|^2}{4enL^2}\right)\leq  \log (ed/k),
	\end{align}
	where (a) and (b) follows from Jensen's inequality. Then, we have
	\begin{align}
		\sqrt{\mbb{E}\sum_{i=1}^{s}|X_i^\sharp|^2}\leq L\sqrt{4ens\log(ed/s)},\quad \max_{i=s+1,\cdots, d}\mbb{E} \frac{|X_i^\sharp|}{\sqrt{4en\log(ed/i)}} \leq L,
	\end{align}
	Combining the arguments, we have
	\begin{align}
		\mbb{E} \sup_{\vecv \in \mr{c}\mbb{B}_\lambda^d \cap\mr{c}' \mbb{B}^d_\Sigma} \sum_{i=1}^n\epsilon_i\vecv^\top \vecx_i \leq  \sqrt{4e}(\mr{c}n L+\mr{c}'L\sqrt{ns\log(ed/s)}).
	\end{align}
\end{proof}
\begin{lemma}
	\label{l:L1L2}
	Assume $\vecx \in \mbb{R}^d$ is an $L$-subGaussian random vector. Then, for any fixed $\vecv \in \mbb{R}$, we have
	\begin{align}
		\left(\mbb{E}|\vecx^\top\vecv|^2\right)^\frac{1}{2}\leq 72L^4\mbb{E}|\vecx^\top \vecv|.
	\end{align}
\end{lemma}
\begin{proof}
	From \eqref{ine:d:l-02}, 
	\begin{align}
		\left( \mbb{E}(\vecx^\top\vecv)^4\right)^\frac{1}{4}\leq 2L(\mbb{E}(\vecx^\top \vecv)^2)^\frac{1}{2}.
	\end{align}
	From this and Paley--Zygmund's inequality \cite{PenGin2012Decoupling}, for any $0<t<1$, we have 
	\begin{align}
		\mbb{P}(|\vecx^\top\vecv|\geq t(\mbb{E}(\vecx^\top \vecv)^2)^\frac{1}{2})\geq \left(\frac{1-t^2}{4L^2}\right)^2.
	\end{align}
	Then, from Proposition 1 of \cite{LecLer2017Robust}, we have 
	\begin{align}
		\left(\mbb{E}|\vecx^\top\vecv|^2\right)^\frac{1}{2}\leq 72L^4\mbb{E}|\vecx^\top \vecv|.
	\end{align}
\end{proof}
\begin{lemma}
	\label{l:L1ev}
	Suppose that Assumptions \ref{a:cov} and \ref{a:noise} hold. Let $\vecx$ and $\xi$ be random variables such that $\vecx$ has the same distribution as each $\vecx_i$ in $\{\vecx_i\}_{i=1}^n$, and $\xi$ has the same distribution as each $\xi_i$ in $\{\xi_i\}_{i=1}^n$. Then, for any $\vecv \in  \mr{c}\mbb{B}_{\veclambda}^d\cap \mr{c}'\mbb{B}_\Sigma^d$, we have
	\begin{align}
		\mbb{E}\frac{1}{n}\sum_{i=1}^n|\vecx_i^\top \vecv+\xi_i|& \leq C\frac{L(\rho \sqrt{n}\mr{c}+\mr{c}')+\sigma}{n}+\sup_{\vecv  \in \mr{c}\mbb{B}^d_\lambda \cap \mr{c}'\mbb{B}^d_\Sigma}\left(\mbb{E}|\vecx^\top \vecv+\xi|\right)\nonumber\\
		\mbb{E}\frac{1}{n}\sum_{i=1}^n|\vecx_i^\top \vecv+\xi_i|& \geq -C\frac{L(\rho \sqrt{n}\mr{c}+\mr{c}')+\sigma}{n}+\inf_{\vecv  \in \mr{c}\mbb{B}^d_\lambda \cap \mr{c}'\mbb{B}^d_\Sigma}\left(\mbb{E}|\vecx^\top \vecv+\xi|\right).
	\end{align}
\end{lemma}
\begin{proof}
	We only give the upper bound. The proof of the lower bound is almost the same to the one of the upper bound.
	In this proof, define
	\begin{align}
		\mr{V}:=\{\vecv \in \mbb{R}^d\,\mid\, \|\vecv\|_\Sigma=1,\,\|\vecv\|_{\veclambda}\leq r_{\veclambda}\},
	\end{align}
	We have
	\begin{align}
		\label{ine:L1ecv-1}
		\sup_{\vecv \in \mr{V}}\frac{1}{n}\sum_{i=1}^n|\vecx_i^\top \vecv+\xi_i|& \leq 	\sup_{\vecv \in \mr{V}}\left(\frac{1}{n}\sum_{i=1}^n|\vecx_i^\top \vecv+\xi_i|-\mbb{E}\frac{1}{n}\sum_{i=1}^n|\vecx_i^\top \vecv+\xi_i|\right)+	\sup_{\vecv \in \mr{V}}\mbb{E}\frac{1}{n}\sum_{i=1}^n|\vecx_i^\top \vecv+\xi_i|.
	\end{align}
	In the reminder of this proof, we evaluate the first term of the R.H.S of \eqref{ine:L1ecv-1}. 
	For any $\vecv_1,\vecv_2 \in \mr{c}\mbb{B}_\lambda^d \cap \mr{c}'\mbb{B}^d_\Sigma$, and $i \in \{1,\cdots,n\}$, we have
	\begin{align}
		&\left\||\vecx_i^\top \vecv_1+\xi_i|-\mbb{E}|\vecx_i^\top \vecv_1+\xi_i|-|\vecx_i^\top \vecv_2+\xi_i|+\mbb{E}|\vecx_i^\top \vecv_2+\xi_i|\right\|_{\psi_2}\nonumber\\
		&\quad \quad \leq 	\left\||\vecx_i^\top \vecv_1+\xi_i|-|\vecx_i^\top \vecv_2+\xi_i|\right\|_{\psi_2}+	\left\|\mbb{E}|\vecx_i^\top \vecv_1+\xi_i|-\mbb{E}|\vecx_i^\top \vecv_2+\xi_i|\right\|_{\psi_2}\nonumber\\
		&\quad \quad \stackrel{(a)}{\leq} 	\left\|\vecx_i^\top (\vecv_1-\vecv_2)\right\|_{\psi_2}+		\left\|\mbb{E}|\vecx_i^\top (\vecv_1-\vecv_2)|\right\|_{\psi_2}\nonumber\\
		&\quad \quad \stackrel{(b)}{\leq} L\|\vecv_1-\vecv_2\|_\Sigma+\|\vecv_1-\vecv_2\|_\Sigma \stackrel{(c)}{\lesssim} L\|\vecv_1-\vecv_2\|_\Sigma,
	\end{align}
	where (a) follows from the $1$-Lipschitz continuity of $|\cdot|$, and (b) follows from the $L$-subGaussian assumption and H{\"o}lder's inequality, and (c) follows from $L\geq 1$.
	From the i.i.d. assumption on $\{\epsilon_i,\vecx_i\}_{i=1}^n$ and Proposition 2.6.2 of \cite{Ver2018High}, with probability at most $e^{-u}$, we have
	\begin{align}
		\left|\sum_{i=1}^n\frac{|\vecx_i^\top \vecv_1+\xi_i|-\mbb{E}|\vecx_i^\top \vecv_1+\xi_i|-|\vecx_i^\top \vecv_2+\xi_i|+\mbb{E}|\vecx_i^\top \vecv_2+\xi_i|}{n}\right| \lesssim L\sqrt{\frac{u}{n}}\|\vecv_1-\vecv_2\|_\Sigma.
	\end{align}
		
	Then, applying Theorem 3.2 of \cite{Dir2015Tail} with $t_0=(0,\cdots,0)$, we have
	\begin{align}
		&\mbb{E}\sup_{\vecv  \in \mr{c}\mbb{B}^d_\lambda \cap \mr{c}'\mbb{B}^d_\Sigma}\left(\frac{1}{n}\sum_{i=1}^n|\vecx_i^\top \vecv+\xi_i|-\mbb{E}\frac{1}{n}\sum_{i=1}^n|\vecx_i^\top \vecv_0+\xi_i|\right)\nonumber\\
		&\quad \quad\leq C\frac{L}{\sqrt{n}}\gamma_2 (\mr{c}\mbb{B}_\lambda^d \cap \mr{c}'\mbb{B}^d_\Sigma,\|\cdot\|_\Sigma)\nonumber\\
		&\quad \quad +2\sup_{\vecv\in  \mr{c}\mbb{B}_\lambda^d \cap \mr{c}'\mbb{B}^d_\Sigma}\mbb{E}\left|\frac{1}{n}\sum_{i=1}^n|\vecx_i^\top \vecv+\xi_i|-\mbb{E}\frac{1}{n}\sum_{i=1}^n|\vecx_i^\top \vecv+\xi_i|\right|.
	\end{align}
	For the first term, from majorizing measure theorem (Theorem 2.4.1 \cite{Tal2014Upper}), we have
	\begin{align}
		\gamma_2 (\mr{c}\mbb{B}_\lambda^d \cap \mr{c}'\mbb{B}^d_\Sigma,\|\cdot\|_\Sigma)\lesssim \mbb{E}\sup_{\vecv \in \mr{c}\mbb{B}_\lambda^d \cap \mr{c}'\mbb{B}^d_\Sigma}\vecv^\top \Sigma^\frac{1}{2}\vecg \stackrel{(a)}{\lesssim} \rho \sqrt{n} \mr{c}+\mr{c}'
	\end{align}
	where (a) follows from Corollary \ref{c:gw}.

	For the second term, we have
	\begin{align}
		\sup_{\vecv\in  \mr{c}\mbb{B}_\lambda^d \cap \mr{c}'\mbb{B}^d_\Sigma}\mbb{E}\left|\frac{1}{n}\sum_{i=1}^n|\vecx_i^\top \vecv+\xi_i|-\mbb{E}\frac{1}{n}\sum_{i=1}^n|\vecx_i^\top \vecv+\xi_i|\right|&\leq \sup_{\vecv \in  \mr{c}\mbb{B}_\lambda^d \cap \mr{c}'\mbb{B}^d_\Sigma}\frac{1}{n}\sqrt{\sum_{i=1}^n\mbb{E}(|\vecx_i^\top \vecv+\xi_i|-\mbb{E}|\vecx_i^\top \vecv+\xi_i|)^2}\nonumber\\
		&\leq \sup_{\vecv \in  \mr{c}\mbb{B}_\lambda^d \cap \mr{c}'\mbb{B}^d_\Sigma}\frac{1}{n}\sqrt{\sum_{i=1}^n\mbb{E}(\vecx_i^\top \vecv+\xi_i)^2}\nonumber\\
		&= \sup_{\vecv \in  \mr{c}\mbb{B}_\lambda^d \cap \mr{c}'\mbb{B}^d_\Sigma}\frac{1}{n}\sqrt{\sum_{i=1}^n\mbb{E}(\vecv^\top \vecx_i)^2+\sum_{i=1}^n\mbb{E}\xi_i^2}\nonumber\\
		&\leq \frac{\mr{c}'+\sigma^*}{\sqrt{n}}.
	\end{align}
	Combining the arguments and from  $L\geq 1$, the proof the upper bound is complete.
\end{proof}

\begin{lemma}
	\label{l:L1+}
	Suppose $\vecx \in \mbb{R}^d$ is a random vector and $\xi, \xi_1,\xi_2 \in \mbb{R}$ are  random variables with  $\mbb{E}\vecx = \bf{0}$ and $\mbb{E}\xi=\mbb{E}\xi_1=\mbb{E}\xi_2=0$, and suppose that $\vecx$ and $\xi$ are independent each other, and $\xi_1$ and $\xi_2$ are also independent each other.  Then, for any fixed $\vecv \in \mbb{R}^d$, we have
	\begin{align}
		\mbb{E}|\vecx^\top \vecv+\xi|&\geq \frac{1}{2\sqrt{2}}\max\{\mbb{E}|\vecx^\top \vecv|,\mbb{E}|\xi|\},\nonumber\\
		\mbb{E}|\xi_1 +\xi_2|/\sqrt{2} &\geq \frac{1}{4}\max\{\mbb{E}|\xi_2|,\mbb{E}|\xi_1|\}.
	\end{align}
\end{lemma}
\begin{proof}
	This proof is similar to the one of Lemma 6.2 of \cite{Men2015lLearning}. First, we consider the first inequality.
	From symmetrization (Lemma 6.4.2 of \cite{Ver2018High}), we have
	\begin{align}
		\mbb{E}|\vecx^\top \vecv+\xi|\geq \frac{1}{2} \mbb{E}_{\vecx,\xi}\mbb{E}_{\epsilon_1,\epsilon_2}|\epsilon_1\vecx^\top \vecv+\epsilon_2\xi|,
	\end{align}
	and from Khintchine's inequality (the result in \cite{Haa1982Best}), we have
	\begin{align}
		\mbb{E}_{\vecx,\xi}\mbb{E}_{\epsilon_1,\epsilon_2}|\epsilon_1\vecx^\top \vecv+\epsilon_2\xi|\geq \frac{1}{\sqrt{2}}\mbb{E}_{\vecx,\xi} \sqrt{(\vecx^\top \vecv)^2+\xi^2}.
	\end{align}
	Note that 
	\begin{align}
		\mbb{E}_{\vecx,\xi} \sqrt{(\vecx^\top \vecv)^2+\xi^2}\geq \max\{\mbb{E}_{\vecx} |\vecx^\top \vecv|,\mbb{E}_{\xi} |\xi|\},
	\end{align}
	and the proof is complete. The proof of the second inequality is almost the same to the one of the first inequality, and we omit it.
\end{proof}
\begin{lemma}
	\label{l:huber_smooth}
	For any $b'$ and $b$ such that $b\geq b'>0$, and $a\in \mbb{R}$, we have
	\begin{align}
		\left| h\left(\frac{a}{b}\right)- h\left(\frac{a}{ b'}\right) \right|&\leq \frac{|b-b'|}{b}.
	\end{align}
\end{lemma}
\begin{proof}

	We consider the following four cases.

	\noindent
	\underline{Case I: $|a/b|<1$ and $|a/b'|<1$}
	
	Note that $|a|< \min(b,b')=b'$. Therefore, we have
	\begin{align}
		\left| h\left(\frac{a}{b}\right)- h\left(\frac{a}{ b'}\right) \right| =\left|\frac{a}{b}-\frac{a}{b'}\right|=|a|\frac{|b-b'|}{bb'} \leq\frac{|b-b'|}{b}.
	\end{align}

	\noindent
	\underline{Case II: $|a/b|\geq 1$ and $|a/b'|\geq1$}

	For a number $a \in \mbb{R}$, Define $\mr{sgn(a)}$ as the sign of $a$. Note that $\mr{sgn}(a/b)=\mr{sgn}(a/b')$. Therefore, we have
	\begin{align}
		\left| h\left(\frac{a}{b}\right)- h\left(\frac{a}{ b'}\right) \right| =0.
	\end{align}

	\noindent
	\underline{Case III: $|a/b|\geq 1$ and $|a/b'|<1$}
	
	This is impossible because of the assumption $b\geq b'>0$.

	\noindent
	\underline{Case IV: $|a/b|< 1$ and $|a/b'|\geq1$}

	Note that $b'\leq |a|<b$.  Therefore, we have
	\begin{align}
		\left| h\left(\frac{a}{b}\right)- h\left(\frac{a}{ b'}\right) \right| =\left|\frac{a}{b}-\mr{sgn}(a)\right|=\frac{|a-\mr{sgn}(a) \times b|}{b}\leq \frac{|b-b'|}{b}.
	\end{align}

\end{proof}

\begin{lemma}
	\label{l:Che}
	Suppose that  Assumption \ref{a:cov} holds. Then,  for any $\vecu =(u_1,\cdots,u_n)\in \mbb{R}^n$, we have
	\begin{align}
		\sup_{\vecv  \in \mr{c}\mbb{B}^d_\lambda \cap \mr{c}'\mbb{B}^{d}_\Sigma}\sum_{i=1}^ncu_i(\vecx_i^\top \vecv)& \lesssim L\left(\mr{c}'\sqrt{n}+ \rho\sqrt{n}\mr{c}+\mr{c}'\sqrt{\log(1/\delta)}\right)\|\vecu\|_2
	\end{align}
\end{lemma}
\begin{proof}
	For any fixed $\vecv,\vecv'\in \mr{c}\mbb{B}_\lambda^d \cap \mr{c}'\mbb{B}^d_\Sigma$, and $\vecu,\vecu' \in \mbb{S}^{n-1}$, we have
	\begin{align}
		&\sum_{i=1}^nu_i(\vecx_i^\top \vecv)-\sum_{i=1}^n u'_i(\vecx_i^\top \vecv')=\sum_{i=1}^n u_i(\vecx_i^\top \vecv)-\sum_{i=1}^n u'_i(\vecx_i^\top \vecv)+\sum_{i=1}^n u'_i(\vecx_i^\top \vecv)-\sum_{i=1}^n u'_i(\vecx_i^\top \vecv')
	\end{align}
	For $\sum_{i=1}^n u_i(\vecx_i^\top \vecv)-\sum_{i=1}^nu'_i(\vecx_i^\top \vecv)$, for any $i \in \{1,\cdots,n\}$, from the $L$-subGaussian assumption, we have
	\begin{align}
		\| u_i(\vecx_i^\top \vecv)-u'_i(\vecx_i^\top \vecv)\|_{\psi_2}=|u_i-u_i'|\|\vecx_i^\top \vecv\|_{\psi_2}\leq L|u_i-u_i'|\mr{c}',
	\end{align}
	and from Proposition 2.6.2 of \cite{Ver2018High}, we have
	\begin{align}
		\mbb{P}\left(\left|\sum_{i=1}^n u_i(\vecx_i^\top \vecv)-\sum_{i=1}^n u'_i(\vecx_i^\top \vecv)\right|\gtrsim L\mr{c}' \|\vecu-\vecu'\|_2t \right)\leq e^{-t^2}
	\end{align}
	For $\sum_{i=1}^n u'_i(\vecx_i^\top \vecv)-\sum_{i=1}^n u'_i(\vecx_i^\top \vecv')$, for any $i \in \{1,\cdots,n\}$, from the $L$-subGaussian assumption, we have
	\begin{align}
		\| u'_i(\vecx_i^\top \vecv)-u'_i(\vecx_i^\top \vecv')\|_{\psi_2}=|u_i'|\|\vecx_i^\top (\vecv-\vecv')\|_{\psi_2}\leq L|u_i'|\|\vecv-\vecv'\|_\Sigma,
	\end{align}
	and  from  Proposition 2.6.2 of \cite{Ver2018High}, we have
	\begin{align}
		\mbb{P}\left(\left|\sum_{i=1}^n u'_i(\vecx_i^\top \vecv)-\sum_{i=1}^n u'_i(\vecx_i^\top \vecv')\right|\gtrsim L\|\vecv-\vecv'\|_\Sigma t\right)\leq e^{-t^2}.
	\end{align}
	We see that, from \eqref{ine:d:l-02}, $\sum_{i=1}^n u_i(\vecx_i^\top \vecv)$ satisfies \textit{the $\psi_\alpha$ condition} with $\alpha=2$ in \cite{Dir2015Tail}. Also, we have
	\begin{align}
		\left\|u_i(\vecx_i^\top \vecv)-u'_i(\vecx_i^\top \vecv')\right\|_{\psi_2}& \leq \left\|u_i(\vecx_i^\top \vecv)\right\|_{\psi_2}+\left\|u'_i(\vecx_i^\top \vecv')\right\|_{\psi_2}\leq (|u_i|+|u'_i|) L\mr{c}',
	\end{align}
	and from  Proposition 2.6.2 of \cite{Ver2018High} and $\sum_{i=1}^n(|u_i|+|u'_i|)^2\leq 2\|\vecu\|_2+2\|\vecu'\|_2 \leq 4$, we have
	\begin{align}
		\mbb{P}\left(\left|\sum_{i=1}^n u_i(\vecx_i^\top \vecv)-\sum_{i=1}^n u'_i(\vecx_i^\top \vecv')\right|\gtrsim L\mr{c}' t\right)\leq e^{-t^2}.
	\end{align}

	Then, from Theorem 3.2 of \cite{Dir2015Tail} with $t_0 = (\underbrace{0,\cdots,0}_{d},1,\underbrace{0,\cdots,0}_{n-1})$ and $\Delta_d(T)= L\mr{c}'$,  with probability at least $1-\delta$,  we have
	\begin{align}
		&\sup_{\vecv  \in \mr{c}\mbb{B}^d_\lambda \cap \mr{c}'\mbb{B}^{d}_\Sigma,\,\vecu \in \mbb{S}^{n-1}}\left|\sum_{i=1}^n\vecu_i(\vecx_i^\top \vecv)\right|\nonumber\\
		&\quad \quad \lesssim  L\mr{c}' \gamma_2( \mbb{S}^{n-1},\|\cdot\|_2)+L\gamma_2( \mr{c}\mbb{B}^d_\lambda \cap \mr{c}'\mbb{B}^{d}_\Sigma,\|\cdot\|_\Sigma)+L\mr{c}'\sqrt{\log(1/\delta)}\nonumber\\
		&\quad \quad \stackrel{(a)}{\lesssim} L\mr{c}' \mbb{E}\sup_{\vecu\in  \mbb{S}^{n-1}}\langle \vecu,\vecg^n\rangle+\mbb{E}L\sup_{\vecv \in  \mr{c}\mbb{B}^d_\lambda \cap \mr{c}'\mbb{B}^{d}_\Sigma}\langle \vecv ,\Sigma^\frac{1}{2}\vecg^d \rangle+L\mr{c}'\sqrt{\log(1/\delta)}\nonumber\\
		&\quad \quad\stackrel{(b)}{\lesssim}  L\mr{c}'\sqrt{n}+L( \rho\sqrt{n}\mr{c}+\mr{c}')+L\mr{c}'\sqrt{\log(1/\delta)},
	\end{align}
	where (a) follows from the majorizing measure theorem (Theorem 2.4.1 \cite{Tal2014Upper}) and (b) follows from H{\"o}lder's inqualiy and Corollary \ref{c:gw}. After normalizing $\vecu = \vecu/\|\vecu\|_2$, the proof is complete.

\end{proof}

\begin{lemma}
	\label{l:no-rad-exp}
	Suppose that  Assumption \ref{a:cov} holds. 
	 For any $\vecv \in \mbb{R}^d$, we have
	 \begin{align}
		\mbb{P}\left(\sum_{i=1}^n \frac{1}{n}|\vecx_i^\top \vecv|\geq  -CL\rho \|\vecv\|_{\lambda}-CL\sqrt{\frac{\log (1/\delta)}{n}}\|\vecv\|_\Sigma+\frac{1}{72L^4}\|\vecv\|_\Sigma\right)\geq 1-\delta.
	 \end{align}
\end{lemma}
\begin{proof}
	First, we derive a concentration inequality for $\sum_{i=1}^n |\vecx_i^\top \vecv|$ for any $\vecv \in \mr{V}$, and second, derive a expectation.
	In this proof, define 
	\begin{align}
		\mr{U}:=\{\vecv \in \mbb{R}^d\mid \|\vecv\|_{\veclambda} \leq r_{\veclambda}, \|\vecv\|_\Sigma =1 \}\, \quad \mr{U}':=\{\vecv \in \mbb{R}^d\mid \|\vecv\|_{\veclambda} \leq r_{\veclambda}, \|\vecv\|_\Sigma \leq 1 \}.
	\end{align}
	Decompose 
	\begin{align}
		\label{ine:absm}
		\inf_{\vecv \in \mr{U}} \sum_{i=1}^n \frac{1}{n}|\vecx_i^\top \vecv|& = \inf_{\vecv \in \mr{V}}\left( \sum_{i=1}^n |\vecx_i^\top \vecv|-\mbb{E}\sum_{i=1}^n\frac{1}{n}|\vecx_i^\top \vecv|+ \mbb{E}\sum_{i=1}^n\frac{1}{n}|\vecx_i^\top \vecv|\right)\nonumber \\
		&\geq \inf_{\vecv \in  \mr{U}}\left( \sum_{i=1}^n \frac{1}{n} |\vecx_i^\top \vecv|-\mbb{E}\sum_{i=1}^n \frac{1}{n}|\vecx_i^\top \vecv|\right)+\inf_{\vecv \in  \mr{U}} \mbb{E}|\vecx^\top \vecv|\nonumber \\
		&\stackrel{(a)}{\geq} \inf_{\vecv \in \mr{U} }\left( \sum_{i=1}^n \frac{1}{n} |\vecx_i^\top \vecv|-\mbb{E}\sum_{i=1}^n \frac{1}{n}|\vecx_i^\top \vecv|\right)+\inf_{\vecv \in  \mr{U}} \frac{\sqrt{\mbb{E}(\vecx^\top \vecv)^2}}{72L^4}\nonumber \\
		&\geq  -\sup_{\vecv \in \mr{U}'}\left| \sum_{i=1}^n \frac{1}{n} |\vecx_i^\top \vecv|-\mbb{E}\sum_{i=1}^n \frac{1}{n}|\vecx_i^\top \vecv|\right|+\frac{1}{72L^4},
	\end{align}
	where (a) follows from Lemma \ref{l:L1L2}. 
	For any fixed $\vecv,\vecv'\in \mr{U}'$ and for any $i\in \{1,\cdots,n\}$, we have
	\begin{align}
		\left\||\vecx_i^\top \vecv|-|\vecx_i^\top \vecv'|\right\|_{\psi_2}\stackrel{(a)}{\leq} \left\|\langle \vecv-\vecv', \vecx_i\rangle\right\|_{\psi_2}\stackrel{(b)}{\leq } L\sqrt{p}\|\vecv-\vecv'\|_\Sigma,
	\end{align}
	where (a) follows from Lipschitz continuity of $|\cdot|$, and (b) follows from \eqref{ine:d:l-02}, and  from the i.i.d. assumption on $\{\vecx_i\}_{i=1}^n$, Proposition 2.6.1 of \cite{Ver2018High} and  \eqref{ine:d:l-02}, we have
	\begin{align}
		\left\|\frac{1}{n}\sum_{i=1}^n|\vecx_i^\top \vecv|-\frac{1}{n}\sum_{i=1}^n|\vecx_i^\top \vecv|\right\|_{\psi_2}\lesssim \frac{L\sqrt{p}}{\sqrt{n}} \|\vecv-\vecv'\|_\Sigma. 
	\end{align}
	We see that, from \eqref{ine:d:l-02}, $\frac{1}{n}\sum_{i=1}^n|\vecx_i^\top \vecv|$ satisfies \textit{the $\psi_\alpha$ condition} with $\alpha=2$ in \cite{Dir2015Tail}.
	Then, from Theorem 3.2 of \cite{Dir2015Tail} with $t_0 = \bf{0}$, with probability at least $1-\delta$, 
	\begin{align}
		\sup_{\vecv  \in  \mr{U}'}\left|\frac{1}{n}\sum_{i=1}^n|\vecx_i^\top \vecv|-\mbb{E}\frac{1}{n}\sum_{i=1}^n|\vecx_i^\top \vecv| \right|& \lesssim  \frac{L \gamma_2(\mr{U}' ,\|\cdot\|_\Sigma )}{\sqrt{n}}+\frac{L\sqrt{\log(1/\delta)}}{\sqrt{n}}\nonumber\\
		&\stackrel{(a)}{\lesssim}  \frac{L \sup_{\vecv \in \mr{U}'}\langle \vecg^d,\vecv\rangle}{\sqrt{n}}+L\sqrt{\frac{\log(1/\delta)}{n}}\nonumber\\
		&\stackrel{(b)}{\lesssim}  L\frac{ \rho\sqrt{n}r_{\veclambda}+1}{\sqrt{n}}+L\sqrt{\frac{\log(1/\delta)}{n}} ,
	\end{align}
	where (a) follows from majorizing measure theorem (Theorem 2.4.1 \cite{Tal2014Upper}) and (b) follows from Corollary \ref{c:gw}. Combining this and \eqref{ine:absm}, we have
	\begin{align}
		\mbb{P}\left(\inf_{\vecv \in \mr{U}} \sum_{i=1}^n \frac{1}{n}|\vecx_i^\top \vecv| \geq- LC\frac{ \rho\sqrt{n}r_{\veclambda}+1}{\sqrt{n}}-LC\sqrt{\frac{\log(1/t)}{n}}+\frac{1}{72L^4}\right)\geq 1-t,
	\end{align}
	and from Lemma \ref{l:peeling3}, for any $\vecv \in \mr{U}$,  we have
	\begin{align}
		\mbb{P}\left(\sum_{i=1}^n \frac{1}{n}|\vecx_i^\top \vecv|\geq  -CL\frac{\rho \sqrt{n}\|\vecv\|_{\lambda} + 1}{\sqrt{n}}-CL\sqrt{\frac{\log (2/t)}{n}}+\frac{1}{72L^4}\right)\geq 1-t,
 \end{align}
 and normalizing $\vecv = \vecv/\|\vecv\|_\Sigma$, for any $\vecv\in \mbb{R}^d$, we have
\begin{align}
 \mbb{P}\left(\sum_{i=1}^n \frac{1}{n}|\vecx_i^\top \vecv|\geq  -CL\rho \|\vecv\|_{\lambda}-CL\sqrt{\frac{\log (1/\delta)}{n}}\|\vecv\|_\Sigma+\frac{1}{72L^4}\|\vecv\|_\Sigma\right)\geq 1-\delta,
\end{align}
where we use the assumption on $\delta$.
\end{proof}

\begin{lemma}
	\label{l:noiseconcentration}
	Assume Assumption \ref{a:noise:EI}.
	Then, with probability at least $ 1-4\delta$, we have
	\begin{align}
		\label{ine:noisparcon1}
	\left|\frac{1}{n}\sum_{i=\max\{\log(1/\delta),o\}}^{n}\xi_i^{\sharp}\right|&\lesssim \sigma_\mr{M} \sqrt{\frac{\log(1/\delta)}{n}}+\sigma_{\mr{M}}\left(\frac{\max\{\log(1/\delta),o\}}{n}\right)^{1-\frac{1}{\mr{M}}}.
	\end{align}
\end{lemma}
\begin{proof}
	From Lemma A.1 of \cite{MinNdaWan2024Robust}, with probability at least $1-2\exp(-\max\{\log(1/\delta),o\})$, we have
	\begin{align}
		\label{ine:noiseA1}
		\max_{i\geq \max\{\log(1/\delta),o\}}|\xi_i^\sharp|\lesssim \sigma_\mr{M}\left(\frac{n}{\max\{\log(1/\delta),o\}}\right)^\frac{1}{\mr{M}}(:=R).
	\end{align}
	Therefore, with probability at least $1-2\exp(-\max\{\log(1/\delta),o\})$, we have
	\begin{align}
		\sum_{i=\max\{\log(1/\delta),o\}}^{n}\xi_i^{\sharp} = \sum_{i=1}^n \xi_i \mathbb{I}_{|\xi_i|\leq R}-\sum_{i=1}^{\max\{\log(1/\delta),o\}}\xi_i^{\sharp} \mathbb{I}_{|\xi_i^{\sharp}|\leq R}.
		\end{align}
To apply Bernstein's inequality, note that, from any $i\in \{1,\cdots,n\}$, we have
\begin{align}
		\mbb{E } \left|| \xi_i \mathbb{I}_{|\xi_i|\leq R}|\right|^2\leq (\sigma^*)^2,\quad\mbb{E} \left|| \xi_i \mathbb{I}_{|\xi_i|\leq R}|\right|^q \leq  \mbb{E} \left|| \xi_i \mathbb{I}_{|\xi_i|\leq R}|\right|^2 R^{q-2}\leq (\sigma^*)^2 R^{q-2},
\end{align} 
and from Bernstein's inequality (Theorem 2.10 of \cite{BouLugMas2013concentration}), with probability at least $1-\delta$, we have
\begin{align}
	\label{ine:l:noiseconcentration-01}
	\left|\sum_{i=1}^n \xi_i \mathbb{I}_{|\xi_i|\leq R}-\mbb{E}\sum_{i=1}^n \xi_i \mathbb{I}_{|\xi_i|\leq R}\right| \lesssim \sigma^* \sqrt{n\log(1/\delta)}+R\log(1/\delta)\lesssim \sigma_\mr{M} \sqrt{n\log(1/\delta)},
\end{align}
where the last inequality follows from the definition of $R$ and  monotonicity of the moment.
The expectation term in \eqref{ine:l:noiseconcentration-01} is evaluated as 
\begin{align}
	|\mbb{E} \xi_i \mathbb{I}_{|\xi_i|\leq R}| = |\mbb{E}\xi_i  -\mbb{E}\xi_i \mathbb{I}_{|\xi_i|\geq R} |=|\mbb{E}\xi_i \mathbb{I}_{|\xi_i|\geq R} |\stackrel{(a)}{\leq} \sigma_{\mr{M}} (\mbb{E}\mathbb{I}_{|\xi|\geq R})^\frac{\mr{M}-1}{\mr{M}}\stackrel{(b)}{\leq }\sigma_{\mr{M}} \left(\frac{\sigma_{\mr{M}}^\mr{M}}{R^\mr{M}}\right)^\frac{\mr{M}-1}{\mr{M}},
\end{align}
where (a) follows from H{\"o}lder's inequality and (b) follows from Markov's inequality.

Lastly, from simple algebra, we have
\begin{align}
	\left|\sum_{i=1}^{\max\{\log(1/\delta),o\}}\xi_i^{\sharp} \mathbb{I}_{|\xi_i^{\sharp}|\leq R}\right|\leq \max\{\log(1/\delta),o\}R.
	\end{align}
	Combining the arguments, taking union bound and $(1-2e^{-\max\{\log(1/\delta),o\}})(1-\delta)\geq 1-3\delta$, the proof is complete.

\end{proof}
\begin{lemma}
	\label{l:noiseconcentration2}
	Assume Assumption \ref{a:noise:EI2}.
	Then, with probability at least $ 1-4\delta$, we have
	\begin{align}
	\left|\frac{1}{n}\sum_{i=\max\{\log(1/\delta),o\}}^{n}\xi_i^{\sharp}\right|&\lesssim \sigma_\mr{M} \sqrt{\frac{\log(1/\delta)}{n}}+\sigma_{\mr{M}}\frac{o}{n}\sqrt{\log \frac{n}{o}}.
	\end{align}
\end{lemma}
\begin{proof}
	This proof is almost the same as the one of the previous lemma. From Lemma A.1 of \cite{MinNdaWan2024Robust}, with probability at least $1-2\exp(-\max\{\log(1/\delta),o\})$, we have
	\begin{align}
		\max_{i\geq \max\{\log(1/\delta),o\}}|\xi_i^\sharp|\lesssim \sigma_\mr{M}\left(\log \frac{en}{o}\right)^\frac{1}{2}(:=R),
	\end{align}
and note that 
\begin{align}
	|\mbb{E} \xi_i \mathbb{I}_{|\xi|\leq R}| = |\mbb{E}\xi_i  -\mbb{E}\xi_i \mathbb{I}_{|\xi|\geq R} |=|\mbb{E}\xi_i \mathbb{I}_{|\xi|\geq R} |\stackrel{(a)}{\leq} \sigma^* (\mbb{E}\mathbb{I}_{|\xi|\geq R})^\frac{1}{2}&\stackrel{(b)}{\leq }2\sigma^*  \exp\left(-\frac{R^2}{(\sigma^*)^2 c^2_{\mr{M}}}\right)\nonumber\\
	& \lesssim  \sigma^* \frac{o}{n},
\end{align}
where (a) follows from H{\"o}lder's inequality and (b) follows from the relation of expectation and tail probability and Assumption \ref{a:noise:EI2}.
Then, from the same argument of the proof of the previous lemma, we complete the proof.

\end{proof}

\begin{lemma}[Peeling Lemma I]
	\label{l:peeling1}
	Let $\mr{V}$ be a set and $h_1, h_2, h_3: \mr{V}\to \mbb{R}_{+}$ be functions. Let $g_1, g_2 : \mbb{R}_+ \to \mbb{R}_+$ be right-continuous and non-decreasing functions. For $i=1,2$, define $g_i^{-1}(x) = \inf\{\mr{a} \in \mbb{R}_+\,:\,g_i(\mr{a})\geq x\}$, which are generalized inverses of $g_1$ and $g_2$.
	Let $\mr{a}_1$ and $\mr{a}_2$ be positive constants such that $\mr{a_2}/\mr{a_1}\leq 1$ and $\mr{c}\geq 1$ be a constant. Suppose for any $r_1, r_2, r_3 >0$ and $\delta \in (0,1/\mr{c}]$, with probability at least $1-\mr{c}\delta$, we have
	\begin{align}
		\sup_{\vecv \in \mr{V},\,(h_1,h_2,h_3)(\vecv) \leq r_1,r_2,r_3} \mr{M}(\vecv)\leq \mr{a}_1(g_1(r_1)+g_2(r_2r_3))+\mr{a}_2\sqrt{\log(1/\delta)}.
	\end{align}
	Define the event above as $\mr{E}(r_1,r_2r_3, \delta)$. 

	Then, for any $\delta \in (0, 1/\mr{c}]$, with probability at least $1-\mr{c}\delta$, for any $\vecv \in \mr{V}$,  we have
	\begin{align}
		 \mr{M}(\vecv)&\lesssim \mr{a}_1\left(g_1\circ h_1 (\vecv) +g_2 (h_2(\vecv)\cdot h_3(\vecv))  +1\right)+\mr{a}_2\sqrt{\log\left(4/\delta\right)}.
	\end{align}
\end{lemma}
\begin{proof}
	We remark that the proofs of the peeling lemmas in this paper resemble those in \cite{DalTho2019Outlier,Tho2020Outlier,Tho2023Outlier}.
	Let $\alpha_0=0$ and $\alpha_k:= \zeta^{k-1}$, where $\zeta >1$ is a sufficiently large numerical constant determined later.  Decompose $\mr{V}$ as
	\begin{align}
		\mr{V}_{k_1,k_2} :=\{\vecv \in \mr{V}\,\mid\, \alpha_{k_1}\leq (g_1 \circ h_1)(\vecv)<\alpha_{k_1+1},\,\alpha_{k_2}\leq g_2 (h_2(\vecv)\cdot h_3(\vecv))<\alpha_{k_2+1}\}.
	\end{align}
	Define $\beta_{1,m} := g_1^{-1}(\alpha_m)$ and $\beta_{2,m}:=g_2^{-1}(\alpha_m)$, then we have $g_1(\beta_{1,m})=\alpha_m$ and $g_2(\beta_{2,m})=\alpha_m$.
	Note that $\sum_{k_1\geq1, k_2\geq 1}(k_1k_2)^{-2}\leq 4$ and we have
	\begin{align}
		\label{ine:l:peeling-01}
		\mbb{P}\left((\mr{E}:=) \cap_{k_1\geq 1, k_2\geq 1}^\infty \mr{E}\left(\beta_{1,k_1},\beta_{2,k_2},\frac{\delta}{4(k_1k_2)^2}\right) \right)\geq 1-\mr{c}\delta
	\end{align}
	for fixed $\delta \in (0,1/\mr{c}]$. We assume the event $\mr{E}$ is realized. 

	We note that for every $\vecv \in \mr{V}$, there are $l_1,l_2$ such that $\vecv \in \mr{V}_{l_1,l_2}$ because $g_1,g_2$ are non-decreasing.

	\noindent
	\underline{Case 1: $l_1=l_2=0$}

	In this case, 
	\begin{align}
		\mr{V}_{0,0} :=\{\vecv \in \mr{V}\,\mid\, 0 \leq (g_1 \circ h_1)(\vecv)<1,\,0\leq g_2 (h_2(\vecv)\cdot h_3(\vecv))< 1\},
	\end{align}
	and we have $h_1(\vecv)\leq \beta_{1,1}=g^{-1}_1(1)$, $h_2(\vecv)\cdot h_3(\vecv)\leq \beta_{2,1} = g^{-1}_2(1)$, then from \eqref{ine:l:peeling-01}, we have
	\begin{align}
			\mr{M}(\vecv)&\leq \mr{a}_1(g_1(\beta_{1,1})+g(\beta_{2,1}))+\mr{a}_2\sqrt{\log(4/\delta)}\leq 2\mr{a}_1+\mr{a}_2\sqrt{\log\left(4/\delta\right)}.
	\end{align}

	\noindent
	\underline{Case 2: $l_1\geq 1, l_2=0$}

	In this case	
	\begin{align}
		\mr{V}_{l_1,0} :=\{\vecv \in \mr{V}\,\mid\, \alpha_{l_1}\leq (g_1 \circ h_1)(\vecv)<\alpha_{l_1+1},\,0\leq h_2(\vecv)\cdot h_3(\vecv)<1\}.
	\end{align}
	Note that 
	\begin{align}
		g_1(\beta_{1,l_1+1}) = \zeta \alpha_{l_1} = \zeta \alpha_{l_1+1} + \zeta \alpha_{l_1}-\zeta \alpha_{l_1+1} \leq \zeta^2 g_1\circ h_1 (\vecv) + \zeta \alpha_{l_1}-\zeta \alpha_{l_1+1},
	\end{align}
	$h_1(\vecv) \leq \beta_{1,l_1+1}=g_1^{-1}(\alpha_{l_1+1})$ and $h_2(\vecv)\cdot h_3(\vecv) \leq \beta_{2,1}=g_2^{-1}(1)$. Then, from \eqref{ine:l:peeling-01},  we have
	\begin{align}
			\mr{M}(\vecv)&\leq \mr{a}_1( g_1(\beta_{1,l_1+1})+1)+\mr{a}_2\sqrt{\log\left(\frac{4(l_1+1)^2}{\delta}\right)}\nonumber\\
			&\leq \mr{a}_1\left(\zeta^2 g_1\circ h_1 (\vecv) + \zeta \alpha^{l_1}-\zeta \alpha^{l_1+1}+1\right)+\mr{a}_2\sqrt{\log\left(4/\delta\right)}+\mr{a}_2\sqrt{2\log(l_1+1)}.
	\end{align}
	We note 
	\begin{align}
		\mr{a}_1\left(\zeta \alpha^{l_1}-\zeta \alpha^{l_1+1}\right)+\mr{a}_2\sqrt{2\log(l_1+1)}
		&=\mr{a}_1\zeta \alpha_{l_1} \left(\left( 1-\zeta\right)+\frac{\mr{a}_2}{\mr{a}_1}\frac{\sqrt{2\log(l_1+1)}}{ \zeta^{l_1-1}}\right)\nonumber\\
		&\stackrel{(a)}{\leq} \mr{a}_1\zeta \alpha_{l_1} \left(\left( 1-\zeta\right)+\frac{\sqrt{2\log(l_1+1)}}{ \zeta^{l_1-1}}\right)\stackrel{(b)}{\leq} 0,
	\end{align}
	where (a) follows from $\mr{a}_2/\mr{a}_1\leq 1$, and (b) follows from $\zeta>1$ is sufficiently large.

	\noindent
	\underline{Case 3: $l_1=0, l_2\geq1$}

	In this case,
	\begin{align}
		\mr{V}_{0,l_2} :=\{\vecv \in \mr{V}\,\mid\, 0\leq (g_1 \circ h_1)(\vecv)<1,\,\alpha_{l_2}\leq g_2 (h_2(\vecv)\cdot h_3(\vecv))<\alpha_{l_2+1}\}.
	\end{align}
	Note that
	\begin{align}
		g_2(\beta_{2,l_2+1}) &= \zeta \alpha_{l_2} = \zeta \alpha_{l_2+1} + \zeta \alpha_{l_2}-\zeta \alpha_{l_2+1} \leq \zeta^2 g_2 (h_2(\vecv)\cdot h_3(\vecv)) + \zeta \alpha_{l_2}-\zeta \alpha_{l_2+1},
	\end{align}
	$h_1(\vecv) \leq \beta_{1,2}=g_1^{-1}(1)$ and $h_2(\vecv) \cdot h_3(\vecv)\leq \beta_{2,l_2+1}=g_2^{-1}(\alpha_{l_2+1})$, then from \eqref{ine:l:peeling-01},  
	we have
	\begin{align}
			\mr{M}(\vecv)&\leq \mr{a}_1(1 +g(\beta_{2,l_2+1}))+\mr{a}_2\sqrt{\log\left(\frac{4(l_2+1)^2}{\delta}\right)}\nonumber\\
			&\leq		 \mr{a}_1\left(1+\zeta^2 g_2 (h_2(\vecv)\cdot h_3(\vecv)) + \zeta \alpha_{l_2}-\zeta \alpha_{l_2+1}\right)+\mr{a}_2\sqrt{\log\left(4/\delta\right)}+\mr{a}_2\sqrt{2\log(l_2+1)}.
	\end{align}
	We note that
	\begin{align}
		\mr{a}_1\left(\zeta \alpha_{l_2}-\zeta \alpha_{l_2+1}\right)+\mr{a}_2\sqrt{2\log(l_2+1)}
		&=\mr{a}_1\zeta \alpha^{l_2} \left(\left( 1-\zeta\right)+\frac{\mr{a}_2}{\mr{a}_1}\frac{\sqrt{2\log(l_2+1)}}{ \zeta^{l_2-1}}\right)\nonumber\\
		&\stackrel{(a)}{\leq} \mr{a}_1\zeta \alpha_{l_2} \left(\left( 1-\zeta\right)+\frac{\sqrt{2\log(l_2+1)}}{ \zeta^{l_2-1}}\right)\stackrel{(b)}{\leq} 0,
	\end{align}
	where (a) follows from $\mr{a}_2/\mr{a}_1\leq 1$, and (b) follows from the fact that $\zeta>1$ is sufficiently large.

	\noindent
	\underline{Case 4: $l_1\geq 1, l_2\geq1$}

	In this case,
	\begin{align}
		\mr{V}_{l_1,l_2} :=\{\vecv \in \mr{V}\,\mid\, \alpha_{l_1}\leq (g_1 \circ h_1)(\vecv)<\alpha_{l_1+1},\,\alpha_{l_2}\leq g_2 (h_2(\vecv)\cdot h_3(\vecv))<\alpha_{l_2+1}\}.
	\end{align}
	Remember that
	\begin{align}
		g_1(\beta_{1,l_1+1}) &= \zeta \alpha_{l_1} = \zeta \alpha_{l_1+1} + \zeta \alpha_{l_1}-\zeta \alpha_{l_1+1} \leq \zeta^2 g_1\circ h_1 (\vecv) + \zeta \alpha_{l_1}-\zeta \alpha_{l_1+1},\nonumber\\
		g_2(\beta_{2,l_2+1}) &= \zeta \alpha_{l_2} = \zeta \alpha_{l_2+1} + \zeta \alpha_{l_2}-\zeta \alpha_{l_2+1} \leq \zeta^2 g_2 (h_2(\vecv)\cdot h_3(\vecv)) + \zeta \alpha_{l_2}-\zeta \alpha_{l_2+1},
	\end{align}
	and note that $h_1(\vecv) \leq \beta_{1,l_1+1}=g_1^{-1}(\alpha_{l_1+1})$ and $h_2(\vecv)\cdot h_3(\vecv) \leq \beta_{2,l_2+1}=g_2^{-1}(\alpha_{l_2+1})$. Then from \eqref{ine:l:peeling-01},  
	we have
	\begin{align}
			\mr{M}(\vecv)&\leq \mr{a}_1(g_1(\beta_{1,l_1+1}) +g(\beta_{2,l_2+1}))+\mr{a}_2\sqrt{\log\left(\frac{4((l_1+1)(l_2+1))^2}{\delta}\right)}\nonumber\\
			&\leq		 \mr{a}_1\left(\zeta^2 g_1\circ h_1 (\vecv) + \zeta \alpha_{l_1}-\zeta \alpha_{l_1+1}+\zeta^2 g_2 (h_2(\vecv)\cdot h_3(\vecv)) + \zeta \alpha_{l_2}-\zeta \alpha_{l_2+1}\right)\nonumber \\
			&\quad \quad +\mr{a}_2\sqrt{\log\left(4/\delta\right)}+\mr{a}_2\sqrt{2\log(l_1+1)(l_2+1)}.
	\end{align}
	Remember that 
	\begin{align}
		\mr{a}_1\left(\zeta \alpha^{l_1}-\zeta \alpha^{l_1+1}\right)+\mr{a}_2\sqrt{2\log(l_1+1)}
		\leq 0,\,\mr{a}_1\left(\zeta \alpha^{l_2}-\zeta \alpha^{l_2+1}\right)+\mr{a}_2\sqrt{2\log(l_2+1)}
		\leq 0,
	\end{align}
	and the proof is complete.
\end{proof}

\begin{lemma}[Peeling Lemma II]
	\label{l:peeling2}
	Let $\mr{V}$ be a set and $h : \mr{V}\to \mbb{R}_{+}$ are functions. Let $g : \mbb{R}_+ \to \mbb{R}_+$ be a right-continuous and non-decreasing function. 
	Define $g^{-1}(x) = \inf\{\mr{a} \in \mbb{R}_+\,:\,g(\mr{a})\geq x\}$, which are generalized inverses of $g$.
	Let $\mr{a}, \mr{a}_1,\mr{a}_2,\mr{a}_3$ be positive constants such that $\mr{a}_2/\mr{a}_1\leq 1$, $\mr{a_3}/\mr{a}_1\leq \mr{a}$ and $\mr{c}\geq 1$ be a constant. Suppose for any $r >0$ and $\delta \in (0,1/\mr{c}]$, with probability at least $1-\mr{c}\delta$, we have
	\begin{align}
		\sup_{\vecv \in \mr{V},\,h(\vecv) \leq r} \mr{M}(\vecv)\leq \mr{a}_1g(r)+\mr{a}_2\mr{a}\sqrt{\log(1/\delta)}+\mr{a}_3 \log(1/\delta).
	\end{align}
	Define the event above as $\mr{E}(r,\delta)$. 

	Then, for any $\delta \in (0, 1/\mr{c}]$, with probability at least $1-\mr{c}\delta$, for any $\vecv \in \mr{V}$, xwe have
	\begin{align}
		\mr{M}(\vecv)
		&\leq \mr{a}_1(g\circ h (\vecv) +\mr{a})+\mr{a}_2\mr{a}\sqrt{\log\left(2/\delta\right)}+\mr{a}_3\log\left(2/\delta\right).
	\end{align}
\end{lemma}
\begin{proof}
	Let $\alpha_0=0$ and $\alpha_k:=  \mr{a}\zeta^{k-1}$, where $\zeta >1$ is a sufficiently large numerical constant determined later. Decompose $\mr{V}$ as
	\begin{align}
		\mr{V}_{k} :=\{\vecv \in \mr{V}\,\mid\, \alpha_{k}\leq (g\circ h)(\vecv)<\alpha_{k}\}.
	\end{align}
	Define $\beta_m := g^{-1}(\alpha_m)$, then we have $g(\beta_m)=\alpha_m$.
	We note that $\sum_{k\geq1}k^{-2}\leq 2$ and we have
	\begin{align}
		\label{ine:l:peeling-02}
		\mbb{P}\left((\mr{E}':=) \cap_{k\geq 1}^\infty \mr{E}\left(\beta_k,\frac{\delta}{2k^2}\right) \right)\geq 1-\mr{c}\delta
	\end{align}
	for fixed $\delta \in (0,1/\mr{c}]$. We assume $\mr{E}'$ is realized.

	We note that for every $\vecv \in \mr{V}$, there are $l$ such that $\vecv \in \mr{V}_l$ because $g$ is non-decreasing.

	\noindent
	\underline{Case 1: $l=0$}

	In this case, 
	\begin{align}
		\mr{V}_{0} :=\{\vecv \in \mr{V}\,\mid\, 0 \leq (g \circ h)(\vecv)<\mr{a}\},
	\end{align}
	and we have $h(\vecv)\leq \beta_1=g^{-1}(\mr{a})$, then from \eqref{ine:l:peeling-02}, we have
	\begin{align}
			\mr{M}(\vecv)&\leq \mr{a}_1g(\beta_1)+\mr{a}_2\mr{a}\sqrt{\log\left(\frac{2(l+1)^2}{\delta}\right)}+\mr{a}_3\log\left(\frac{2(l+1)^2}{\delta}\right)\nonumber\\
			&\leq \mr{a}_1\mr{a}+\mr{a}_2\mr{a}\sqrt{\log\left(2/\delta\right)}+\mr{a}_3\log\left(2/\delta\right).
	\end{align}

	\noindent
	\underline{Case 2: $l\geq 1$}

	In this case	
	\begin{align}
		\mr{V}_{l} :=\{\vecv \in \mr{V}\,\mid\, \alpha_{l}\leq (g \circ h)(\vecv)<\alpha_{l+1}\}.
	\end{align}
	Note that 
	\begin{align}
		g(\beta_{l+1}) = \zeta \alpha_{l} = \zeta \alpha_{l+1} + \zeta \alpha_{l}-\zeta \alpha_{l+1} \leq \zeta^2 g\circ h (\vecv) + \zeta \alpha_{l}-\zeta \alpha_{l+1},
	\end{align}
	$h(\vecv) \leq \beta_{l+1}=g^{-1}(\alpha_{l+1})$. Then, from \eqref{ine:l:peeling-02},  we have
	\begin{align}
			\mr{M}(\vecv)&\leq \mr{a}_1 g(\beta_{l+1})+\mr{a}_2\mr{a}\sqrt{\log\left(\frac{2(l+1)^2}{\delta}\right)}+\mr{a}_3\log\left(\frac{2(l+1)^2}{\delta}\right)\nonumber\\
			&\leq \mr{a}_1(\zeta^2 g\circ h (\vecv) + \zeta \alpha_{l}-\zeta \alpha_{l+1})+\mr{a}_2\mr{a}\sqrt{\log\left(\frac{2(l+1)^2}{\delta}\right)}+\mr{a}_3\log\left(\frac{2(l+1)^2}{\delta}\right).
	\end{align}
	We note that
	\begin{align}
		&\mr{a}_1\left(\zeta \alpha_{l}-\zeta \alpha_{l+1}\right)+\mr{a}_2\mr{a}\sqrt{2\log(l+1)}+2\mr{a}_3\log(l+1)\nonumber\\
		&\quad \quad=\mr{a}_1\zeta \alpha_{l_1} \left(\left( 1-\zeta\right)+\mr{a}\frac{\mr{a}_2}{\mr{a}_1}\frac{\sqrt{2\log(l+1)}}{\mr{a}\zeta^{l-1}}+\frac{\mr{a}_3}{\mr{a}_1}\frac{2\log(l+1)}{\mr{a} \zeta^{l-1}}\right)\nonumber\\
		&\quad \quad\stackrel{(a)}{\leq } \mr{a}_1\zeta \alpha_{l_1} \left(\left( 1-\zeta\right)+\frac{\sqrt{2\log(l+1)}}{ \zeta^{l-1}}+\frac{2\log(l+1)}{ \zeta^{l-1}}\right)\stackrel{(b)}{\leq} 0
	\end{align}
	where (a) follows from $\mr{a}_2/\mr{a}_1\leq 1$, $\mr{a_3}/\mr{a}_1\leq \mr{a}$, and (b) follows from the fact that $\zeta$ is sufficiently large.
\end{proof}

\begin{lemma}[Peeling Lemma III]
	\label{l:peeling3}
	Let $\mr{V}$ be a set and $h: \mr{V}\to \mbb{R}_{+}$ is a function. Let $g : \mbb{R}_+ \to \mbb{R}_+$ be right-continuous and non-decreasing function. Define $g^{-1}(x) = \inf\{\mr{a} \in \mbb{R}_+\,:\,g(\mr{a})\geq x\}$, which are generalized inverses of $g$.
	Let $\mr{a}_1,\mr{a}_2, \mr{a}_3$ be positive constants such that $\mr{a}_2/\mr{a}_1\leq 1$ and $\mr{c}\geq 1$ be a constant. Suppose for any $r >0$ and $\delta \in (0,1/\mr{c}]$, with probability at least $1-\mr{c}\delta$, we have
	\begin{align}
		\inf_{\vecv \in \mr{V},\,h(\vecv) \leq r} \mr{M}(\vecv)\geq -\mr{a}_1g(r)-\mr{a}_2\sqrt{\log(1/\delta)}+\mr{a}_3.
	\end{align}
	Define the event above as $\mr{E}(r,\delta)$. 

	Then, for any $\delta \in (0, 1/\mr{c}]$, with probability at least $1-\mr{c}\delta$, for any $\vecv \in \mr{V}$,  we have
	\begin{align}
		 \mr{M}(\vecv)&\geq  -C\mr{a}_1( g\circ h (\vecv)+1)-\mr{a}_2\sqrt{\log (2/\delta)}+\mr{a}_3
	\end{align}
\end{lemma}
\begin{proof}
	Let $\alpha_0=0$ and $\alpha_k:=  \zeta^{k-1}$, where $\zeta >1$ is a sufficiently large numerical constant determined later. Decompose $\mr{V}$ as
	\begin{align}
		\mr{V}_{k} :=\{\vecv \in \mr{V}\,\mid\, \alpha_{k}\leq (g\circ h)(\vecv)<\alpha_{k}\}.
	\end{align}
	Define $\beta_m := g^{-1}(\alpha_m)$, then we have $g(\beta_m)=\alpha_m$.
	We note that $\sum_{k\geq1}k^{-2}\leq {2}$ and we have
	\begin{align}
		\label{ine:l:peeling-03}
		\mbb{P}\left((\mr{E}':=) \cap_{k\geq 1}^\infty \mr{E}\left(\beta_k,\frac{\delta}{2k^2}\right) \right)\geq 1-\mr{c}\delta
	\end{align}
	for fixed $\delta \in (0,1/\mr{c}]$. We assume $\mr{E}''$ is realized.

	We note that for every $\vecv \in \mr{V}$, there are $l$ such that $\vecv \in \mr{V}_l$ because $g$ is non-decreasing.

	\noindent
	\underline{Case 1: $l=0$}

	In this case, 
	\begin{align}
		\mr{V}_{0} :=\{\vecv \in \mr{V}\,\mid\, 0 \leq (g \circ h)(\vecv)<1\},
	\end{align}
	and we have $h(\vecv)\leq \beta_1=g^{-1}(1)$, then from \eqref{ine:l:peeling-03}, we have
	\begin{align}
			\mr{M}(\vecv)&\geq -\mr{a}_1g(\beta_1)-\mr{a}_2\sqrt{\log\left(\frac{2(l+1)^2}{\delta}\right)}+\mr{a}_3\geq -\mr{a}_1-\mr{a}_2\sqrt{\log\left(2/\delta\right)}+\mr{a}_3.
	\end{align}

	\noindent
	\underline{Case 2: $l\geq 1$}

	In this case	
	\begin{align}
		\mr{V}_{l} :=\{\vecv \in \mr{V}\,\mid\, \alpha_{l}\leq (g \circ h)(\vecv)<\alpha_{l+1}\}.
	\end{align}
	Note that 
	\begin{align}
		g(\beta_{l+1}) = \zeta \alpha_{l} = \zeta \alpha_{l+1} + \zeta \alpha_{l}-\zeta \alpha_{l+1} \leq \zeta^2 g\circ h (\vecv) + \zeta \alpha_{l}-\zeta \alpha_{l+1},
	\end{align}
	$h(\vecv) \leq \beta_{l+1}=g^{-1}(\alpha_{l+1})$. Then, from \eqref{ine:l:peeling-03},  we have
	\begin{align}
			\mr{M}(\vecv)&\geq - \mr{a}_1 g(\beta_{l+1})- \mr{a}_2\sqrt{\log\left(\frac{2(l+1)^2}{\delta}\right)}+\mr{a}_3\nonumber\\
			&\geq -\mr{a}_1(\zeta^2 g\circ h (\vecv) + \zeta \alpha_{l}-\zeta \alpha_{l+1})-\mr{a}_2\sqrt{\log\left(\frac{2(l+1)^2}{\delta}\right)}+\mr{a}_3.
	\end{align}
	When $\zeta$ is sufficiently large, we have
	\begin{align}
		&-\mr{a}_1\left(\zeta \alpha_{l}-\zeta \alpha_{l+1}\right)-\mr{a}_2\sqrt{2\log(l+1)}=-\mr{a}_1\zeta \alpha_{l_1} \left(\left( 1-\zeta\right)+\frac{\mr{a}_2}{\mr{a}_1}\frac{\sqrt{2\log(l+1)}}{\zeta^{l-1}}\right)\geq 0,
	\end{align}
	and the proof is complete.
\end{proof}

\end{document}